\documentclass[a4paper,10pt,reqno]{amsart}
\usepackage{a4wide,color,array}
\usepackage{cite}
\usepackage{hyperref}
\usepackage{amsmath,amssymb,amsthm,color}
\hypersetup{linktocpage,colorlinks,linkcolor={red},citecolor={blue}}
\usepackage[english]{babel}
\usepackage[T1]{fontenc}
\usepackage[utf8]{inputenc}
\usepackage{dsfont}
\usepackage{enumitem}
\usepackage[a4paper,left=2.6cm,right=2.6cm,top=3cm,bottom=3.5cm]{geometry}%
\usepackage{marginnote}%

\usepackage{float}
\usepackage{subcaption} 
\usepackage{accents}

\usepackage{xspace}
\usepackage{bm}				
\usepackage{bbm,upgreek}		
\usepackage[e]{esvect}		
\usepackage{esint}			
\usepackage{etoolbox}			
\usepackage{calc}				
\usepackage{eso-pic}			
\usepackage{datetime}			

\usepackage{lmodern}
\numberwithin{equation}{section}

\usepackage{enumitem}
\setlist{nosep}
\setlist{noitemsep}

\newtheorem{theorem}{Theorem}%

\newtheorem{proposition}{Proposition}[section]
\newtheorem{lemma}[proposition]{Lemma}%
\newtheorem{conjecture}{Conjecture}%
\newtheorem{remark}{Remark}[section]%

\newtheorem{assumption}[remark]{Assumption}
\newtheorem{definition}[proposition]{Definition}%

\numberwithin{equation}{section}%

\newcommand{\dR}{\mathbb{R}}%

\newcommand{\E}{\mathbb{E}}

\newcommand{\ve}{\varepsilon}

\newcommand{\dE}{\mathbb{E}}%
\newcommand{\Var}{\mathrm{Var}}%
\newcommand{\Cov}{\mathrm{Cov} }%

\newcommand{\supp}{\mathrm{supp}}

\newcommand{\mc}{\mathcal}

\newcommand{\Id}{\mathrm{Id}}
\newcommand{\diag}{\mathrm{diag}}
\newcommand{\dd}{\mathrm{d}}

\newcommand{\TAP}{\mathrm{TAP}}
\newcommand{\GOE}{\mathrm{GOE}}
\newcommand{\tr}{\mathrm{Tr}}

\newcommand{\Span}
{\mathrm{Span}}

\newcommand{\eq}{\mathrm{eq}}

\def\indic{\mathbf{1}}

\newcommand{\Good}{\mathrm{Good}}
\newcommand{\SUSY}{\mathrm{SUSY}}
\newcommand{\Test}{\mathrm{Test}}
\newcommand{\Law}{\mathrm{Law}}
\newcommand{\Lawmatch}{\mathrm{LawMatch}}

\newcommand{\Prefix}{\mathrm{Prefix}}
\newcommand{\Skel}{\mathrm{Skel}}
\newcommand{\dist}{\mathrm{dist}}
\newcommand{\bfm}{\mathbf{m}}
\newcommand{\Pari}{\mathcal{P}\mathrm{arisi}}

\newcommand{\Lip}{\mathrm{Lip}}

\newcommand{\Ppar}{\mathcal{P}}
\newcommand{\Cpx}{\mathcal{C}}
\usepackage{graphicx}

\title[The Legendre structure of the TAP complexity for the Ising spin glass]{The Legendre structure of the TAP complexity\\ for the Ising spin glass}

\author{Jeanne Boursier}
\address{Department of Mathematics, Columbia University, New York, NY, USA}
\email{jb4893@columbia.edu}

\begin{document}

	\begin{abstract}
		We study the complexity of the Thouless--Anderson--Palmer (TAP) free energy
		for Ising spin glasses with a general mixed $p$-spin covariance, working with the generalized TAP functional of Chen, Panchenko, and Subag. We formulate three conjectures about the complexity (i.e. number of critical points). First, the annealed complexity is given by the Legendre transform of a variational functional constructed from the Parisi formula subject to a constraint on the overlap mass at zero, thereby establishing a precise link between the enumeration of TAP states and the large-deviation rate function of the partition function. Second, the quenched complexity is governed by the Legendre transform of a closely related functional in which the mass up to—but not including—the supremum of the support is constrained. Third, TAP states at any non-equilibrium free-energy level are organized into an ultrametric hierarchy, with ancestor states at other levels appearing only in subexponential number. 
		
		Using a Kac--Rice computation combined with a supersymmetric ansatz, we establish a lower bound on the annealed complexity that matches the prediction of the first conjecture. We further extend the analysis to a conditional setting in which a hierarchical ``skeleton'' of ancestors is prescribed, providing additional evidence in support of the second and third conjectures.
	\end{abstract}
	\maketitle
	
	\setcounter{tocdepth}{1}
	\tableofcontents

	\section{Introduction}
	
	\subsection{Setting}\label{sub:setting}
	
	Mean-field spin glasses are a paradigmatic source of rugged random landscapes: the interplay of disorder and high dimensionality produces an energy function with a complex hierarchy of valleys, saddles, and barriers whose geometry encodes the thermodynamic and algorithmic behavior of the model.
	
	Let $\xi:[-1,1]\to\dR$ be a covariance structure function satisfying Assumption~\ref{ass:xi} below. Given $\xi$, we consider the centered Gaussian process $\{H(\bfm)\}_{\bfm\in[-1,1]^N}$ on the full cube $[-1,1]^N$, defined by the covariance relation
	\begin{equation}\label{eq:Ham full}
		\dE[H(\bfm)H(\bfm')]=N\xi\!\left(\frac{1}{N}\langle \bfm,\bfm'\rangle\right), \quad \bfm,\bfm'\in [-1,1]^N.
	\end{equation}
	We refer to $H$ as the \emph{Hamiltonian} of the model. The classical Ising spin glass is defined by restricting $H$ to the Boolean cube $\{-1,1\}^N$: for spin configurations $\sigma,\tau\in\{-1,1\}^N$,
	\begin{equation}\label{eq:Ham}
		\Cov[H(\sigma),H(\tau)]=N\xi\!\left(\frac{1}{N}\langle \sigma,\tau\rangle\right).
	\end{equation}
	The extension of $H$ to the full cube $[-1,1]^N$ will be essential for the TAP analysis developed below, since magnetization vectors---the natural coarse-grained description of the Gibbs measure---take values in $[-1,1]^N$ rather than $\{-1,1\}^N$.
	
	The Parisi formula, proved in \cite{talagrand2006parisi,panchenko2014parisi}, identifies the asymptotic free energy of the Ising model as
	\begin{equation}\label{eq:Parisi formula}
		\lim_{N\to \infty}\dE\!\left[\frac{1}{N} \log \sum_{\sigma\in \{-1,1\}^N} e^{-H(\sigma)}\right]=\inf_{\zeta\in \mc{P}([0,1])}\Pari(\zeta),
	\end{equation}
	where the \emph{Parisi functional} $\Pari:\mc{P}([0,1])\to\dR$ is defined by
	\begin{equation}\label{def:Par}
		\Pari(\zeta)
		:=
		\Phi_\zeta(0,0)
		-\frac{1}{2}\int_0^1 t\,\xi''(t)\,\zeta([0,t])\,\dd t,
	\end{equation}
	and, for each $\zeta\in \mc{P}([0,1])$, $\Phi_\zeta$ solves the backward Parisi PDE
	\begin{equation}\label{eq:PPDE}
		\begin{cases}
			\partial_t \Phi_\zeta(t,x)
			=-\dfrac{\xi''(t)}{2}\Bigl[
			\partial_{xx}\Phi_\zeta(t,x)
			+\zeta\bigl([0,t]\bigr)\bigl(\partial_x\Phi_\zeta(t,x)\bigr)^2
			\Bigr],\\[1ex]
			\Phi_\zeta(1,x)=\log(2\cosh(x)).
		\end{cases}
	\end{equation}
	
	The Parisi formula encodes a rich geometric structure: the Gibbs measure organizes into an ultrametric tree whose branching levels are determined by the support of the minimizing measure $\zeta$, while $\zeta([0,t])$ plays the role of an effective temperature profile along the hierarchy (see Figure~\ref{fig:Parisi tree}). Recall that $\zeta$ also governs the asymptotic law of the overlap $\frac{1}{N}\langle \sigma^{(1)},\sigma^{(2)}\rangle$ between two independent replicas sampled from the Gibbs measure with the same disorder realization.
	
	\medskip
	
	Rather than tracking individual spin configurations, it is natural to pass to a coarser description in terms of the \emph{magnetization vector} $\bfm=(m_1,\ldots,m_N)\in [-1,1]^N$. Informally, a point $\bfm$ qualifies as a magnetization if it arises, with probability that does not decay exponentially in $N$, as the barycenter of a large collection of replicas (see Figure~\ref{fig:Parisi tree}).

	\begin{figure}[H]
		\centering
		\fbox{\includegraphics[width=0.4\textwidth]{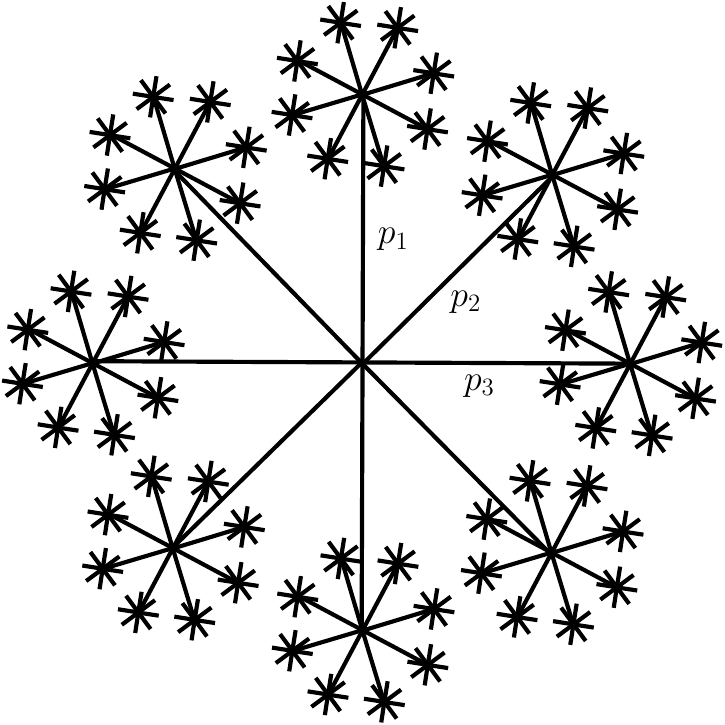}} 
		\caption{Schematic tree of magnetizations. The root (center) is $0\in\dR^N$. A magnetization is represented by a node of the tree (an internal branching point or a leaf). A pure state corresponds to a leaf.}\label{fig:Parisi tree}\end{figure}

	Individual configurations sit in pure states within bands orthogonal to the corresponding magnetizations. In the two-level hierarchy depicted in Figure~\ref{fig:Parisi tree}, the overlap between any two magnetizations---and hence between the corresponding replicas---takes only the values $0$, $q_1$, and $q_2$, where $q_1$ and $q_2$ are the squared distances (scaled) from the origin to the first- and second-level branching points, respectively. In this example, the Parisi optimizer $\zeta$ in~\eqref{eq:Parisi formula} has support $\{0,q_1,q_2\}$. The weights $p_i$, giving the probability of landing in subtree $i$, are distributed as the cluster frequencies of a Chinese restaurant process, i.e.\ according to a Poisson--Dirichlet law $\mathrm{PD}(\alpha,0)$ with $\alpha=\zeta(\{0\})$, and the same construction repeats recursively within each subtree. In general, the full hierarchy is described by a fragmentation of Chinese restaurant processes, in which the magnetizations play the role of tables and the replicas play the role of customers.

	\medskip
	
	A central question now presents itself: \emph{how can the magnetizations be characterized?} It turns out that they are precisely the approximate local maximizers of a random function on~$[-1,1]^N$ known as the TAP free energy. But what, exactly, is this functional, and where does it come from? A rigorous definition was given by Chen, Panchenko, and Subag~\cite{chen2018generalized}, who introduced $F_{\TAP}(\bfm)$ as the free-energy cost of constraining a large number of replicas to have barycenter~$\bfm$. With this interpretation, it becomes essentially transparent that magnetizations must be critical points of~$F_{\TAP}$ and that their TAP free-energy values concentrate at the equilibrium level. We refer to the critical points of~$F_{\TAP}$ as \emph{TAP states}; the magnetizations are then the distinguished TAP states whose free energy equals the equilibrium free energy. Carrying out a Parisi-type variational principle in the band orthogonal to~$\bfm$, Chen, Panchenko, and Subag arrive at the following explicit definition.

	\begin{definition}[TAP free energy]\label{def:zetam}
		The TAP free energy is the random function $F_{\TAP}: [-1,1]^N \to \mathbb{R}$ defined by
		\begin{equation}\label{def:FTAP}
			F_{\TAP}(\bfm) =  \inf_{\zeta\in \mc{P}([q,1]) } F_{\TAP, \zeta}(\bfm)
		\end{equation}
		where 
		\begin{equation}
			F_{\TAP,\zeta}(\bfm) = H(\bfm)-\sum_{i=1}^N \Phi_\zeta^*(q_\bfm,\cdot)(m_i)-\frac{1}{2}N\int_{q_\bfm}^1 t\xi''(t)\zeta([0,t])\dd t \, ,
		\end{equation}
		where $q_\bfm:=\frac{1}{N}\Vert \bfm\Vert^2$, $\Phi_\zeta^*(q,\cdot)$ denotes the Legendre transform of $\Phi_\zeta(q,\cdot)$, as defined in \eqref{eq:PPDE}, and $\zeta([0,t])$ stands with an abuse of notation for $\zeta([q,t])$. We denote by $\zeta_\bfm\in \mc{P}([q_\bfm,1])$ the minimizer (unique by \cite{chen2018generalized}) achieving the infimum above.
	\end{definition}

	Intuitively, $\zeta_\bfm$ corresponds to the law of the overlap between two replicas sampled in the band orthogonal to~$\bfm$, which explains why it is supported on $[q,1]$. The TAP free energy originally proposed by physicists~\cite{ThoulessAndersonPalmer1977} corresponds to the replica-symmetric case, where the minimizer $\zeta_\bfm$ is a Dirac mass at $q=\frac{1}{N}\Vert \bfm\Vert^2$. The optimality of $\delta_q$ imposes additional conditions on $\bfm$ that were not exploited in the original physics literature, rendering the system studied there underdetermined.
	
	\medskip
	
	It is convenient to express the TAP free energy in macroscopic form, parametrized by the empirical measure of the magnetization coordinates rather than the magnetization vector itself. For a probability measure $\mu\in\mc{P}([-1,1])$ and $\zeta\in\mc{P}([q,1])$, define
	\begin{equation}\label{def:TAPmuzeta}
		\TAP(\mu,\zeta):=-\int_{-1}^1 \Phi_\zeta^*(q,m)\,\dd \mu(m)-\frac{1}{2}\int_q^1 t\,\xi''(t)\,\zeta([0,t])\,\dd t,\qquad q:=\int m^2\, \dd \mu(m),
	\end{equation}
	and
	\begin{equation}\label{def:TAPmu}
		\TAP(\mu):=\inf_{\zeta\in \mc{P}([q,1])}\, \TAP(\mu,\zeta).
	\end{equation}
	When $\mu=\mu_N^{(\bfm)}:=\frac{1}{N}\sum_{i=1}^N\delta_{m_i}$ is the empirical measure of a magnetization vector $\bfm$, one recovers the deterministic part of~\eqref{def:FTAP}. In particular, the Parisi functional is the TAP free energy of the trivial state $\bfm=\mathbf{0}$:
	\begin{equation}\label{eq:Pari=TAP0}
		\Pari(\zeta)=\TAP(\delta_0,\zeta).
	\end{equation}
	Indeed, since $x \mapsto \Phi_{\zeta}(0,x)$ is even and convex (hence minimized at $0$), we have
	\[
	\Phi_{\zeta}^*(0,0) = -\Phi_{\zeta}(0,0),
	\]
	yielding \eqref{eq:Pari=TAP0}.
	
		\subsection{Main results}\label{sub:main result}
	
	We now state our principal results. The first concerns the \emph{annealed} complexity of TAP critical points — the exponential growth rate of the expected number of such points at a prescribed free-energy level. The second extends the computation to a \emph{conditional annealed} setting by introducing a hierarchical skeleton of ancestor states, which constitutes a key step toward the quenched complexity.

	Clearly, the problem makes sense only if there does indeed exist a Gaussian process such that \eqref{eq:Ham} holds. Thus, we work under the following assumption:
	\begin{assumption}[Covariance function]\label{ass:xi}
		The structure function $\xi:[-1,1]\to\dR$ is $C^\infty$, strictly convex on $[-1,1]$, and satisfies $\xi(0)=\xi'(0)=0$.
		Moreover, for every $N\ge1$, the kernel
		\[
		K_N(\bfm,\bfm') \;:=\; N\,\xi\!\left(\frac{\langle \bfm,\bfm'\rangle}{N}\right),
		\qquad \bfm,\bfm'\in[-1,1]^N,
		\]
		is positive semidefinite, so that a centered Gaussian process
		$\{H(\bfm)\}_{\bfm\in[-1,1]^N}$ with covariance
		\[
		\dE[H(\bfm)H(\bfm')]=K_N(\bfm,\bfm')
		\]
		exists.
	\end{assumption}
	\begin{remark}
		By a classical theorem of Schur and Schoenberg \cite{schoenberg1942}, a sufficient condition for Assumption~\ref{ass:xi} to be met is that $\xi$ admits a power-series expansion
		\[
		\xi(t)=\sum_{p\ge 1} a_p\, t^p
		\]
		with non-negative coefficients $a_p\ge 0$. To further ensure strict convexity, we keep only even coefficients, which yields the mixed $p$-spin form
		\[
		\xi(t)=\sum_{\substack{p\ge 2 \\ p\ \text{even}}}\beta_p^2\, t^p.
		\]
	\end{remark}

	\subsubsection{Annealed complexity}

	For $f\in\dR$ and $\ve>0$, let $\mc{N}_\ve(f)$
	denote the number of critical points $\bfm\in[-1,1]^N$ of $F_{\TAP}$
	satisfying $F_{\TAP}(\bfm)\in N(f-\ve,f+\ve)$.

	Our first main result provides a rigorous lower bound on the annealed complexity. Before stating it, we introduce a useful definition.

	\begin{definition}[Probability measures with an $(n+1)$-atom prefix]%
		\label{def:nprefix}
		Fix $n\geq 1$ and set $q_0=0$. Given a strictly increasing sequence
		$(q_i)_{1\leq i\leq n}$ in $(0,1)$ and a strictly increasing sequence of
		weights $(u_i)_{1\leq i\leq n}$ in $(0,1)$, we write
		$\Prefix_{n+1}\!\bigl((u_i)_{1\leq i\leq n},(q_i)_{1\leq i\leq n}\bigr)$
		for the collection of probability measures $\zeta$ on $[0,1]$ satisfying
		\begin{equation*}
			\supp(\zeta)\cap[0,q_n]=\{q_0,q_1,\ldots,q_n\}
			\qquad\text{and}\qquad
			\zeta([0,q_{i-1}])=u_i \;\text{ for all }\; i\in\{1,\ldots,n\}.
		\end{equation*}
		We say that such $\zeta$ has an \emph{$(n+1)$-atom prefix} $\{0,q_1,\ldots,q_n\}$.
	\end{definition}

	\begin{theorem}[Lower bound on the annealed complexity]\label{theorem:annealed}
		Let $f\in\dR$, $u\in(0,1)$, $q\in(0,1)$, and
		$\zeta\in\Prefix_2(u,q)$.
		Assume:
		\begin{enumerate}
			\item[\emph{(i)}]
			\emph{(Tail optimality.)}
			$\zeta$ minimizes $\Pari$ over $\Prefix_2(u,q)$.
			\item[\emph{(ii)}]
			\emph{(Stationarity.)}
			$(u,q)$ is a critical point of
			$(u',q')\mapsto u'\!\left(\inf_{\zeta'\in\Prefix_2(u',q')}\Pari(\zeta')-f\right)$.
		\end{enumerate}
		Let $\ve\in(0,1)$ and let $\mc{N}_\ve(f,\zeta)$ be the number of
		critical points of $F_{\TAP,\zeta}$ with
		$\frac{1}{N}F_{\TAP,\zeta}(\bfm)\in(f-\ve,f+\ve)$.
		Then
		\begin{equation}\label{eq:th annealed}
			\lim_{\ve\to 0}\lim_{N\to\infty}
			\frac{1}{N}\log\dE\bigl[\mc{N}_\ve(f,\zeta)\bigr]
			\;\geq\;
			\zeta(\{0\})\,\bigl(\Pari(\zeta)-f\bigr).
		\end{equation}
	\end{theorem}

	Theorem~\ref{theorem:annealed} is a first step toward the following conjecture, which asserts that the annealed complexity is governed by the Legendre transform of an explicit function.

	\begin{conjecture}[Annealed complexity formula]\label{conj:annealed}
		Define
		\begin{equation}\label{def:Lambda}
			\Lambda(\theta)=\theta\inf_{\zeta:\zeta(\{0\})=\theta}\Pari(\zeta).
		\end{equation}
		There exists an open subset $D\subset\dR$ containing $[\inf_\zeta \Pari(\zeta),+\infty)$ such that for every $f\in D$,
		\begin{equation*}
			\lim_{\ve\to 0}\lim_{N\to\infty}
			\frac{1}{N}\log\dE\bigl[\mc{N}_\ve(f)\bigr]=-\Lambda^*(f)=\inf_\theta \bigl(\Lambda(\theta)-\theta f\bigr).
		\end{equation*}
	\end{conjecture}

	The right-hand side of~\eqref{eq:th annealed} is precisely the value that the conjectured Legendre transform assigns to an optimizer $\zeta$ with a two-atom prefix. A natural direction for future work is to establish the full conjecture through a purely probabilistic argument — one that bypasses the Kac--Rice machinery entirely — by connecting the counting problem to the large deviation principle for the partition function and thereby making the Legendre duality between complexity and free energy fully explicit.

	The difficulty of upgrading Theorem~\ref{theorem:annealed} to the full conjecture via direct computation reflects a broader phenomenon in spin glass theory: variational formulas, even when available, are notoriously resistant to explicit analysis. An instructive example is the Sherrington--Kirkpatrick model with $\xi(x)=\beta^2 x^2$, for which the left-hand side of~\eqref{eq:Parisi formula} is manifestly convex in $\beta$, yet this convexity remains entirely opaque from the variational expression on the right-hand side.

	\begin{remark}
		Recall from~\eqref{eq:Parisi formula} that the equilibrium free energy is $f_\eq=\inf_\zeta \Pari(\zeta)$. Conjecture~\ref{conj:annealed} implies, in particular, that the complexity is strictly negative for $f>f_\eq$ and vanishes at $f=f_\eq$, in agreement with the physical expectation that the equilibrium level marks the threshold above which an exponential number of metastable states ceases to exist.
	\end{remark}

	\begin{remark}[Extreme-value heuristic]
		Let $f_\eq=\inf_\zeta \Pari(\zeta)$ be the equilibrium free energy. If $\zeta$ denotes the optimizer in the Parisi formula~\eqref{eq:Parisi formula}, then
		\[
		(-\Lambda^*)'(f_\eq)=-\zeta(\{0\}).
		\]
		Assuming Conjecture~\ref{conj:annealed}, this means that $-\zeta(\{0\})$ coincides with the slope of the asymptotic annealed complexity at $f_\eq$, which is exactly what one expects from the following extreme-value heuristics.

		Denote by $\Sigma$ the asymptotic annealed complexity. Recalling that
		\[
		\#\{\alpha:\,F_\TAP(\alpha)\approx Nu\}\approx e^{N\Sigma(u)},
		\]
		a first-order expansion around $f_\eq$ (recalling $\Sigma(f_\eq)=0$) gives for $y\geq 0$,
		\[
		\#\left\{\alpha:\,F_\TAP(\alpha)\ge Nf_\eq-y\right\}\approx e^{-\Sigma'(f_\eq)y}.
		\]
		Hence the limiting process in the shifted variable $Nf_\eq-F_\TAP(\alpha)$ should be a Poisson point process (PPP) on $\dR$ with cumulative intensity
		\[
		\mu((-\infty,y))=e^{-\Sigma'(f_\eq)y},\qquad \mu(\dd y)=-\Sigma'(f_\eq)e^{-\Sigma'(f_\eq)y}\,\dd y.
		\]
		With the change of variables $x=e^{-y}$, the corresponding extremal weights $e^{F_\TAP(\alpha)-Nf_\eq}$ form a PPP on $(0,\infty)$ with intensity
		\[
		\lambda(\dd x)=-\Sigma'(f_\eq)x^{\Sigma'(f_\eq)-1}\,\dd x.
		\]
		By \cite{constructionTala}, the same PPP has parameter $\zeta(\{0\})$, so $\zeta(\{0\})=-\Sigma'(f_\eq)$.
	\end{remark}

	\subsubsection{Toward the quenched complexity}

	We now turn to the quenched complexity, that is, the typical logarithmic count of TAP critical points. We first state a conditional result that computes the complexity with a hierarchical skeleton of ancestor states fixed, and we then formulate two conjectures describing the structure of the quenched complexity.

	The idea is to fix a hierarchical skeleton of ancestor states — each required to be a critical point of $F_{\TAP,\zeta}$ at the equilibrium free-energy level — and to count descendant critical points in the orthogonal band of the deepest ancestor. For convenience, we restrict to the mixed case, but our result also holds in the pure case.

	\begin{theorem}[Lower bound on the conditional annealed complexity]\label{theorem:annealed cond}
		Suppose that $\xi$ has at least two distinct nonzero coefficients $\beta_p^2$ (mixed case). Let $n\ge2$. Let $\bm{u}=(u_1,\ldots,u_{n})$ and $\bm{q}=(q_1,\ldots,q_n)$
		be strictly increasing sequences in $(0,1)$, and let
		$\zeta\in\Prefix_{n+1}(\bm{u};\bm{q})$.
		Using the notation of~\eqref{def:Par}, define the \emph{prefix free en\-ergy}
		\begin{equation*}
			\Ppar^{(n)}(\bm{u};\bm{q})
			\;:=\;
			\inf_{\zeta'\,\in\,\Prefix_{n+1}(\bm{u};\bm{q})}
			\Pari(\zeta'),
		\end{equation*}
		and the \emph{$n$-step complexity functional at level $f$}
		\begin{equation*}
			\Cpx^{(n)}_f(\bm{u};\bm{q})
			\;:=\;
			u_n\,\bigl(\Ppar^{(n)}(\bm{u};\bm{q})-f\bigr).
		\end{equation*}
		
		We assume:
		\begin{enumerate}
			\item[\emph{(i)}]
			\emph{(Tail optimality)} The infimum $\Ppar^{(n)}(\bm{u};\bm{q})$ is
			attained at $\zeta$.
			\item[\emph{(ii)}]
			\emph{(Stationarity)} $(\bm{u};\bm{q})$ is a critical point of
			$\Cpx^{(n)}_f$.
		\end{enumerate}
		Fix $\ve\in(0,1)$ and fix an ancestor skeleton
		$(\bfm^{(1)},\ldots,\bfm^{(n-1)})$
		satisfying the admissibility conditions stated in Section~\ref{section:multiple}
		(Definitions~\ref{def:susyreplicas} and~\ref{def:matchlaw}).
		
		Let $\mc{N}_\ve(f,\zeta,q_n,\bfm^{(n-1)})$ be the number of critical points
		$\bfm$ of $F_{\TAP,\zeta}$ lying in the orthogonal band of $\bfm^{(n-1)}$
		(and satisfying $\tfrac1N F_{\TAP,\zeta}(\bfm)\in(f-\ve,f+\ve)$). Set
		\begin{equation*}
			\Skel_{n-1}=\left(\nabla F_{\TAP,\zeta}(\bfm^{(1)}),F_{\TAP,\zeta}(\bfm^{(1)}),\ldots,\nabla F_{\TAP,\zeta}(\bfm^{(n-1)}), F_{\TAP,\zeta}(\bfm^{(n-1)})\right).
		\end{equation*}
		
		Then
		\begin{multline}\label{eq:hint}
			\lim_{\ve\to0}\lim_{N\to\infty}
			\frac{1}{N}\log\dE\Bigl[
			\mc{N}_\ve(f,\zeta,q_n,\bfm^{(n-1)})
			\;\Big|\;
			\Skel_{n-1}=(0,N\Pari(\zeta),\ldots,0,N\Pari(\zeta))
			\Bigr]
			\geq
			\Cpx_f^{(n)}(\bm{u};\bm{q})
			\\=
			\zeta([0,q_n))\,(\Pari(\zeta)-f).
		\end{multline}
	\end{theorem}
	
	Note that the conditioning on the skeleton $\Skel_{n-1}$ also appears in \cite{dembo2024disordered} in the spherical setting.

Theorem~\ref{theorem:annealed cond} provides evidence for the following two conjectures, which together describe the structure of the quenched complexity.

	\begin{conjecture}[Quenched complexity formula]\label{conj:quenched formula}
		Define
		\begin{equation}\label{def:tildeLambda}
			\tilde{\Lambda}(\theta)=\theta\inf_{\zeta:\zeta([0,\sup(\supp \zeta)))=\theta}\Pari(\zeta).
		\end{equation}
		Then there exists an open subset $D\subset\dR$ such that for every $f\in D$,
		\begin{equation}\label{eq:quenched conj}
			\lim_{\ve\to 0}\lim_{N\to\infty}
			\frac{1}{N}\dE\bigl[\log\mc{N}_\ve(f)\bigr]=-\tilde{\Lambda}^*(f)=\inf_\theta \bigl(\tilde{\Lambda}(\theta)-\theta f\bigr).
		\end{equation}
	\end{conjecture}

	Note that it has been proved (at least in certain regimes, e.g. for the SK model slightly above criticality) that the Parisi optimizer has a jump at the top of its support; equivalently,
	\[
	\zeta\bigl([0,\sup(\supp \zeta))\bigr)\neq 1.
	\]
	See \cite{auffinger2015properties}.

	\begin{conjecture}[Ultrametric organization of TAP states]\label{conj:quenched full}
		For $f\in D$, let $\theta$ be the optimizer in~\eqref{eq:quenched conj} and $\zeta$ the minimizer
		in~\eqref{def:tildeLambda} for this~$\theta$.
		Then:
		\begin{enumerate}[label=\emph{(\roman*)}]
			\item \emph{(Ultrametric structure.)} The TAP critical points at free-energy level $f$ are organized in an ultrametric tree: the overlap between any two TAP states at level $f$ taken uniformly takes values in the support of $\zeta$.
			\item \emph{(Separation of free-energy levels.)} Whenever $f\neq f_\eq:=\inf_\zeta \Pari(\zeta)$, the ancestor states in this ultrametric tree have TAP free energy $\Pari(\zeta)\neq f$, i.e., the ancestors live at a different free-energy level than their descendants.
			\item \emph{(Subexponential number of ancestors.)} The quenched complexity of the ancestor states is $0$: they exist, but their number grows at most subexponentially in $N$.
		\end{enumerate}
	\end{conjecture}

	\begin{remark}
		Conjecture~\ref{conj:quenched formula} is implicit in the physics literature. In~\cite{annibale2003supersymmetric}, the authors observe that the supersymmetric quenched complexity of the TAP states coincides with the Legendre transform of the static free energy with respect to the largest breaking point of the overlap matrix, and that this identification yields the exact quenched complexity in the SK model as the Legendre transform of the full-RSB static free energy. We refer the reader also to~\cite{kent2023count} for related developments.
	\end{remark}

	Let us explain why Theorem~\ref{theorem:annealed cond} constitutes progress toward Conjectures~\ref{conj:quenched formula} and~\ref{conj:quenched full}. The connection rests on two observations. First, we expect the right-hand side of~\eqref{eq:hint} to coincide with the true quenched complexity of the restricted set of TAP states descending from the skeleton. Second, good ancestors satisfying the conditioning in Theorem~\ref{theorem:annealed cond} should exist with high probability: evaluating~\eqref{eq:hint} at $f=\Pari(\zeta)$ yields a vanishing complexity, which is consistent with the existence of a subexponential (but nonzero) number of such ancestor states. Moreover, this suggests that the quenched complexity is nothing but some annealed conditional complexity.

	\subsection{Literature on the TAP free energy}
	\subsubsection{Physics background}

	A central challenge in the theory of mean-field spin glasses has been to identify the functional governing the law of magnetizations. An early approach was proposed by Thouless, Anderson, and Palmer~\cite{ThoulessAndersonPalmer1977}, who derived a set of self-consistent equations for the local magnetizations of the Sherrington--Kirkpatrick model via diagrammatic expansions.

	The counting of metastable states in mean-field spin glasses was initiated by Bray and Moore~\cite{AJBray,bray1981metastable}, who counted solutions of the original TAP equations~\cite{ThoulessAndersonPalmer1977} for the Sherrington--Kirkpatrick model, later extended to $p$-spin models by Rieger~\cite{rieger1992number}. Cavagna, Giardina, and Parisi~\cite{cavagna1998stationary,cavagna1998quenched} then placed the computation on firmer ground by introducing a systematic Kac--Rice approach for stationary points of any index in the spherical $p$-spin model, while Cavagna, Giardina, Parisi, and M\'ezard~\cite{cavagnaformal} established a formal equivalence between the TAP and thermodynamic methods in the SK model.
	
	The role of the Becchi--Rouet--Stora--Tyutin (BRST) supersymmetry was identified by the Rome group in a series of works~\cite{annibale2003supersymmetric,annibalerole,crisanti2003complexity}. A key finding of Crisanti, Leuzzi, Parisi, and Rizzo is that the BRST-invariant (SUSY) complexity equals the Legendre transform of the disorder-averaged free energy. However, their quenched computation~\cite{crisanti2003quenched,crisanti2004spin} showed that this SUSY solution is unstable at every free energy level except the equilibrium one, so that it does not predict an extensive number of metastable states away from equilibrium. For the Ising $p$-spin model in the 1RSB regime, both a SUSY and a non-SUSY solution coexist and cross at the free energy where the replicon eigenvalue vanishes~\cite{crisanti2005complexity}; related phenomena were observed for spherical models~\cite{annibale2004coexistence}. The mechanism of SUSY breaking was further investigated by Parisi and Rizzo~\cite{parisi2004supersymmetry}, and numerical studies~\cite{cavagna2004numerical} confirmed that TAP stationary points in the SK model organize into minimum--saddle pairs compatible with the SUSY-breaking scenario.
	
	The overall picture is rich but not fully coherent: different approaches lead to partially contradictory conclusions regarding which states should be counted, the role of supersymmetry, and whether the annealed approximation is exact. We stress that all of these computations use the original TAP equations~\cite{ThoulessAndersonPalmer1977}, where the overlap measure $\zeta_\bfm$ is replaced by the Dirac mass $\delta_q$, leaving the system underdetermined. The present paper works with the generalized TAP free energy~\eqref{def:FTAP}, where $\zeta_\bfm$ is fixed by the variational principle, which resolves several of these ambiguities. For the spherical case, Kent-Dobias and Kurchan~\cite{kent2023count} have recently derived a general solution incorporating Parisi's ground-state formula and the full RSB hierarchy.
	
	Finally, we mention the complementary \emph{real replica} method of Monasson~\cite{monasson1995structural}, in which $m$ coupled clones are constrained to the same state and $m\to 1$. The connection between this approach, the tilted Parisi formula of \cite{LDTalagrand}, and the R\'enyi entropy was recently clarified by Javerzat, Bertin, and Ozawa~\cite{javerzat2025renyi}, who showed that the R\'enyi complexity with index $m$ coincides with the $m$-component annealed Franz--Parisi potential~\cite{franz1995recipes}.
	\subsubsection{Rigorous approaches}
	
	The generalized TAP free energy was introduced and rigorously identified by Chen, Panchenko, and Subag~\cite{chen2018generalized}. The key idea is to compute the free energy cost of constraining the barycenter of a large number of replicas to equal a prescribed magnetization $\bfm\in [-1,1]^N$: the resulting partition function reduces to a Parisi-type variational formula evaluated on the band of configurations orthogonal to $\bfm$. This yields a TAP representation of the free energy in which the classical Onsager correction is replaced by a more general term, determined by the Parisi measure restricted to $[\frac{1}{N}\|\bfm\|^2,1]$. The follow-up work~\cite{chen2021generalized} simplified the energy representation of the generalized TAP states — showing in particular that all states at the same distance from the origin share the same energy — and extended the results to zero temperature. In the spherical setting, Subag~\cite{subag2018free} developed a parallel TAP theory by defining free energy landscapes via thin bands around each point inside the sphere, and established a TAP representation for the free energy at any multi-samplable overlap.
	
	For spherical spin glasses, the complexity of critical points was first computed (at the annealed level) by Auffinger, Ben Arous, and \v{C}ern\'{y}~\cite{auffinger2013random}. In the pure $p$-spin spherical model ($p\geq 3$), the quenched complexity coincides with the annealed one: this was established by Subag~\cite{subag2017complexity} via a second moment argument and later strengthened to concentration at arbitrary energy levels by Subag and Zeitouni~\cite{subag2021concentration}. The success of the second moment method in the spherical case relies crucially on the rotational invariance of the model, which forces two independently sampled critical points to be nearly orthogonal. For Ising spin glasses, such orthogonality fails: even in regimes where annealed and quenched complexities are expected to agree, the second moment of the number of critical points is not controlled by the square of its mean, making the passage from annealed to quenched considerably more delicate.
	
	A closely related line of work connects the TAP landscape to planted (or spiked) spin glass models. In such models, the disorder is tilted so that the Gibbs measure concentrates near a planted signal $x\in \{-1,1\}^N$; the resulting model is the Bayesian posterior for a high-dimensional inference problem. Interestingly, this model is always replica symmetric \cite{deshpande2017asymptotic,barbier2019adaptive}. In \cite{fan2021tap}, Fan, Mei and Montanari show that the TAP complexity at the equilibrium free energy for the planted model is \(0\), and it becomes strictly negative above and below that value when the spike is large enough. They also show that the planted configuration and independent replicas from the Gibbs measure fall into the same TAP state. In fact, the TAP magnetization \(\bfm\) matches the posterior mean (the Bayes estimator), and the mean squared error is \(1-q\), where \(q=\frac{1}{N}\|\bfm\|^2\).
	
	\subsection{Perspectives}

	The function $\Lambda$ defined in~\eqref{def:Lambda} is directly linked to
	the large deviations of the free energy. Indeed, as shown
	in~\cite{LDTalagrand} (see also \cite{chen2026onesided}), for every $\theta\in(0,1)$,
	\begin{equation*}
		\lim_{N\to\infty}\frac{1}{N}\log\dE[Z_N^\theta]=\Lambda(\theta),
	\end{equation*}
	where 
	\begin{equation*}
		Z_N=\sum_{\sigma\in \{-1,1\}^N} e^{-H(\sigma)}.
	\end{equation*}
	With a little more work, this implies that the free energy $\frac{1}{N}\log Z_N$ satisfies a large
	deviation principle on $[\inf_\zeta\Pari(\zeta),\infty)$ with rate
	function~$\Lambda^*$.
	
	This provides a conceptual explanation for
	Conjecture~\ref{conj:annealed} in the regime
	$f>\inf_\zeta\Pari(\zeta)$: the probability that the free energy reaches
	an atypical value~$f$ is governed by the existence of at least one TAP
	state at level~$f$, and because the number of such states concentrates
	around its expectation, the large deviation rate is controlled by the
	annealed complexity. In other words, $\Lambda^*$ plays a dual role — it
	is simultaneously the rate function for the free energy and the
	(negative) annealed complexity of TAP states. This duality is the
	rigorous counterpart of Monasson's real replica
	method~\cite{monasson1995structural}. We will develop this connection in
	a forthcoming paper.

	The quenched complexity is not accessible through direct computation, as the second moment method is not expected to yield matching bounds in the Ising setting. A promising approach is to relate the quenched complexity to constrained thermodynamics: one would compute a Franz--Parisi potential~\cite{franz1995recipes} using classical tools and then establish that the barycenters of the constrained Gibbs measure are TAP states.

	\subsection{Outline of the proof and plan of the paper}
	
	\paragraph{\bf{Section \ref{section:base}: annealed complexity and the SUSY ansatz.}}
	We begin with an \emph{annealed} computation of the TAP complexity. We consider an optimizer $\zeta$ of the variational problem of Theorem \ref{theorem:annealed}. We count the critical points $\bfm\in [-1,1]^N$ satisfying
	\begin{equation}\label{eq:condzetam}
		\zeta_\bfm\approx \zeta|_{(q,1]}+\zeta([0,q]) \delta_q\quad \text{and} \quad \frac{1}{N}\Vert \bfm\Vert^2\approx q,
	\end{equation}
	where $\zeta_\bfm$ is the optimal overlap measure in the band orthogonal to $\bfm$ (see Definition \ref{def:zetam}). We emphasize that \eqref{eq:condzetam} is equivalent to $\zeta_\bfm([0,t])=\zeta([0,t])$ for every $t\in [q,1]$.
	
	The annealed complexity of these TAP states at free energy level $f$ can be expressed as an integral over $\bfm$ via the Kac--Rice formula. A direct computation, however, yields a complicated variational formula over probability measures on $[-1,1]$ — representing the empirical distribution of the coordinates of the TAP states — involving many Lagrange multipliers. As observed in the physics literature \cite{crisanti2003complexity,parisi2004supersymmetry,cavagnaformal,annibalerole}, there exists a distinguished local maximum of this macroscopic action at which the Lagrange multipliers cancel one another exactly. These cancellations are driven by a supersymmetry.
	
	More precisely, the annealed complexity can be written as a partition function over bosonic variables (the magnetization and the Lagrange multipliers associated with imposing $\nabla F_{\TAP,\zeta}(\bfm)=0$ and $F_{\TAP,\zeta}(\bfm)=Nf$) and fermionic fields (arising from the Berezin integral representation of the determinant). This action carries a BRST supersymmetry that yields Ward identities. One of these identities gives the fundamental relation
	\begin{equation}\label{eq:Wardcorr}
		\langle \bar{\psi},\psi\rangle=\langle x,\bfm\rangle,
	\end{equation}
	where the fermionic correlator $\langle \bar{\psi},\psi\rangle$ is, up to a multiplicative constant, the subordination function of the Hessian evaluated at $0$, and the bosonic correlator $\langle x,\bfm\rangle$ represents the longitudinal linear response in the direction of $\bfm$.
	
	A crucial distinction from the physics literature must be emphasized here. The SUSY ansatz of \cite{crisanti2003complexity,parisi2004supersymmetry,cavagnaformal,annibalerole} yields the identity \eqref{eq:Wardcorr} but not the value of these correlators, which was determined by an ad hoc choice. This ambiguity arose because the physicists were not working with the correct TAP free energy, leaving the system underdetermined. In our setting, the condition \eqref{eq:condzetam} determines the correlators uniquely, yielding
	\begin{equation*}
		\langle \bar{\psi},\psi\rangle=N\int_q^1 \zeta([0,t])\dd t.
	\end{equation*}
	
	The Ward identity \eqref{eq:Wardcorr} is then used as an ansatz to evaluate the annealed complexity. By the Kac--Rice formula, we need to compute
	\begin{multline*}
		\int_{[-1,1]^N} \dE\bigl[|\det(\nabla^2 F_{\TAP,\zeta}(\bfm))|\bigm| \nabla F_{\TAP,\zeta}(\bfm)=0,\,F_{\TAP,\zeta}(\bfm)=Nf\bigr]\\ \times\varphi_{(\nabla F_{\TAP,\zeta}(\bfm),F_{\TAP,\zeta}(\bfm))}(0,Nf)\,\indic_{\mc{A}}(\bfm)\,\dd \bfm,
	\end{multline*}
	where $\mc{A}$ is the set of $\bfm\in [-1,1]^N$ satisfying \eqref{eq:condzetam} and the SUSY condition $\langle \partial_m h_\zeta(q,\bfm),\bfm\rangle=N\xi'(q)\int_q^1 \zeta([0,t])\dd t$. The density of the Gaussian vector $(\nabla F_{\TAP,\zeta}(\bfm),F_{\TAP,\zeta}(\bfm))$ at $(0,Nf)$ admits the convex-dual representation
	\begin{equation}\label{eq:varphidual}
		\begin{split}
			\log \varphi_{(\nabla F_{\TAP,\zeta}(\bfm),F_{\TAP,\zeta}(\bfm))}(0,Nf)&=-\tfrac12\langle z,\Gamma^{-1}z\rangle\;-\;\tfrac12\log\det(2\pi\Gamma)\\
			&=\frac{1}{2}\inf_{w\in\mathbb{R}^{N+1}}\bigl\{\langle w,\Gamma w\rangle-2\langle w,z\rangle\bigr\}-\tfrac12\log\det(2\pi\Gamma),
		\end{split}
	\end{equation}
	where $\Gamma:=\Cov((\nabla H(\bfm),H(\bfm)))$ and $z\in \dR^{N+1}$ is an explicit vector. Writing $w=(x,v)$ with $x\in \dR^N$ and $v\in \dR$, the field $x$ plays the role of the Lagrange field associated with fixing the gradient to $0$, while $v$ is the Lagrange parameter associated with fixing the free energy to $Nf$. We bound \eqref{eq:varphidual} from above by substituting the SUSY ansatz values $v=\zeta(\{0\})$ and $\langle x,\bfm\rangle=N\int_q^1 \zeta([0,t])\dd t$, obtaining an upper bound on the annealed complexity.
	
	This upper bound corresponds to typical TAP states whose empirical distribution is given by the law of $\partial_x \Phi_\zeta(q,X_q)$, where $(X_t)$ is the Auffinger--Chen SDE. To obtain the matching lower bound, we verify that this saddle point satisfies the required constraints, using (i) properties of the Auffinger--Chen SDE and (ii) the optimality of $\zeta$ in the variational problem.
	
	\paragraph{\bf{Section \ref{section:multiple}: conditional annealed complexity via a hierarchical skeleton.}}
	We extend the method of Section \ref{section:base} to compute a \emph{conditional annealed} complexity. We consider an optimizer $\zeta$ of the variational problem of Theorem \ref{theorem:annealed cond}. We fix an \emph{ancestor skeleton} of states $\bfm^{(1)},\ldots,\bfm^{(n-1)}$ such that $\Vert \bfm^{(i)}\Vert^2=Nq_i$ for every $i\in [n-1]$, and $\langle \bfm^{(i)}-\bfm^{(i-1)},\bfm^{(j)}\rangle=0$ for every $i\in \{2,\ldots,n-1\}$ and $j\in \{1,\ldots,i-1\}$. We further require that $(\bfm^{(1)},\ldots,\bfm^{(n-1)})$ has empirical statistics, tested against a specific family of test functions, consistent with the law of $(\partial_x\Phi_\zeta(q_1,X_{q_1}),\ldots,\partial_x\Phi_\zeta(q_{n-1},X_{q_{n-1}}))$. We then compute the annealed complexity of states $\bfm$ lying in the band orthogonal to $\bfm^{(n-1)}$ at free energy $f$, \emph{conditionally} on the skeleton being critical with prescribed free energy, i.e., conditionally on
	\begin{equation*}
		\nabla F_{\TAP,\zeta}(\bfm^{(i)})=0\quad\text{and}\quad F_{\TAP,\zeta}(\bfm^{(i)})=N\Pari(\zeta), \qquad i=1,\ldots,n-1.
	\end{equation*}
	The resulting SUSY ansatz is verified using the same arguments as in Section \ref{section:base}.

	\subsection{Notation}
	\begin{itemize}
		\item For every random variable $X$ on $\dR$ absolutely continuous with respect to Lebesgue, we denote by $\varphi_X$ the probability density function of $X$. Given another random variable $Y$ on $\dR$ absolutely continuous with respect to Lebesgue, we denote by $\varphi_{X\mid Y}(\cdot \mid y)$ the probability density function of $X$ given $Y=y$.
		\item We let $\mathbb{C}^+=\{z=x+iy:\ y>0\}$. 
		Let $\nu\in\mc{P}(\dR)$ be compactly supported and set
		\[
		a:=\inf\supp\nu,\qquad b:=\sup\supp\nu .
		\]
		We define its Stieltjes transform by
		\begin{equation}\label{def:Gnu}
			G_\nu:\bigl(\mathbb{C}^+\cup(\dR\setminus[a,b])\bigr)\to\mathbb{C},
			\qquad
			z\mapsto \int_{\dR}\frac{\dd\nu(x)}{z-x}.
		\end{equation}
		We then let $G_\nu(a)=\lim_{x\uparrow a}G_\nu(x)$ and $G_\nu(b)=\lim_{x\downarrow b}G_\nu(x)$. One can show that $G_\nu|_{\dR}$ is a bijection from $\dR\setminus [a,b]$ to $(G_\nu(a),G_\nu(b))\setminus \{0\}$.
		\item We denote $[N]$ for $\{1,\ldots,N\}$.
		\item For every probability measure $\zeta$ on $[0,1]$, we let \begin{align} \label{def:mcUzeta}
			\mc{U}_\zeta(q)&=\frac{1}{2}\int_q^1 s\xi''(s)\zeta([0,s])\dd s,\\ \label{def:hqm}
			h_\zeta(q,m)&=\Phi_\zeta^*(q,\cdot)(m),\\ \label{def:kqm}
			k_\zeta(q,m)&=\partial_m h_\zeta(q,m)+m\xi''(q)\int_q^1 \zeta([0,t])\dd t.
		\end{align}
		\item We write $\bfm\in [-1,1]^N$ for the vector $\bfm=(m_1,\ldots,m_N)$.
		\item For $\bfm\in [-1,1]^N$ we denote 
		\begin{equation*}
			q_\bfm=\frac{1}{N}\Vert \bfm\Vert^2.
		\end{equation*}
		\item For $\bfm\in [-1,1]^N$ we denote 
		\begin{equation*}
			\mu_N^{(\bfm)}=\frac{1}{N}\sum_{i=1}^N \delta_{m_i}.
		\end{equation*}
		\item For every $a<b$, we endow $\mc{P}([a,b])$ with the metric
		\begin{equation}\label{def:dist}
			\dist(\zeta,\zeta'):=\int_a^b\big|\zeta([a,s])-\zeta'([a,s])\big|\,\dd s,
		\end{equation}
		which corresponds to the $1$-Wasserstein distance on $\mc{P}([a,b])$ and hence metrizes weak convergence. 
		
		We will sometimes abuse notation when 
		$\zeta \in \mathcal{P}([a,b])$ and $\zeta' \in \mathcal{P}([a',b])$ with 
		$a \geq a'$ by extending $\zeta$ to $[a',b]$ and setting
		\begin{equation*}
			\mathrm{dist}(\zeta, \zeta') 
			:= \int_a^b \bigl|\zeta([a,s]) - \zeta'([a,s])\bigr| \, \mathrm{d}s.
		\end{equation*}
		\item Recall from \eqref{def:TAPmuzeta} and \eqref{def:TAPmu},
		\begin{equation*}
			\TAP(\mu,\zeta)=-\int_{-1}^1 \Phi_\zeta^*(q,m)\dd \mu(m)-\frac{1}{2}\int_q^1 t\xi''(t)\zeta([0,t])\dd t,\quad q:=\int m^2 \dd \mu(m),
		\end{equation*}
		\begin{equation*}
			\TAP(\mu)=\inf_{\zeta\in \mc{P}([q,1])} \TAP(\mu,\zeta).
		\end{equation*}
		\item If $\zeta\in \mc{P}([q,1])$, we often abuse notation and write $\zeta([0,t])$ for $t\in [q,1]$ instead of $\zeta([q,t])$.
		\item Recall from Definition~\ref{def:zetam} the TAP free energy at fixed overlap measure,
		\begin{equation*}
			F_{\TAP,\zeta}(\bfm)=H(\bfm)-\sum_{i=1}^N h_\zeta(q_\bfm,m_i)-N\mc{U}_\zeta(q_\bfm),
		\end{equation*}
		and the optimal overlap measure $\zeta_\bfm\in\mc{P}([q_\bfm,1])$, i.e.\ the unique minimizer in~\eqref{def:FTAP}.
		\item We write $\partial_m h_\zeta(q,m)$ and $\partial_{mm} h_\zeta(q,m)$ for the first and second partial derivatives of $h_\zeta(q,\cdot)$ with respect to~$m$.
		\item For $\bfm\in[-1,1]^N$ and $\zeta\in\mc{P}([0,1])$, we set
		\begin{equation*}
			Z(\bfm):=\bigl(\nabla F_{\TAP,\zeta}(\bfm),\,F_{\TAP,\zeta}(\bfm)\bigr)\in\dR^{N+1}.
		\end{equation*}
	\end{itemize}

	\subsection*{Acknowledgements}
	This work originates in a question posed by G\'erard Ben Arous, to whom we are deeply indebted for many illuminating discussions. An earlier version of this project was carried out jointly with Alice Guionnet, and we thank her for her contributions as well as for numerous further exchanges. We are also grateful to Jean-Christophe Mourrat for many generous discussions.

	This project was partly supported by the ERC Project LDRAM: ERC-2019-ADG Project 884584.

	\section{The base annealed computation}\label{section:base}
	
	In this section, we perform an annealed computation of the TAP free energy on a set of admissible states.

	\subsection{Statement of the annealed complexity result}
	
	In this subsection, we state the main result of Section~\ref{section:base}: the annealed computation of the TAP complexity. We introduce the SUSY constraint set $\SUSY_1(\zeta,\ve)$, which localizes the magnetization vector near a given prefix, and formulate the complexity in terms of the Parisi functional.
	
	Throughout Section~\ref{section:base}, we set
	\[
	\delta_\ve:=\ve^{\frac{1}{2}}.
	\]
	The only properties we use are that $\delta_\ve\downarrow 0$ and
	$\ve/\delta_\ve\to 0$ as $\ve\downarrow 0$.
	
	\begin{definition}\label{def:SUSY1eps}
		Let $\zeta\in\Prefix_2(u,q)$ for some $u,q\in(0,1)$.
		We define $\SUSY_1(\zeta,\ve)$ as the set of $\bfm\in[-1+\delta_\ve,1-\delta_\ve]^N$ such that
		\begin{equation*}
			\left|\frac{1}{N}\|\bfm\|^2-q\right|\leq\ve,
			\qquad
			\left|\frac{1}{N}\langle\partial_m h_\zeta(q,\bfm),\bfm\rangle
			-\xi'(q)\int_0^1\zeta([0,t])\,\dd t\right|\leq \ve,
		\end{equation*}
		and
		\begin{equation*}
			\dist\!\bigl(\zeta_\bfm,\,\zeta|_{(q,1]}+\zeta([0,q])\,\delta_q\bigr)\leq\ve,
		\end{equation*}
		where we recall that $\dist$ is as in \eqref{def:dist}.
	\end{definition}
	
	\begin{proposition}[Annealed complexity]\label{prop:annealed 0}
		Let $f\in\dR$, $u\in(0,1)$, $q\in(0,1)$, and $\zeta\in\Prefix_2(u,q)$. Recall $\Pari$ from \eqref{def:Par}.
		Define the \emph{prefix free energy} 
		\begin{equation*}
			\Ppar(u;q):=\inf_{\zeta'\in\Prefix_2(u,q)}\Pari(\zeta')
		\end{equation*}
		and the \emph{complexity functional at level $f$}
		\begin{equation}\label{def:Cf}
			\Cpx_f(u;q)
			\;:=\;
			u\,\bigl(\Ppar(u;q)-f\bigr).
		\end{equation}
		
		We assume:
		\begin{enumerate}
			\item[\emph{(i)}]
			\emph{(Tail optimality.)} The infimum $\Ppar(u;q)$ is attained at $\zeta$.
			\item[\emph{(ii)}]
			\emph{(Stationarity.)} $(u,q)$ is a critical point of $\Cpx_f$.
		\end{enumerate}
		
		Let $\ve\in(0,1)$ and let $\mc{N}_\ve(f,\zeta,q)$ be the number of critical
		points of $F_{\TAP,\zeta}$ in $\SUSY_1(\zeta,\ve)$ with
		$\frac{1}{N}F_{\TAP,\zeta}(\bfm)\in(f-\ve,f+\ve)$.  Then
		\[
		\lim_{\ve\to 0}\lim_{N\to\infty}
		\frac{1}{N}\log\dE\bigl[\mc{N}_\ve(f,\zeta,q)\bigr]
		\;=\;
		\Cpx_f(u;q)
		\;=\;
		\zeta(\{0\})\,\bigl(\Pari(\zeta)-f\bigr).
		\]
	\end{proposition}

	\subsection{Kac--Rice formula}
	
	We now develop the Kac--Rice machinery needed to count critical points of the
	TAP free energy.  The first step is an exact algebraic reduction, special to the
	pure $p$-spin case, that expresses the value of $F_{\TAP,\zeta}$ at a critical
	point in terms of a deterministic remainder.
	
	\begin{lemma}[Pure $p$-spin exact reduction]\label{lemma:reduction pure}
		Assume $\xi(t)=\beta_p^2 t^p$ for some $p\ge 2$.
		For $\bfm\in(-1,1)^N$, let
		\[
		q_\bfm=\frac1N\|\bfm\|^2,\qquad
		\mc S_\zeta(\bfm)=\sum_{i=1}^N h_\zeta(q_\bfm,m_i)+N\mc U_\zeta(q_\bfm),
		\]
		and
		\begin{equation}\label{eq:def-Rex-KR}
			\mc R^{\rm ex}_\zeta(\bfm)
			:=\frac1p\langle \bfm,\nabla \mc S_\zeta(\bfm)\rangle-\mc S_\zeta(\bfm).
		\end{equation}
		Then
		\[
		F_{\TAP,\zeta}(\bfm)
		=
		\frac1p\langle \bfm,\nabla F_{\TAP,\zeta}(\bfm)\rangle
		+\mc R^{\rm ex}_\zeta(\bfm).
		\]
		In particular, if $\nabla F_{\TAP,\zeta}(\bfm)=0$, then
		\[
		F_{\TAP,\zeta}(\bfm)=\mc R^{\rm ex}_\zeta(\bfm).
		\]
	\end{lemma}
	
	\begin{proof}
		Since $\xi(t)=\beta_p^2 t^p$, the Hamiltonian $H$ is a homogeneous polynomial
		of degree~$p$ in $\bfm$.  Euler's identity for homogeneous functions therefore
		gives
		\[
		H(\bfm)=\frac1p\langle \bfm,\nabla H(\bfm)\rangle,
		\qquad \bfm\in(-1,1)^N.
		\]
		Writing $F_{\TAP,\zeta}=H-\mc S_\zeta$ and inserting the identity above, we
		obtain
		\begin{align*}
			F_{\TAP,\zeta}(\bfm)
			&=
			\frac1p\langle \bfm,\nabla H(\bfm)\rangle-\mc S_\zeta(\bfm)\\
			&=
			\frac1p\langle \bfm,\nabla F_{\TAP,\zeta}(\bfm)\rangle
			+\frac1p\langle \bfm,\nabla \mc S_\zeta(\bfm)\rangle-\mc S_\zeta(\bfm)\\
			&=
			\frac1p\langle \bfm,\nabla F_{\TAP,\zeta}(\bfm)\rangle
			+\mc R^{\rm ex}_\zeta(\bfm),
		\end{align*}
		where the last step uses the definition~\eqref{eq:def-Rex-KR}.
		The second assertion follows by setting
		$\nabla F_{\TAP,\zeta}(\bfm)=0$.
	\end{proof}
	
	With this reduction in hand, we turn to the Kac--Rice formulas themselves.
	In the mixed case, both the gradient and the free-energy value appear as
	conditioning variables in the density, whereas in the pure case
	Lemma~\ref{lemma:reduction pure} eliminates the need to condition on the
	free-energy value, replacing it by a deterministic indicator.
	
	\begin{lemma}[Kac--Rice: mixed and pure cases]\label{lemma:KacRice}
		Let $\zeta$ be a probability measure on $[0,1]$ with a $2$-atom prefix
		$\{0,q\}$ for some $q\in(0,1)$, and let $\ve\in(0,q/2)$.
		Let $\mc N_\ve(f,\zeta,q)$ be the number of critical points $\bfm$ of
		$F_{\TAP,\zeta}$ such that
		\[
		\bfm\in \SUSY_1(\zeta,\ve)
		\qquad\text{and}\qquad
		\frac1N F_{\TAP,\zeta}(\bfm)\in(f-\ve,f+\ve).
		\]
		For $\bfm\in(-1,1)^N$, write $q_\bfm:=\frac1N\|\bfm\|^2$.
		
		\begin{enumerate}[label=(\roman*)]
			\item \textup{(Mixed case.)}
			Assume that for every $\bfm\in(-1,1)^N$ with
			$q_\bfm\in(q-\ve,q+\ve)$, the Gaussian vector
			\[
			\bigl(\nabla F_{\TAP,\zeta}(\bfm),\,F_{\TAP,\zeta}(\bfm)\bigr)
			\]
			is non-degenerate.  In particular, this holds for a genuine mixed
			$p$-spin Hamiltonian, i.e.\ when
			\[
			\xi(t)=\sum_{p\ge 2}\beta_p^2 t^p
			\]
			has at least two non-zero coefficients.
			Then
			\begin{multline}\label{eq:KacRice}
				\dE[\mc N_\ve(f,\zeta,q)]
				=N\int_{[-1,1]^N}\int_{(f-\ve,f+\ve)}
				\dE\!\left[
				|\det(\nabla^2 F_{\TAP,\zeta}(\bfm))|
				\,\middle|\,
				\nabla F_{\TAP,\zeta}(\bfm)=0,\,
				F_{\TAP,\zeta}(\bfm)=Nf'
				\right]
				\\
				\times
				\varphi_{(\nabla F_{\TAP,\zeta}(\bfm),\,F_{\TAP,\zeta}(\bfm))}(0,Nf')\,
				\indic_{\SUSY_1(\zeta,\ve)}(\bfm)\,
				\dd\bfm\,\dd f'.
			\end{multline}
			
			\item \textup{(Pure case.)}
			Assume that $\xi(t)=\beta_p^2 t^p$ is pure $p$-spin for some $p\ge2$.
			Recalling the notation~\eqref{eq:def-Rex-KR},
			\begin{multline}\label{eq:KacRice-pure}
				\dE[\mc N_\ve(f,\zeta,q)]
				=
				\int_{[-1,1]^N}
				\dE\!\left[
				|\det(\nabla^2 F_{\TAP,\zeta}(\bfm))|
				\,\middle|\,
				\nabla F_{\TAP,\zeta}(\bfm)=0
				\right]
				\varphi_{\nabla F_{\TAP,\zeta}(\bfm)}(0)
				\\
				\times
				\indic_{\SUSY_1(\zeta,\ve)}(\bfm)\,
				\indic_{(f-\ve,f+\ve)}
				\!\left(\frac1N\mc R^{\rm ex}_\zeta(\bfm)\right)
				\dd\bfm.
			\end{multline}
		\end{enumerate}
	\end{lemma}
	
	\begin{proof}
		\noindent\textbf{Step 1: reduction to the interior.}
		Set
		\[
		U:=(-1,1)^N,
		\qquad
		U_{q,\ve}:=\left\{\bfm\in U:\left|\frac1N\|\bfm\|^2-q\right|<\ve\right\}.
		\]
		Since $\ve<q/2$, every $\bfm\in U_{q,\ve}$ satisfies
		$q_\bfm>q/2>0$.  By definition of $\SUSY_1(\zeta,\ve)$,
		every critical point counted by $\mc N_\ve(f,\zeta,q)$ belongs to $U_{q,\ve}$.
		
		We first verify that no critical points lie on the boundary
		$\partial([-1,1]^N)$.  Recall that
		\[
		F_{\TAP,\zeta}(\bfm)
		=
		H(\bfm)-\sum_{i=1}^N \Phi_\zeta^*(q_\bfm,m_i)-N\mc U_\zeta(q_\bfm),
		\qquad
		q_\bfm=\frac1N\|\bfm\|^2.
		\]
		Fix $r\in[0,1]$.  Since $x\mapsto \Phi_\zeta(r,x)$ is strictly convex,
		the map $x\mapsto \partial_x\Phi_\zeta(r,x)$ is strictly increasing with
		range $(-1,1)$, and therefore
		\[
		\partial_m\Phi_\zeta^*(r,m)
		=
		(\partial_x\Phi_\zeta(r,\cdot))^{-1}(m),
		\qquad m\in(-1,1).
		\]
		In particular,
		\[
		\lim_{m\uparrow 1}\partial_m\Phi_\zeta^*(r,m)=+\infty,
		\qquad
		\lim_{m\downarrow -1}\partial_m\Phi_\zeta^*(r,m)=-\infty.
		\]
		Hence $F_{\TAP,\zeta}$ is $C^2$ on~$U$, while if some coordinate
		$m_i=\pm1$, then the partial derivative $\partial_iF_{\TAP,\zeta}(\bfm)$
		diverges because of the term $\partial_m\Phi_\zeta^*(q_\bfm,m_i)$.
		Thus there are no critical points on $\partial([-1,1]^N)$, and the critical
		points of $F_{\TAP,\zeta}$ in $[-1,1]^N$ coincide with those in~$U$.
		
		\medskip\noindent\textbf{Step 2: mixed case.}
		Assume~\textup{(i)}.  Since the TAP correction is deterministic, the covariance
		of
		$\bigl(\nabla F_{\TAP,\zeta}(\bfm),F_{\TAP,\zeta}(\bfm)\bigr)$
		coincides with that of $\bigl(\nabla H(\bfm),H(\bfm)\bigr)$.
		For $\bfm\in U_{q,\ve}$ and $r:=q_\bfm$, the covariance identities read
		\[
		\Cov(\partial_iH(\bfm),\partial_jH(\bfm))
		=
		\xi'(r)\delta_{ij}+\frac{\xi''(r)}{N}m_i m_j,
		\]
		\[
		\Cov(\partial_iH(\bfm),H(\bfm))
		=
		\xi'(r)m_i,
		\qquad
		\Var(H(\bfm))=N\xi(r).
		\]
		Hence for every $(x,v)\in\dR^N\times\dR$,
		\begin{equation}\label{eq:var-mixed-KR}
			\Var\!\bigl(\langle x,\nabla H(\bfm)\rangle+vH(\bfm)\bigr)
			=
			\xi'(r)\|x\|^2+\frac{\xi''(r)}{N}\langle \bfm,x\rangle^2
			+2v\xi'(r)\langle \bfm,x\rangle
			+N\xi(r)v^2.
		\end{equation}
		If $\xi(t)=\sum_{p\ge2}\beta_p^2 t^p$ has at least two non-zero coefficients,
		write $x=x_\perp+a\bfm$ with $\langle x_\perp,\bfm\rangle=0$.  Then
		\[
		\eqref{eq:var-mixed-KR}
		=
		\xi'(r)\|x_\perp\|^2
		+N\begin{pmatrix}a&v\end{pmatrix}
		\begin{pmatrix}
			r(\xi'(r)+r\xi''(r)) & r\xi'(r)\\
			r\xi'(r) & \xi(r)
		\end{pmatrix}
		\begin{pmatrix}a\\ v\end{pmatrix}.
		\]
		The determinant of the $2\times2$ matrix equals
		\begin{align*}
			&\det\begin{pmatrix}
				r(\xi'(r)+r\xi''(r)) & r\xi'(r)\\
				r\xi'(r) & \xi(r)
			\end{pmatrix} \\
			&=
			\left(\sum_{p\ge2}\beta_p^2 r^p\right)
			\left(\sum_{p\ge2}p^2\beta_p^2 r^p\right)
			-
			\left(\sum_{p\ge2}p\beta_p^2 r^p\right)^2
			=
			\frac12\sum_{p,\ell\ge2}\beta_p^2\beta_\ell^2(p-\ell)^2r^{p+\ell}>0,
		\end{align*}
		so the quadratic form is positive for every $(x,v)\neq(0,0)$.
		The non-degeneracy hypothesis in~\textup{(i)} is therefore satisfied in the
		mixed $p$-spin case.
		
		We now apply the Kac--Rice formula (e.g.\ \cite{adlertaylor},
		Corollaries~11.3.2 and~11.3.5) to the $(N+1)$-dimensional Gaussian field
		\[
		\bfm\mapsto \bigl(\nabla F_{\TAP,\zeta}(\bfm),F_{\TAP,\zeta}(\bfm)\bigr)
		\]
		on the open set $U_{q,\ve}$, restricted to the Borel set
		$\SUSY_1(\zeta,\ve)\cap U_{q,\ve}$ and the level window
		$F_{\TAP,\zeta}(\bfm)\in(N(f-\ve),N(f+\ve))$.  This gives
		\begin{multline*}
			\dE[\mc N_\ve(f,\zeta,q)]
			=
			\int_{U_{q,\ve}}
			\int_{(N(f-\ve),N(f+\ve))}
			\dE\!\left[
			|\det(\nabla^2 F_{\TAP,\zeta}(\bfm))|
			\,\middle|\,
			\nabla F_{\TAP,\zeta}(\bfm)=0,\,
			F_{\TAP,\zeta}(\bfm)=t
			\right]
			\\
			\times
			\varphi_{(\nabla F_{\TAP,\zeta}(\bfm),\,F_{\TAP,\zeta}(\bfm))}(0,t)\,
			\indic_{\SUSY_1(\zeta,\ve)}(\bfm)\,
			\dd\bfm\,\dd t.
		\end{multline*}
		The change of variables $t=Nf'$ yields $\dd t=N\,\dd f'$ and therefore
		\begin{multline*}
			\dE[\mc N_\ve(f,\zeta,q)]
			=
			N\int_{U_{q,\ve}}
			\int_{(f-\ve,f+\ve)}
			\dE\!\left[
			|\det(\nabla^2 F_{\TAP,\zeta}(\bfm))|
			\,\middle|\,
			\nabla F_{\TAP,\zeta}(\bfm)=0,\,
			F_{\TAP,\zeta}(\bfm)=Nf'
			\right]
			\\
			\times
			\varphi_{(\nabla F_{\TAP,\zeta}(\bfm),F_{\TAP,\zeta}(\bfm))}(0,Nf')\,
			\indic_{\SUSY_1(\zeta,\ve)}(\bfm)\,
			\dd\bfm\,\dd f'.
		\end{multline*}
		Finally, since $\indic_{\SUSY_1(\zeta,\ve)}(\bfm)=0$ outside $U_{q,\ve}$ and
		$[-1,1]^N\setminus U$ has Lebesgue measure zero, we may extend the outer
		integral to $[-1,1]^N$, which gives~\eqref{eq:KacRice}.
		
		\medskip\noindent\textbf{Step 3: pure case.}
		Assume now~\textup{(ii)} and set $Y(\bfm):=\nabla F_{\TAP,\zeta}(\bfm)$ for
		$\bfm\in U$.  By Lemma~\ref{lemma:reduction pure},
		\begin{equation}\label{eq:pure-critical-identity-KR}
			Y(\bfm)=0
			\quad\Longrightarrow\quad
			F_{\TAP,\zeta}(\bfm)=\mc R^{\rm ex}_\zeta(\bfm).
		\end{equation}
		Therefore
		\[
		\mc N_\ve(f,\zeta,q)
		=
		\#\left\{
		\bfm\in U_{q,\ve}:
		Y(\bfm)=0,\,
		\bfm\in\SUSY_1(\zeta,\ve),\,
		\frac1N\mc R^{\rm ex}_\zeta(\bfm)\in(f-\ve,f+\ve)
		\right\}.
		\]
		
		For each fixed $\bfm\in U_{q,\ve}$, the vector $Y(\bfm)$ is Gaussian with
		covariance
		\[
		\Cov(Y(\bfm))
		=
		\Cov(\nabla H(\bfm))
		=
		\xi'(q_\bfm)\Id+\frac{\xi''(q_\bfm)}{N}\bfm\bfm^\top.
		\]
		Since $q_\bfm>q/2>0$ and $\xi'(0)=0<\xi'(q_\bfm)$ by strict convexity
		of~$\xi$, this covariance matrix is positive definite.  Hence $Y(\bfm)$ is
		non-degenerate for every $\bfm\in U_{q,\ve}$.
		
		We may therefore apply the Kac--Rice formula to the $N$-dimensional Gaussian
		field~$Y$ on the Borel set
		\[
		A_{\zeta,\ve,f}
		:=
		\left\{
		\bfm\in U_{q,\ve}:
		\bfm\in \SUSY_1(\zeta,\ve),\,
		\frac1N\mc R^{\rm ex}_\zeta(\bfm)\in(f-\ve,f+\ve)
		\right\}.
		\]
		Since $\nabla Y(\bfm)=\nabla^2 F_{\TAP,\zeta}(\bfm)$, this yields
		\begin{multline*}
			\dE[\mc N_\ve(f,\zeta,q)]
			=
			\int_{U_{q,\ve}}
			\dE\!\left[
			|\det(\nabla^2 F_{\TAP,\zeta}(\bfm))|
			\,\middle|\,
			\nabla F_{\TAP,\zeta}(\bfm)=0
			\right]
			\varphi_{\nabla F_{\TAP,\zeta}(\bfm)}(0)
			\\
			\times
			\indic_{\SUSY_1(\zeta,\ve)}(\bfm)\,
			\indic_{(f-\ve,f+\ve)}
			\!\left(\frac1N\mc R^{\rm ex}_\zeta(\bfm)\right)
			\dd\bfm.
		\end{multline*}
		As before, $\indic_{\SUSY_1(\zeta,\ve)}$ vanishes outside $U_{q,\ve}$ and
		$[-1,1]^N\setminus U$ has Lebesgue measure zero, so we may extend the integral
		to $[-1,1]^N$, giving~\eqref{eq:KacRice-pure}.
	\end{proof}

	\subsection{Reduction from the Ward identities}
	
	In the Kac--Rice formula \eqref{eq:KacRice}, we must evaluate the joint density of 
	\((\nabla F_{\TAP,\zeta}(\bfm),\,F_{\TAP,\zeta}(\bfm))\) at \((0,Nf)\). Leveraging the SUSY
	heuristics from Section \ref{sub:SUSY}, we derive a particularly simple representation in the next lemma.

	\begin{lemma}\label{lemma:prob0 non deg}
		Assume that $\xi$ is a genuine mixture, i.e.\ $\xi(t)=\sum_p \beta_p^2 t^p$ with
		at least two nonzero coefficients.  In particular, the quantity
		\begin{equation}\label{eq:Dq def}
			D(q):=\xi(q)\bigl(\xi'(q)+q\,\xi''(q)\bigr)-q\,\xi'(q)^2
			=\tfrac12\sum_{a,b}\beta_a^2\beta_b^2(a-b)^2\,q^{a+b-1}
		\end{equation}
		is strictly positive on $(0,1)$.
		Let $\ve>0$, $u\in(0,1)$, $q\in(0,1)$ and $f\in\dR$.
		Let $\zeta\in\Prefix_2(u,q)$ where $\Prefix_2$ is as in
		Definition~\ref{def:nprefix}.
		
		Let $\delta\in[\ve,1)$ and $\bfm\in[-1+\delta,1-\delta]^N$ and denote
		$Z(\bfm):=\bigl(\nabla F_{\TAP,\zeta}(\bfm),\,F_{\TAP,\zeta}(\bfm)\bigr)$. Suppose in addition that $q_\bfm\in[q-\ve,q+\ve]$ and
		$\dist\!\bigl(\zeta_\bfm,\,\zeta|_{(q,1]}+\zeta([0,q])\,\delta_q\bigr)\leq \ve$,
		where both measures are viewed as elements of $\mc{P}([0,1])$ and $\dist$
		is the metric~\eqref{def:dist}.
		
		\begin{enumerate}[label=(\roman*)]
			\item \label{eq:prob0 1} The log-density of $Z(\bfm)$ at $(0,Nf)$ satisfies
			\begin{multline}\label{eq:upper bound proba}
				\log \varphi_{Z(\bfm)}(0,Nf)\\ \leq \frac{1}{2}N\xi(q)u^2-\frac{1}{2}N\xi''(q)\left(\int_q^1 \zeta([0,t])\dd t\right)^2 -\frac{1}{2\xi'(q)}\Vert \partial_m h_\zeta(q,\bfm)-u\xi'(q)\bfm\Vert^2\\-\frac{N}{2}\log (2\pi \xi'(q))+\frac{N}{2\xi'(q)q}\left( \frac{1}{N}\langle \partial_m h_\zeta(q,\bfm),\bfm\rangle-\xi'(q)\int_0^1 \zeta([0,t])\dd t \right)^2\\- u \left(\sum_{i=1}^N h_\zeta(q,m_i)+N\mc{U}_\zeta(q)+Nf\right)
				+O\!\left(N\frac{\ve}{\delta}\right)+O(\log N),
			\end{multline}
			where the implicit constant depends only on $\xi$, $q$, and $u$.
			
			\item \label{eq:prob0 2} Suppose in addition that
			\begin{equation}\label{eq:ass1}
				\langle \partial_m h_\zeta(q,\bfm),\bfm\rangle=N\xi'(q)\int_0^1 \zeta([0,t])\dd t +O(N\ve)
			\end{equation}
			and
			\begin{equation}\label{eq:ass2}
				\sum_{i=1}^N (h_\zeta(q,m_i)-\partial_m h_\zeta(q,m_i)m_i)+N\int_0^q t\xi''(t)\zeta([0,t])\dd t+\frac{1}{2}N \int_q^1 t\xi''(t)\zeta([0,t]) \dd t=-Nf+O(N\ve).
			\end{equation}
			Then,
			\begin{multline}\label{eq:lower bound proba}
				\log \varphi_{Z(\bfm)}(0,Nf) \\= \frac{1}{2}N\xi(q)u^2-\frac{1}{2}N\xi''(q)\left(\int_q^1 \zeta([0,t])\dd t\right)^2 -\frac{1}{2\xi'(q)}\Vert \partial_m h_\zeta(q,\bfm)-u\xi'(q)\bfm\Vert^2-\frac{N}{2}\log (2\pi \xi'(q))\\-u \left(\sum_{i=1}^N h_\zeta(q,m_i)+N\mc{U}_\zeta(q)+Nf\right)
				+O\!\left(N\frac{\ve}{\delta}\right)+O(\log N).
			\end{multline}
		\end{enumerate}
		
	\end{lemma}
	
	\begin{proof}
		Throughout we set $q_\bfm:=\frac{1}{N}\|\bfm\|^2$.
		
		\smallskip\noindent\textbf{Step 1: exact Gaussian representation at $q_\bfm$}
		
		By definition~\eqref{def:FTAP},
		\[F_{\TAP,\zeta}(\bfm)
		=H(\bfm)-\sum_{i=1}^N h_\zeta(q_\bfm,m_i)-N\mc{U}_\zeta(q_\bfm).\]
		For the gradient, the exact computation in the proof of Lemma~\ref{lemma:gradient TAP}
		(see~\eqref{eq:diG reduced} and~\eqref{eq:xx}) gives, for every $i\in[N]$,
		\begin{equation}\label{eq:grad claim}
			\partial_i F_{\TAP,\zeta}(\bfm)
			=\partial_i H(\bfm)-k_\zeta(q_\bfm,m_i)+\Delta_i(\bfm,\zeta),
		\end{equation}
		where $\Delta_i$ is controlled by~\eqref{eq:gradTap-stability}.
		Indeed the comparison measure
		$\zeta|_{(q_\bfm,1]}+\zeta([0,q_\bfm])\delta_{q_\bfm}$ has an atom of size
		$\zeta([0,q_\bfm])\ge u$ at $q_\bfm$, so
		\[
		|\Delta_i|
		\le \frac{C}{\delta}\,
		\dist\!\bigl(\zeta_\bfm,\,\zeta|_{(q_\bfm,1]}+\zeta([0,q_\bfm])\delta_{q_\bfm}\bigr).
		\]
		Moreover, the $\dist$-distance between the measures
		$\zeta|_{(q_\bfm,1]}+\zeta([0,q_\bfm])\delta_{q_\bfm}$ and
		$\zeta|_{(q,1]}+\zeta([0,q])\delta_q$ is at most $C\ve$; together with the
		hypothesis
		$\dist(\zeta_\bfm,\,\zeta|_{(q,1]}+\zeta([0,q])\delta_q)<\ve$ and the triangle
		inequality, this yields
		\[
		|\Delta_i|\le C'\frac{\ve}{\delta}
		\]
		componentwise.
		
		Write $\Delta:=(\Delta_1,\ldots,\Delta_N)$ and define
		\[
		z_\bfm:=\begin{pmatrix}
			k_\zeta(q_\bfm,\bfm)-\Delta\\[3pt]
			Nf+\sum_{i=1}^N h_\zeta(q_\bfm,m_i)+N\mc{U}_\zeta(q_\bfm)
		\end{pmatrix},
		\qquad
		\Gamma_\bfm:=\operatorname{Cov}\bigl(Z(\bfm)\bigr).
		\]
		Then $Z(\bfm)$ is Gaussian with mean $-z_\bfm$ and covariance $\Gamma_\bfm$, so
		\begin{equation}\label{eq:exact log density}
			\log\varphi_{Z(\bfm)}(0,Nf)
			=-\tfrac{1}{2}\langle z_\bfm,\Gamma_\bfm^{-1}\,z_\bfm\rangle
			-\tfrac{1}{2}\log\det(2\pi\Gamma_\bfm).
		\end{equation}
		
		\smallskip\noindent\textbf{Step 2: block structure of $\Gamma_\bfm$}
		
		The covariance $\Gamma_\bfm$ admits the $(N\!+\!1)\times(N\!+\!1)$ block decomposition
		\[
		\Gamma_\bfm=\begin{pmatrix} A_\bfm & b_\bfm \\ b_\bfm^\top & c_\bfm\end{pmatrix},
		\quad
		A_\bfm=\xi'(q_\bfm)I_N+\tfrac{\xi''(q_\bfm)}{N}\bfm\bfm^\top,
		\quad
		b_\bfm=\xi'(q_\bfm)\bfm,
		\quad
		c_\bfm=N\xi(q_\bfm).
		\]
		By the matrix determinant lemma,
		$\det A_\bfm=\xi'(q_\bfm)^{N-1}\bigl(\xi'(q_\bfm)+q_\bfm\xi''(q_\bfm)\bigr)$,
		and the Schur complement
		$S_\bfm:=c_\bfm-b_\bfm^\top A_\bfm^{-1}b_\bfm
		=\frac{N\,D(q_\bfm)}{\xi'(q_\bfm)+q_\bfm\xi''(q_\bfm)}$
		is strictly positive since $\xi$ is a mixture. Hence
		\begin{equation}\label{eq:logdet Gamma}
			\log\det(\Gamma_\bfm)
			=(N\!-\!1)\log\xi'(q_\bfm)+\log\bigl(\xi'(q_\bfm)+q_\bfm\xi''(q_\bfm)\bigr)
			+\log\frac{N\,D(q_\bfm)}{\xi'(q_\bfm)+q_\bfm\xi''(q_\bfm)}.
		\end{equation}
		Since $|q_\bfm-q|\leq\ve$ and all functions of $q$ appearing above are smooth on $(0,1)$,
		\begin{equation}\label{eq:logdet frozen}
			\log\det(\Gamma_\bfm)=N\log\xi'(q)+O_\ve(N)+O(\log N).
		\end{equation}
		For the inverse, the standard block-inversion formula gives
		\begin{equation}\label{eq:Gamma inv blocks}
			\Gamma_\bfm^{-1}=\begin{pmatrix}
				A_\bfm^{-1}+\frac{1}{S_\bfm}A_\bfm^{-1}b_\bfm b_\bfm^\top A_\bfm^{-1}
				& -\frac{1}{S_\bfm}A_\bfm^{-1}b_\bfm \\[4pt]
				-\frac{1}{S_\bfm}b_\bfm^\top A_\bfm^{-1}
				& \frac{1}{S_\bfm}
			\end{pmatrix},
		\end{equation}
		so the gradient--gradient block of $\Gamma_\bfm^{-1}$ is $O(1)$,
		the scalar entry is $O(1/N)$, and the cross entries are $O(1/N)$.
		
		\smallskip\noindent\textbf{Step 3: upper bound} By the convex-dual identity
		\[-\langle z_\bfm,\Gamma_\bfm^{-1}z_\bfm\rangle
		=\inf_{w}\bigl\{\langle w,\Gamma_\bfm w\rangle-2\langle w,z_\bfm\rangle\bigr\},\]
		we bound the right-hand side from above by evaluating at a test point
		$w=(x,u)$ with the constraint
		$\langle\bfm,x\rangle=N\!\int_q^1\zeta([0,t])\dd t$.
		
		Since $k_\zeta(q_\bfm,m)=\partial_m h_\zeta(q_\bfm,m)+m\xi''(q_\bfm)\int_{q_\bfm}^1\zeta([0,t])\dd t$
		and $\int_{q_\bfm}^1\zeta([0,t])\dd t=\int_q^1\zeta([0,t])\dd t+O(\ve)$, the
		$\xi''$-terms in the quadratic form and the linear form cancel up to
		$O(N\ve/\delta)$ on the constraint set.
		The defect $\Delta$ contributes
		$2|\langle x,\Delta\rangle|\le C\frac{\ve}{\delta}\bigl(\|x\|^2+N\bigr)$,
		which is controlled by the positive $\xi'(q_\bfm)\|x\|^2$ term and an additive
		$O(N\ve/\delta)$.
		After these cancellations,
		\begin{multline}\label{eq:majo}
			-\langle z_\bfm,\Gamma_\bfm^{-1}z_\bfm\rangle
			\leq N\xi(q_\bfm)u^2
			-N\xi''(q_\bfm)\!\left(\int_q^1\zeta([0,t])\dd t\right)^{\!2}
			+O\!\left(N\frac{\ve}{\delta}\right)\\
			+\inf_{x:\langle x,\bfm\rangle=N\int_q^1\zeta([0,t])\dd t}
			\Bigl(\xi'(q_\bfm)\|x\|^2
			-2\langle\partial_m h_\zeta(q_\bfm,\bfm)-u\xi'(q_\bfm)\bfm,\,x\rangle\Bigr)\\
			-2u\bigl(\sum_{i=1}^N h_\zeta(q_\bfm,m_i)+N\mc{U}_\zeta(q_\bfm)+Nf\bigr).
		\end{multline}
		Optimising the constrained infimum over $x$ by Lagrange duality
		(using $\|\bfm\|^2=Nq_\bfm$ and
		$u q_\bfm+\int_q^1\zeta([0,t])\dd t=\int_0^1\zeta([0,t])\dd t+O(\ve)$)
		and then optimising the concave quadratic in the multiplier~$\lambda$,
		\begin{multline}\label{eq:quadra}
			\inf_{x:\langle x,\bfm\rangle=N\int_q^1\zeta([0,t])\dd t}
			\Bigl(\xi'(q_\bfm)\|x\|^2
			-2\bigl\langle\partial_m h_\zeta(q_\bfm,\bfm)-u\xi'(q_\bfm)\bfm,\,x\bigr\rangle\Bigr)
			\\
			=-\frac{\|\partial_m h_\zeta(q_\bfm,\bfm)-u\xi'(q_\bfm)\bfm\|^2}{\xi'(q_\bfm)}
			+\frac{N}{\xi'(q_\bfm)q_\bfm}
			\Bigl(\frac{1}{N}\bigl\langle\partial_m h_\zeta(q_\bfm,\bfm),\bfm\bigr\rangle
			-\xi'(q_\bfm)\!\int_0^1\zeta([0,t])\dd t\Bigr)^2
			+O_\ve(N).
		\end{multline}
		
		All quantities in~\eqref{eq:majo}--\eqref{eq:quadra} involve $q_\bfm$ rather than $q$.
		Since $|q_\bfm-q|\leq\ve$ and the functions $h_\zeta(\cdot,m)$, $\mc{U}_\zeta$, $\xi'$,
		$\xi''$, $\xi$ are smooth in their overlap argument, replacing each occurrence of $q_\bfm$ by
		$q$ introduces errors that are $O(\ve/\delta)$ pointwise in $m\in[-1+\delta,1-\delta]$ and $O(N\ve/\delta)$ when summed over $i$. In the positive-square term, the
		cross-contribution from the $O(\ve)$ shift only inflates the upper bound and is therefore
		harmless. Together with~\eqref{eq:logdet frozen} and~\eqref{eq:exact log density},
		this yields~\eqref{eq:upper bound proba} with overall remainder
		$O(N\ve/\delta)+O(\log N)$.
		
		\smallskip\noindent\textbf{Step 4: lower bound}
		
		Suppose that~\eqref{eq:ass1} and~\eqref{eq:ass2} hold.  Set
		\begin{equation*}
			x:=\frac{1}{\xi'(q_\bfm)}\bigl(\partial_m h_\zeta(q_\bfm,\bfm)-u\xi'(q_\bfm)\bfm\bigr),
			\qquad w:=(x,u)\in\dR^{N+1}.
		\end{equation*}
		Since the quadratic $w'\mapsto\frac{1}{2}\langle w',\Gamma_\bfm w'\rangle-\langle w',z_\bfm\rangle$
		is strictly convex, defining $r:=\Gamma_\bfm w-z_\bfm$ we have
		\begin{equation}\label{eq:approx saddle}
			\tfrac{1}{2}\langle w,\Gamma_\bfm w\rangle-\langle w,z_\bfm\rangle
			=-\tfrac{1}{2}\langle z_\bfm,\Gamma_\bfm^{-1}z_\bfm\rangle
			+\tfrac{1}{2}\langle r,\Gamma_\bfm^{-1}r\rangle.
		\end{equation}
		
		For the gradient block of $r$, a direct computation gives
		\[
		\langle \bfm,x\rangle
		=\frac{1}{\xi'(q_\bfm)}\bigl(\langle\partial_m h_\zeta(q_\bfm,\bfm),\bfm\rangle
		-u\xi'(q_\bfm)Nq_\bfm\bigr).
		\]
		By~\eqref{eq:ass1} and
		$\partial_m h_\zeta(q_\bfm,m_i)=\partial_m h_\zeta(q,m_i)+O(\ve)$,
		we get $\langle\bfm,x\rangle=N\int_{q_\bfm}^1\zeta([0,t])\dd t+O(N\ve)$.
		Using the definition of $k_\zeta(q_\bfm,\cdot)$ and the defect bound
		$\|\Delta\|_\infty\le C'\ve/\delta$ from Step~1, this gives
		\begin{equation}\label{eq:res grad}
			[r]_{\mathrm{grad}}=O(\ve/\delta)\text{ componentwise}.
		\end{equation}
		
		For the scalar block of $r$, using
		$\langle\partial_m h_\zeta(q_\bfm,\bfm),\bfm\rangle
		=\sum_i \partial_m h_\zeta(q_\bfm,m_i)m_i$
		together with~\eqref{eq:ass1} and $\partial_m h_\zeta(q_\bfm,\cdot)=\partial_m h_\zeta(q,\cdot)+O(\ve)$,
		\begin{equation*}
			\sum_{i=1}^N h_\zeta(q_\bfm,m_i)
			=\sum_{i=1}^N \bigl(h_\zeta(q_\bfm,m_i)-\partial_m h_\zeta(q_\bfm,m_i)m_i\bigr)
			+N\xi'(q_\bfm)\Bigl(uq_\bfm+\int_{q_\bfm}^1 \zeta([0,t])\dd t\Bigr)+O(N\ve).
		\end{equation*}
		Combined with $u\bigl(\xi(q_\bfm)-q_\bfm\xi'(q_\bfm)\bigr)=-\int_0^{q_\bfm} t\xi''(t)\zeta([0,t])\dd t$
		(which holds since $\zeta([0,t])=u$ for $t\in[0,q)$),
		\begin{multline*}
			-\sum_{i=1}^N h_\zeta(q_\bfm,m_i)+\xi'(q_\bfm)\langle\bfm,x\rangle+Nu\xi(q_\bfm)-N\mc{U}_\zeta(q_\bfm)\\
			=-\sum_{i=1}^N\bigl(h_\zeta(q_\bfm,m_i)-\partial_m h_\zeta(q_\bfm,m_i)m_i\bigr)
			-N\!\int_0^{q_\bfm}\! t\xi''(t)\zeta([0,t])\dd t
			-\frac{N}{2}\!\int_{q_\bfm}^1\! t\xi''(t)\zeta([0,t])\dd t
			+O(N\ve).
		\end{multline*}
		Replacing $q_\bfm$ by $q$ on the right-hand side costs $O(N\ve)$,
		and by~\eqref{eq:ass2} the resulting expression equals $Nf+O(N\ve)$.
		Since $[z_\bfm]_{N+1}=Nf+\sum_i h_\zeta(q_\bfm,m_i)+N\mc{U}_\zeta(q_\bfm)$,
		\begin{equation}\label{eq:res scalar}
			[r]_{N+1}=O(N\ve).
		\end{equation}
		
		By~\eqref{eq:res grad}, \eqref{eq:res scalar}, and the blockwise scaling of
		$\Gamma_\bfm^{-1}$ from~\eqref{eq:Gamma inv blocks},
		\[
		\langle r,\Gamma_\bfm^{-1}r\rangle
		\le C\bigl(\|r_{\mathrm{grad}}\|^2+\tfrac{1}{N}|r_{N+1}|^2\bigr)
		=O\!\left(N\frac{\ve^2}{\delta^2}\right)+O(N\ve^2)
		=O\!\left(N\frac{\ve}{\delta}\right).
		\]
		
		Evaluating $\frac{1}{2}\langle w,\Gamma_\bfm w\rangle-\langle w,z_\bfm\rangle$
		at $w=(x,u)$ by the same algebra as in Step~3 (the constrained optimisation in~\eqref{eq:quadra}
		is now nearly attained since
		$\langle\bfm,x\rangle=N\int_{q_\bfm}^1\zeta([0,t])\dd t+O(N\ve)$),
		and replacing $q_\bfm$ by $q$ at the cost of $O(N\ve/\delta)$,
		\begin{multline*}
			\tfrac{1}{2}\langle w,\Gamma_\bfm w\rangle-\langle w,z_\bfm\rangle
			=\frac{1}{2}N\xi(q)u^2
			-\frac{1}{2}N\xi''(q)\Bigl(\int_q^1 \zeta([0,t])\dd t\Bigr)^{\!2}
			-\frac{1}{2\xi'(q)}\Vert \partial_m h_\zeta(q,\bfm)-u\xi'(q)\bfm \Vert^2\\
			-u \Bigl(\sum_{i=1}^N h_\zeta(q,m_i)+N\mc{U}_\zeta(q)+Nf\Bigr)
			+O\!\left(N\frac{\ve}{\delta}\right),
		\end{multline*}
		where~\eqref{eq:ass1} ensures that the positive-square term
		from~\eqref{eq:quadra} contributes $O(N\ve^2)=O(N\ve/\delta)$.
		Since $\langle r,\Gamma_\bfm^{-1}r\rangle=O(N\ve/\delta)$
		by~\eqref{eq:approx saddle},
		\[
		-\tfrac{1}{2}\langle z_\bfm,\Gamma_\bfm^{-1}z_\bfm\rangle
		=\tfrac{1}{2}\langle w,\Gamma_\bfm w\rangle-\langle w,z_\bfm\rangle
		+O\!\left(N\frac{\ve}{\delta}\right).
		\]
		Together with~\eqref{eq:exact log density} and~\eqref{eq:logdet frozen},
		this gives~\eqref{eq:lower bound proba}.
	\end{proof}
	
	Let us now discuss the degenerate case.

	\begin{lemma}
		\label{lemma:prob0 pure}
		Assume that $\xi$ is pure $p$-spin. Let $\ve>0$. Let $u\in(0,1)$, $q\in(0,1)$, $f\in\dR$, and let $\zeta\in \Prefix_2(u,q)$ where $\Prefix_2$ is as in Definition \ref{def:nprefix}.
		
		Let $\delta\in[\ve,1)$ and $\bfm\in[-1+\delta,1-\delta]^N$. Suppose moreover that $q_\bfm\in (q-\ve,q+\ve)$ and
		$\dist\!\bigl(\zeta_\bfm,\,\zeta|_{(q,1]}+\zeta([0,q])\,\delta_q\bigr)<\ve$.
		\begin{enumerate}[label=(\roman*)]
			\item \label{eq:prob0 pure density} The gradient $\nabla F_{\TAP,\zeta}(\bfm)$ is a non-degenerate Gaussian vector in $\dR^N$. Its log-density at $0$, tilted by
			\[
			u\!\left(\frac1p\langle \bfm,k_\zeta(q,\bfm)\rangle-\sum_{i=1}^N h_\zeta(q,m_i)-N\mc U_\zeta(q)-Nf\right),
			\]
			satisfies
			\begin{multline}\label{eq:pure-prob0-main}
				\log \varphi_{\nabla F_{\TAP,\zeta}(\bfm)}(0)
				+u\!\left(\frac1p\langle \bfm,k_\zeta(q,\bfm)\rangle-\sum_{i=1}^N h_\zeta(q,m_i)-N\mc U_\zeta(q)-Nf\right)
				\\
				=
				\frac{1}{2}N\xi(q)u^2
				-\frac{1}{2}N\xi''(q)\!\left(\int_q^1 \zeta([0,t])\,\dd t\right)^{\!2}
				-\frac{1}{2\xi'(q)}
				\Vert \partial_m h_\zeta(q,\bfm)-u\xi'(q)\bfm\Vert^2
				\\
				-\frac{N}{2}\log(2\pi \xi'(q))
				+\frac{N(p-1)}{2p\,\xi'(q)\,q}
				\!\left(
				\frac{1}{N}\langle \partial_m h_\zeta(q,\bfm),\bfm\rangle
				-\xi'(q)\int_0^1 \zeta([0,t])\,\dd t
				\right)^{\!2}
				\\
				-u\!\left(\sum_{i=1}^N h_\zeta(q,m_i)+N\mc U_\zeta(q)+Nf\right)
				+O\!\left(N\frac{\ve}{\delta}\right).
			\end{multline}
			
			\item \label{eq:prob0 pure simplified} Suppose moreover that
			\begin{equation}\label{eq:pure-ass1}
				\langle \partial_m h_\zeta(q,\bfm),\bfm\rangle
				=
				N\xi'(q)\int_0^1 \zeta([0,t])\,\dd t+O(N\ve)
			\end{equation}
			and
			\begin{equation}\label{eq:pure-ass2}
				\sum_{i=1}^N\bigl(h_\zeta(q,m_i)-\partial_m h_\zeta(q,m_i)m_i\bigr)
				+N\int_0^q t\xi''(t)\zeta([0,t])\,\dd t
				+\frac{N}{2}\int_q^1 t\xi''(t)\zeta([0,t])\,\dd t
				=
				-Nf+O(N\ve).
			\end{equation}
			Then the squared term in \eqref{eq:pure-prob0-main} is $O(N\ve)$, so that 
			\begin{multline}\label{eq:pure-prob0-simplified}
				\log \varphi_{\nabla F_{\TAP,\zeta}(\bfm)}(0)
				+u\!\left(\frac1p\langle \bfm,k_\zeta(q,\bfm)\rangle-\sum_{i=1}^N h_\zeta(q,m_i)-N\mc U_\zeta(q)-Nf\right)
				\\
				=
				\frac{1}{2}N\xi(q)u^2
				-\frac{1}{2}N\xi''(q)\!\left(\int_q^1 \zeta([0,t])\,\dd t\right)^{\!2}
				-\frac{1}{2\xi'(q)}
				\Vert \partial_m h_\zeta(q,\bfm)-u\xi'(q)\bfm\Vert^2
				\\
				-\frac{N}{2}\log(2\pi \xi'(q))
				-u\!\left(\sum_{i=1}^N h_\zeta(q,m_i)+N\mc U_\zeta(q)+Nf\right)
				+O\!\left(N\frac{\ve}{\delta}\right),
			\end{multline}
			and, on the event $\{\nabla F_{\TAP,\zeta}(\bfm)=0\}$,
			\begin{equation}\label{eq:R-equals-f-pure}
				F_{\TAP,\zeta}(\bfm)=Nf+O\!\left(N\frac{\ve}{\delta}\right).
			\end{equation}
		\end{enumerate}
	\end{lemma}
	
	\begin{proof}
		\smallskip\noindent\textbf{Step 1: reduction to the exact remainder.}
		By Lemma~\ref{lemma:reduction pure} and \eqref{eq:def-Rex-KR},
		\[
		F_{\TAP,\zeta}(\bfm)
		=
		\frac1p\langle \bfm,\nabla F_{\TAP,\zeta}(\bfm)\rangle
		+\mc R^{\rm ex}_\zeta(\bfm),
		\qquad
		\mc R^{\rm ex}_\zeta(\bfm)
		=
		\frac1p\langle \bfm,\nabla \mc S_\zeta(\bfm)\rangle
		-\sum_{i=1}^N h_\zeta(q_\bfm,m_i)-N\mc U_\zeta(q_\bfm).
		\]
		Also, by definition,
		\[
		\mc S_\zeta(\bfm)=\sum_{i=1}^N h_\zeta(q_\bfm,m_i)+N\mc U_\zeta(q_\bfm).
		\]
		
		Since the distribution functions of
		\[
		\zeta|_{(q,1]}+\zeta([0,q])\delta_q
		\qquad\text{and}\qquad
		\zeta|_{(q_\bfm,1]}+\zeta([0,q_\bfm])\delta_{q_\bfm}
		\]
		differ only on the interval between $q$ and $q_\bfm$, we have
		\[
		\dist\!\Bigl(
		\zeta|_{(q,1]}+\zeta([0,q])\delta_q,\,
		\zeta|_{(q_\bfm,1]}+\zeta([0,q_\bfm])\delta_{q_\bfm}
		\Bigr)
		\le |q_\bfm-q|=O(\ve).
		\]
		Hence
		\[
		\dist\!\Bigl(
		\zeta_\bfm,\,
		\zeta|_{(q_\bfm,1]}+\zeta([0,q_\bfm])\delta_{q_\bfm}
		\Bigr)
		=O(\ve).
		\]
		Applying Lemma~\ref{lemma:gradient TAP}, \eqref{eq:gradTap-stability}, to the
		measure $\zeta|_{(q_\bfm,1]}+\zeta([0,q_\bfm])\delta_{q_\bfm}$, and using that
		this measure agrees with $\zeta$ above $q_\bfm$, we get
		\[
		\nabla \mc S_\zeta(\bfm)=k_\zeta(q_\bfm,\bfm)+O(\ve/\delta),
		\]
		where the error is understood componentwise. Since $|m_i|\le 1$ for every $i$,
		this yields
		\begin{equation}\label{eq:F-R-pure}
			F_{\TAP,\zeta}(\bfm)
			=
			\frac1p\langle \bfm,\nabla F_{\TAP,\zeta}(\bfm)\rangle
			+\frac1p\langle \bfm,k_\zeta(q_\bfm,\bfm)\rangle
			-\sum_{i=1}^N h_\zeta(q_\bfm,m_i)-N\mc U_\zeta(q_\bfm)
			+O\!\left(N\frac{\ve}{\delta}\right).
		\end{equation}
		In particular, on the event $\{\nabla F_{\TAP,\zeta}(\bfm)=0\}$,
		\begin{equation}\label{eq:F-R-pure-critical}
			F_{\TAP,\zeta}(\bfm)
			=
			\frac1p\langle \bfm,k_\zeta(q_\bfm,\bfm)\rangle
			-\sum_{i=1}^N h_\zeta(q_\bfm,m_i)-N\mc U_\zeta(q_\bfm)
			+O\!\left(N\frac{\ve}{\delta}\right).
		\end{equation}
		
		\smallskip\noindent\textbf{Step 2: density of the gradient.}
		Set
		\[
		G_\bfm:=\Cov(\nabla H(\bfm))
		=
		\xi'(q_\bfm)\Id+\frac{\xi''(q_\bfm)}{N}\bfm\bfm^\top.
		\]
		Since $\|\bfm\|^2=Nq_\bfm$ and $\xi$ is pure $p$-spin, the identities
		$q_\bfm\xi''(q_\bfm)=(p-1)\xi'(q_\bfm)$ and
		$p\,\xi(q_\bfm)=q_\bfm\,\xi'(q_\bfm)$ hold. Therefore $G_\bfm$ has
		eigenvalue $\xi'(q_\bfm)$ on $\bfm^\perp$ and eigenvalue
		$p\,\xi'(q_\bfm)$ on $\dR\bfm$, so that
		\begin{equation}\label{eq:Ginv-pure}
			G_\bfm^{-1}
			=
			\frac{1}{\xi'(q_\bfm)}\Id
			-\frac{p-1}{p\,\xi'(q_\bfm)\,q_\bfm}\,\frac{\bfm\bfm^\top}{N},
			\qquad
			\log\det(2\pi G_\bfm)=N\log(2\pi \xi'(q_\bfm))+O(1).
		\end{equation}
		
		By Step~1, $\nabla F_{\TAP,\zeta}(\bfm)$ is Gaussian with covariance $G_\bfm$
		and mean $-k_\zeta(q_\bfm,\bfm)+e$ where $\|e\|_\infty\le C\ve/\delta$.
		Since $\|G_\bfm^{-1}\|_{\mathrm{op}}\le 1/\xi'(q_\bfm)$ and $\|k_\zeta(q_\bfm,\bfm)\|^2=O(N)$, the cross-term
		$\langle k_\zeta(q_\bfm,\bfm),G_\bfm^{-1}e\rangle$ is $O(N\ve/\delta)$. Hence
		\begin{equation}\label{eq:density-pure}
			\log \varphi_{\nabla F_{\TAP,\zeta}(\bfm)}(0)
			=
			-\frac12\langle k_\zeta(q_\bfm,\bfm),G_\bfm^{-1}k_\zeta(q_\bfm,\bfm)\rangle
			-\frac12\log\det(2\pi G_\bfm)
			+O\!\left(N\frac{\ve}{\delta}\right).
		\end{equation}
		
		\smallskip\noindent\textbf{Step 3: completing the square.}
		By \eqref{eq:F-R-pure-critical}, the tilt differs from
		$\frac{u}{p}\langle \bfm,k_\zeta(q_\bfm,\bfm)\rangle$ only by the deterministic
		term
		\[
		-u\!\left(\sum_{i=1}^N h_\zeta(q_\bfm,m_i)+N\mc U_\zeta(q_\bfm)+Nf\right).
		\]
		Thus we must compute
		\[
		-\frac12\langle k_\zeta(q_\bfm,\bfm),G_\bfm^{-1}k_\zeta(q_\bfm,\bfm)\rangle
		+\frac{u}{p}\langle \bfm,k_\zeta(q_\bfm,\bfm)\rangle.
		\]
		Recalling that
		\[
		k_\zeta(q_\bfm,\bfm)
		=
		\partial_m h_\zeta(q_\bfm,\bfm)+\xi''(q_\bfm)\!\left(\int_{q_\bfm}^1 \zeta([0,t])\,\dd t\right)\bfm
		\]
		and using \eqref{eq:Ginv-pure}, together with
		$q_\bfm\xi''(q_\bfm)=(p-1)\xi'(q_\bfm)$ and
		$p\,\xi(q_\bfm)=q_\bfm\,\xi'(q_\bfm)$, a direct computation gives
		\begin{multline}\label{eq:complete-square-pure}
			-\frac12\langle k_\zeta(q_\bfm,\bfm),G_\bfm^{-1}k_\zeta(q_\bfm,\bfm)\rangle
			+\frac{u}{p}\langle \bfm,k_\zeta(q_\bfm,\bfm)\rangle
			\\
			=
			\frac{1}{2}N\xi(q_\bfm)u^2
			-\frac{1}{2}N\xi''(q_\bfm)\!\left(\int_{q_\bfm}^1 \zeta([0,t])\,\dd t\right)^{\!2}
			-\frac{1}{2\xi'(q_\bfm)}
			\Vert \partial_m h_\zeta(q_\bfm,\bfm)-u\xi'(q_\bfm)\bfm\Vert^2
			\\
			+\frac{N(p-1)}{2p\,\xi'(q_\bfm)\,q_\bfm}
			\left(
			\frac{1}{N}\langle \partial_m h_\zeta(q_\bfm,\bfm),\bfm\rangle
			-\xi'(q_\bfm)\!\left(uq_\bfm+\int_{q_\bfm}^1 \zeta([0,t])\,\dd t\right)
			\right)^{\!2}.
		\end{multline}
		Substituting \eqref{eq:complete-square-pure} into \eqref{eq:density-pure}, and then replacing $q_\bfm$ by $q$ at cost $O(N\ve/\delta)$, proves \eqref{eq:pure-prob0-main}.
		
		\smallskip\noindent\textbf{Step 4: critical value.}
		From \eqref{eq:F-R-pure-critical}, at a critical point we have
		\[
		F_{\TAP,\zeta}(\bfm)
		=
		\frac{1}{p}\langle \bfm,k_\zeta(q_\bfm,\bfm)\rangle
		-\sum_{i=1}^N h_\zeta(q_\bfm,m_i)-N\mc U_\zeta(q_\bfm)
		+O\!\left(N\frac{\ve}{\delta}\right).
		\]
		Using
		\[
		k_\zeta(q_\bfm,\bfm)
		=
		\partial_m h_\zeta(q_\bfm,\bfm)+\xi''(q_\bfm)\!\left(\int_{q_\bfm}^1 \zeta([0,t])\,\dd t\right)\bfm
		\]
		and $q_\bfm\xi''(q_\bfm)=(p-1)\xi'(q_\bfm)$, we get
		\begin{multline}\label{eq:R-rewrite-pure}
			F_{\TAP,\zeta}(\bfm)
			=
			-\sum_{i=1}^N\bigl(h_\zeta(q_\bfm,m_i)-\partial_m h_\zeta(q_\bfm,m_i)m_i\bigr)
			\\
			+\frac{p-1}{p}\!\left(
			N\xi'(q_\bfm)\int_{q_\bfm}^1 \zeta([0,t])\,\dd t
			-
			\langle \partial_m h_\zeta(q_\bfm,\bfm),\bfm\rangle
			\right)
			-N\mc U_\zeta(q_\bfm)
			+O\!\left(N\frac{\ve}{\delta}\right).
		\end{multline}
		Using $|q_\bfm-q|\leq\ve$, the smoothness of $q\mapsto \xi(q),\xi'(q),\xi''(q),\mc U_\zeta(q)$,
		and the fact that $m\mapsto h_\zeta(q,m)$ and $m\mapsto \partial_m h_\zeta(q,m)$ vary by $O(\ve/\delta)$
		when $q_\bfm$ is replaced by $q$ (uniformly for $|m|\le 1-\delta$), the assumptions \eqref{eq:pure-ass1}--\eqref{eq:pure-ass2}
		imply the corresponding $q_\bfm$-versions up to $O(N\ve/\delta)$. Therefore the square term in
		\eqref{eq:complete-square-pure} is $O(N\ve/\delta)$, which gives \eqref{eq:pure-prob0-simplified};
		inserting the same replacements into \eqref{eq:R-rewrite-pure} yields
		\eqref{eq:R-equals-f-pure} with error $O(N\ve/\delta)$.
	\end{proof}

	\subsection{Determinant asymptotics}

	We compute the expected value of the determinant appearing in the Kac--Rice formula \eqref{eq:KacRice}. By definition, $\zeta_\bfm$ minimizes $\TAP(\mu_N^{(\bfm)},\cdot)$, and it has an atom at $q_\bfm$. Consequently, $q_\bfm$ is an optimizer of a certain auxiliary function. We show that the first- and second-order optimality conditions for this nonconvex problem then determine the value of the subordination function, and hence that of the determinant of the free convolution.

	\begin{lemma}\label{lemma:det asymp}
		Let $\zeta$ be a probability measure on $[0,1]$ with prefix $\{0,q\}$ with $q\in (0,1)$. Let $f\in \dR$. There exists $C_\ve>0$, depending only on $\ve$, $f$, and $\xi$, with $C_\ve\to 0$ as $\ve\to 0$, such that the following holds uniformly over $\bfm\in \SUSY_1(\zeta,\ve)$: Letting $\mc{B}:=\{\nabla F_{\TAP,\zeta}(\bfm)=0,\, F_{\TAP,\zeta}(\bfm)=Nf\}$,
		\begin{equation*}
			\left| \log \dE\bigl[ |\det\nabla^2 F_{\TAP,\zeta}(\bfm)|\mid\mc{B}\bigr]
			- \log\prod_{i=1}^N\partial_{mm}h_\zeta(q,m_i)
			- \frac{N}{2}\xi''(q)\left(\int_q^1 \zeta([0,t])\,\dd t\right)^{\!2}
			\right|
			\leq C_\ve N.
		\end{equation*}
	\end{lemma}
	
	\medskip
	
	\begin{proof}
		Let $q_\bfm=\frac{1}{N}\Vert \bfm\Vert^2\in (q-\ve,q+\ve)$. For a probability measure $\zeta'$ on $[q',1]$, set 
		\begin{equation*}
			T_{\zeta'}:m\in [-1,1]\mapsto \partial_{mm} h_{\zeta'}(q',m)+\xi''(q')\int_{q'}^1 \zeta'([0,t])\dd t.
		\end{equation*}
		By Lemma \ref{lemma:det app}, 
		\begin{multline}\label{eq:limit free}
			\log \dE[ |\det(\nabla^2 F_{\TAP,\zeta}(\bfm))|\mid \nabla F_{\TAP,\zeta}(\bfm)=0,F_{\TAP,\zeta}(\bfm)=Nf]\\= N \int \log|x|\dd (T_\zeta\#\mu_N^{(\bfm)}\boxplus\sigma_{\xi''(q)})(x)+o(N),
		\end{multline}
		with $o(N)$ uniform. Since $\dist(\zeta_\bfm,\zeta|_{(q,1]}+\zeta([0,q])\,\delta_q)<\ve$, $|q-q_{\bfm}|<\ve$ and $\bfm\in [-1+\delta_\ve,1-\delta_\ve]^N$, the maps $T_\zeta$ and $T_{\zeta_\bfm}$ are uniformly close on $[-1+\delta_\ve,1-\delta_\ve]$, so the two pushed-forward measures are close in the Wasserstein sense. Moreover, by the second-order optimality condition \eqref{eq:second order}, the support of $T_{\zeta_\bfm}\#\mu_N^{(\bfm)}\boxplus\sigma_{\xi''(q_\bfm)}$ stays uniformly bounded away from $0$, which ensures that $\log|x|$ is Lipschitz on the support. Hence there exists $C_\ve>0$ with $C_\ve\to 0$ as $\ve\to 0$ such that
		\begin{equation}\label{eq:compare5}
			\Bigl|\int \log|x|\dd (T_\zeta\#\mu_N^{(\bfm)}\boxplus\sigma_{\xi''(q)})(x)-\int \log|x|\dd (T_{\zeta_{\bfm}}\#\mu_N^{(\bfm)}\boxplus\sigma_{\xi''(q_\bfm)})(x)\Bigr|\leq C_\ve.
		\end{equation}
		Moreover, by Lemma \ref{lemma:Hopf Lax} and \eqref{eq:ch},
		\begin{equation*}
			\int \log|x|\,\dd (T_{\zeta_{\bfm}}\#\mu_N^{(\bfm)}\boxplus\sigma_{\xi''(q_\bfm)})(x)
			= \int\log|T_{\zeta_{\bfm}}(x)-\omega|\,\dd \mu_N^{(\bfm)}(x)
			+\frac{\Re(\omega)^2-\Im(\omega)^2}{2\xi''(q_\bfm)},
		\end{equation*}
		where $\omega:=\omega_{T_{\zeta_{\bfm}}\#\mu_N^{(\bfm)},\xi''(q_\bfm)}(0)$ is the subordination function of $T_{\zeta_{\bfm}}\#\mu_N^{(\bfm)}$ at time $\xi''(q_\bfm)$ and location $0$. Set
		\begin{equation*}
			\omega':=\xi''(q_\bfm)\int_{q_\bfm}^1 \zeta_{\bfm}([0,t])\dd t.
		\end{equation*}
		We show below that $\omega=\omega'$. Since $\omega'$ is real, this will imply that $\omega\in \dR$, and therefore the previous identity simplifies to
		\begin{equation}\label{eq:logNu real}
			\int \log|x|\,\dd (T_{\zeta_{\bfm}}\#\mu_N^{(\bfm)}\boxplus\sigma_{\xi''(q_\bfm)})(x)
			= \int\log|T_{\zeta_{\bfm}}(x)-\omega'|\,\dd \mu_N^{(\bfm)}(x)
			+\frac{(\omega')^2}{2\xi''(q_\bfm)}.
		\end{equation}

		First,
		\begin{equation}\label{eq:Stiel}
			G_{T_{\zeta_{\bfm}}\#\mu_N^{(\bfm)}}(\omega')=-\frac{1}{N}\sum_{i=1}^N \frac{1}{\partial_{mm} h_{\zeta_{\bfm}}(q_\bfm,m_i)}.
		\end{equation}
		Set
		\begin{equation*}
			\nu_N^{(\bfm)}:=(\partial_x\Phi_{\zeta_{\bfm}}(q_\bfm,\cdot))^{-1}\#\mu_N^{(\bfm)}.
		\end{equation*}
		Since $h_{\zeta_{\bfm}}(q_\bfm,\cdot)=\Phi_{\zeta_{\bfm}}(q_\bfm,\cdot)^*$, we have 
		\begin{equation}\label{eq:chain}
			\frac{1}{N}\sum_{i=1}^N \frac{1}{\partial_{mm} h_{\zeta_{\bfm}}(q_\bfm,m_i)}=\dE_{\mu_N^{(\bfm)}}\left[\frac{1}{\partial_{mm} h_{\zeta_{\bfm}}(q_\bfm,\cdot)}\right]=\dE_{\nu_N^{(\bfm)} }[\partial_{xx}\Phi_{\zeta_{\bfm}}(q_\bfm,\cdot)].
		\end{equation}
		By Lemma \ref{lemma:xxphi} (since $\zeta_\bfm\in \mc{P}([q_{\bfm},1])$), we have 
		\begin{equation}\label{eq:exprr}
			\partial_{xx}\Phi_{\zeta_\bfm}(q_\bfm,x)=1-\int_{q_\bfm}^1 \dE[(\partial_x\Phi_{\zeta_\bfm}(s,X_s))^2\mid X_{q_\bfm}=x]\dd \zeta_\bfm(s).
		\end{equation}
		Consider the Auffinger--Chen process
		\begin{equation*}
			\begin{cases}    \dd X_t=\xi''(t)\zeta_\bfm([0,t])\partial_x \Phi_{\zeta_\bfm}(t,X_t)\dd t+\sqrt{\xi''(t)} \dd B_t\\
				\mathrm{Law}(\partial_x\Phi_{\zeta_\bfm}(q_\bfm,X_{q_\bfm}))=\mu_N^{(\bfm)}.
			\end{cases}
		\end{equation*}
		By \eqref{eq:exprr},
		\begin{equation}\label{eq:muNs}
			\dE_{\nu_N^{(\bfm)} }[\partial_{xx}\Phi_{\zeta_\bfm}(q_\bfm,\cdot)]=1-\int_{q_\bfm}^1 \dE\left[\left(\partial_x\Phi_{\zeta_\bfm}(s,X_s)\right)^2\right]\dd \zeta_\bfm(s).
		\end{equation}
		By the optimality of $\zeta_\bfm$, we get from Lemma \ref{lemma:first order mu} (see \eqref{eq:opt3}) that for every $s\in \supp(\zeta_\bfm)$, one has 
		\begin{equation*}
			\dE\left[\left(\partial_x\Phi_{\zeta_\bfm}\left(s,X_s\right)\right)^2\right]=s.
		\end{equation*}
		Hence, inserting this into \eqref{eq:muNs} and using Stieltjes integration by parts, we get
		\begin{equation*}
			\dE_{\nu_N^{(\bfm)} }[\partial_{xx}\Phi_{\zeta_\bfm}(q_\bfm,\cdot)]=\int_{q_\bfm}^1 \zeta_\bfm([0,t])\dd t.
		\end{equation*}
		Consequently, by \eqref{eq:chain} and \eqref{eq:Stiel}, 
		\begin{equation*}
			G_{T_{\zeta_{\bfm}}\#\mu_N^{(\bfm)}}(\omega')=-\frac{\omega'}{\xi''(q_\bfm)},   
		\end{equation*}
		which gives 
		\begin{equation*}
			\omega'+\xi''(q_\bfm) G_{T_{\zeta_{\bfm}}\#\mu_N^{(\bfm)}}(\omega')=0.
		\end{equation*}
		By Lemma \ref{lemma:free co}, it remains to prove (recall \eqref{def:Omegamut}) that 
		\begin{equation}\label{eq:omega'O}
			\omega'\in \overline{\Omega_{T_{\zeta_{\bfm}}\#\mu_N^{(\bfm)},\xi''(q_\bfm)}},
		\end{equation}
		or equivalently that
		\begin{equation*}
			\frac{1}{N}\sum_{i=1}^N \frac{1}{(T_{\zeta_{\bfm}}(m_i)-\omega')^2}\leq \frac{1}{\xi''(q_\bfm)}.
		\end{equation*}
		Notice that 
		\begin{equation*}
			\frac{1}{N}\sum_{i=1}^N \frac{1}{(T_{\zeta_{\bfm}}(m_i)-\omega')^2}=\frac{1}{N}\sum_{i=1}^N \frac{1}{(\partial_{mm}h_{\zeta_{\bfm}}(q_\bfm,m_i))^2}.
		\end{equation*}
		By the optimality of $\zeta_\bfm$, we have by Lemma \ref{lemma:first order mu} (see \eqref{eq:second order}) that 
		\begin{equation*}
			\frac{1}{N}\sum_{i=1}^N \frac{1}{(\partial_{mm}h_{\zeta_{\bfm}}(q_\bfm,m_i))^2}\leq \frac{1}{\xi''(q_\bfm)},  
		\end{equation*}
		which proves \eqref{eq:omega'O}. Thus, we have $\omega=\omega'$. Substituting into \eqref{eq:logNu real} gives
		\begin{multline}\label{eq:logNu-expand}
			\int \log|x|\,\dd (T_{\zeta_{\bfm}}\#\mu_N^{(\bfm)}\boxplus\sigma_{\xi''(q_\bfm)})(x)
			\\= \frac{1}{N}\sum_{i=1}^N\log\partial_{mm}h_{\zeta_\bfm}(q_\bfm,m_i)
			+\frac{(\omega')^2}{2\xi''(q_\bfm)}
			= \frac{1}{N}\sum_{i=1}^N\log\partial_{mm}h_{\zeta_\bfm}(q_\bfm,m_i)
			+\frac{\xi''(q_\bfm)}{2}\!\left(\int_{q_\bfm}^1 \zeta_\bfm([0,t])\,\dd t\right)^{\!2},
		\end{multline}
		where we used $T_{\zeta_\bfm}(m)-\omega'=\partial_{mm}h_{\zeta_\bfm}(q_\bfm,m)$.
		Since $|q_\bfm-q|\leq\ve$, $\dist(\zeta_\bfm,\zeta|_{(q,1]}+\zeta([0,q])\delta_q)\leq\ve$, and $|m_i|\leq 1-\delta_\ve$, replacing $(q_\bfm,\zeta_\bfm)$ by $(q,\zeta)$ in \eqref{eq:logNu-expand} introduces an error of at most $C_\ve$ per coordinate (uniformly in $m_i\in[-1+\delta_\ve,1-\delta_\ve]$). Combining with \eqref{eq:compare5} and \eqref{eq:limit free}, we obtain
		\begin{multline*}
			\log \dE[ |\det\nabla^2 F_{\TAP,\zeta}(\bfm)|\mid \nabla F_{\TAP,\zeta}(\bfm)=0,F_{\TAP,\zeta}(\bfm)=Nf]\\ = \log\prod_{i=1}^N\partial_{mm}h_\zeta(q,m_i)+{\frac{N}{2}\xi''(q)\left(\int_q^1 \zeta([0,t])\dd t\right)^2 }+O_\ve(N),
		\end{multline*}
		with $|O_\ve(N)|\leq C_\ve N$ for some $C_\ve>0$ such that $C_\ve\to 0$ as $\ve\to 0$.
	\end{proof}

	\subsection{Upper bound on the complexity}
	
	We now prove the annealed upper bound on the complexity.

	\begin{lemma}\label{lemma:upper bound 0}
		Let $q\in (0,1)$ and $u\in (0,1)$. Let $\zeta\in \Prefix_2(u,q)$. Let $f\in \dR$. Let $\ve\in (0,1)$ and let $\mc{N}_\ve(f,\zeta,q)$ be the number of critical points $\bfm\in [-1,1]^N$ of $F_{\TAP,\zeta}$ in $\SUSY_1(\zeta,\ve)$ such that $\frac{1}{N}F_{\TAP,\zeta}(\bfm)\in (f-\ve,f+\ve)$.
		
		Then, 
		\begin{equation}\label{eq:logNu}
			\limsup_{\ve\to 0} \limsup_{N\to \infty}\frac{1}{N}\log \dE[\mc{N}_\ve(f,\zeta,q)]\leq u\left(\Phi_\zeta(0,0)-\frac{1}{2}\int_0^1 t\xi''(t)\zeta([0,t])\dd t-f\right).
		\end{equation}
	\end{lemma}

	\medskip

	\begin{proof}
		We write $C_\ve$ for a positive constant satisfying $C_\ve\to 0$ as $\ve\to 0$, 
		and $O_\ve(N)$ for a quantity bounded in absolute value by $C_\ve N$.
		
		We first derive the analogue of \eqref{eq:upperN} in the mixed and pure cases.
		
		\smallskip
		\noindent{\bf Mixed case.}
		Recall from Lemma \ref{lemma:KacRice} that 
		\begin{multline*}
			\dE[\mc{N}_\ve(f,\zeta,q)]=N\int_{[-1,1]^N}\int_{(f-\ve,f+\ve)} \dE[ |\det(\nabla^2 F_{\TAP,\zeta}(\bfm))|\mid \nabla F_{\TAP,\zeta}(\bfm)=0,F_{\TAP,\zeta}(\bfm)=Nf'] \\ \times \varphi_{(\nabla F_{\TAP,\zeta}(\bfm),F_{\TAP,\zeta}(\bfm))}(0,Nf')\indic_{\SUSY_1(\zeta,\ve)}(\bfm) \dd \bfm \dd f'.
		\end{multline*} 
		By Lemma \ref{lemma:det asymp}, for every $\bfm\in [-1,1]^N$ such that $q_\bfm=\frac{1}{N}\Vert \bfm\Vert^2\in (q-\ve,q+\ve)$ and $\dist(\zeta_\bfm,\zeta|_{(q,1]}+\zeta([0,q])\delta_q)<\ve$,
		\begin{multline*}
			\log \dE[ |\det\nabla^2 F_{\TAP,\zeta}(\bfm)|\mid \nabla F_{\TAP,\zeta}(\bfm)=0,F_{\TAP,\zeta}(\bfm)=Nf']\\ \leq \log\prod_{i=1}^N\partial_{mm}h_\zeta(q,m_i)+{\frac{N}{2}\xi''(q)\left(\int_q^1 \zeta([0,t])\dd t\right)^2 }+C_\ve N.
		\end{multline*}
		Moreover, by Lemma \ref{lemma:prob0 non deg} (Equation \eqref{eq:upper bound proba}), if in addition $\bfm\in \SUSY_1(\zeta,\ve)$,
		\begin{multline*}
			\log \varphi_{(\nabla F_{\TAP,\zeta}(\bfm),F_{\TAP,\zeta}(\bfm))}(0,Nf')\\ \leq \frac{1}{2}N\xi(q)u^2-\frac{1}{2}N\xi''(q)\left(\int_q^1 \zeta([0,t])\dd t \right)^2 -\frac{1}{2\xi'(q)}\Vert \partial_m h_\zeta(q,\bfm)-u\xi'(q)\bfm\Vert^2-\frac{N}{2}\log (2\pi \xi'(q))\\
			-u \left(\sum_{i=1}^N h_\zeta(q,m_i)+N\mc{U}_\zeta(q)+Nf'\right)+N\frac{\ve^2}{2q\xi'(q)}+O_\ve(N).
		\end{multline*}
		Combining the two above displays, noticing that the two terms $\xi''(q)(\int_q^1 \zeta([0,t])\dd t)^2$ cancel each other, and using that $f'\in (f-\ve,f+\ve)$, we deduce that
		\begin{multline}\label{eq:upperN}
			\log \dE[\mc{N}_\ve(f,\zeta,q)]\leq  \frac{1}{2}N\xi(q)u^2 -uN\mc{U}_\zeta(q)-Nu f-\frac{N}{2}\log(2\pi \xi'(q))\\+ \log \int_{[-1,1]^N} \prod_{i=1}^N \partial_{mm}h_\zeta(q,m_i)\,e^{-u \sum_{i=1}^N h_\zeta(q,m_i)-\frac{1}{2\xi'(q)}\Vert \partial_m h_\zeta(q,\bfm)-u\xi'(q)\bfm\Vert^2}\\
			\times\indic_{\frac{1}{N}\Vert \bfm\Vert^2\in (q-\ve,q+\ve)} \dd \bfm+O_\ve(N).
		\end{multline}
		
		\smallskip
		\noindent{\bf Pure case.}
		By Lemma~\ref{lemma:KacRice}, \eqref{eq:KacRice-pure},
		\begin{multline*}
			\dE[\mc{N}_\ve(f,\zeta,q)]
			=
			\int_{[-1,1]^N}
			\dE\!\left[
			|\det(\nabla^2 F_{\TAP,\zeta}(\bfm))|
			\,\middle|\,
			\nabla F_{\TAP,\zeta}(\bfm)=0
			\right]
			\varphi_{\nabla F_{\TAP,\zeta}(\bfm)}(0)
			\\
			\times
			\indic_{\SUSY_1(\zeta,\ve)}(\bfm)\,
			\indic_{(f-\ve,f+\ve)}
			\!\left(\frac1N\mc R^{\rm ex}_\zeta(\bfm)\right)
			\dd\bfm.
		\end{multline*}
		Fix $\bfm\in \SUSY_1(\zeta,\ve)$ and set
		\[
		f_\bfm:=\frac1N\mc R^{\rm ex}_\zeta(\bfm),
		\qquad
		T(\bfm):=\frac1p\langle \bfm,k_\zeta(q,\bfm)\rangle-\sum_{i=1}^N h_\zeta(q,m_i)-N\mc U_\zeta(q).
		\]
		By \eqref{eq:pure-critical-identity-KR}, on the event
		$\{\nabla F_{\TAP,\zeta}(\bfm)=0\}$ one has
		$F_{\TAP,\zeta}(\bfm)=Nf_\bfm$. Therefore
		\begin{multline*}
			\dE\!\left[
			|\det(\nabla^2 F_{\TAP,\zeta}(\bfm))|
			\,\middle|\,
			\nabla F_{\TAP,\zeta}(\bfm)=0
			\right]
			\\=
			\dE\!\left[
			|\det(\nabla^2 F_{\TAP,\zeta}(\bfm))|
			\,\middle|\,
			\nabla F_{\TAP,\zeta}(\bfm)=0,\,
			F_{\TAP,\zeta}(\bfm)=Nf_\bfm
			\right].
		\end{multline*}
		Hence, on the support of the indicator
		$\indic_{(f-\ve,f+\ve)}(N^{-1}\mc R^{\rm ex}_\zeta(\bfm))$, Lemma~\ref{lemma:det asymp} yields
		\begin{multline*}
			\log \dE\!\left[
			|\det(\nabla^2 F_{\TAP,\zeta}(\bfm))|
			\,\middle|\,
			\nabla F_{\TAP,\zeta}(\bfm)=0
			\right]\\
			\leq
			\log\prod_{i=1}^N\partial_{mm}h_\zeta(q,m_i)
			+\frac{N}{2}\xi''(q)\left(\int_q^1 \zeta([0,t])\dd t\right)^2
			+O_\ve(N).
		\end{multline*}
		
		Comparing the exact identity
		\[
		F_{\TAP,\zeta}(\bfm)
		=
		\frac1p\langle \bfm,\nabla F_{\TAP,\zeta}(\bfm)\rangle
		+\mc R^{\rm ex}_\zeta(\bfm)
		\]
		with \eqref{eq:F-R-pure}, we obtain
		\[
		\mc R^{\rm ex}_\zeta(\bfm)=T(\bfm)+O_\ve(N)
		\]
		uniformly over $\bfm\in \SUSY_1(\zeta,\ve)$. Hence, on the support of
		$\indic_{(f-\ve,f+\ve)}(N^{-1}\mc R^{\rm ex}_\zeta(\bfm))$,
		\[
		u\bigl(T(\bfm)-Nf\bigr)=O_\ve(N).
		\]
		Therefore, by Lemma~\ref{lemma:prob0 pure}, \eqref{eq:pure-prob0-main}, and the bound
		$\frac{p-1}{p}\le 1$, we get
		\begin{multline*}
			\log \varphi_{\nabla F_{\TAP,\zeta}(\bfm)}(0)\\
			\leq
			\frac{1}{2}N\xi(q)u^2
			-\frac{1}{2}N\xi''(q)\left(\int_q^1 \zeta([0,t])\dd t \right)^2
			-\frac{1}{2\xi'(q)}\Vert \partial_m h_\zeta(q,\bfm)-u\xi'(q)\bfm\Vert^2
			-\frac{N}{2}\log (2\pi \xi'(q))
			\\
			+\frac{N}{2\xi'(q)q}\left(
			\frac{1}{N}\langle \partial_m h_\zeta(q,\bfm),\bfm\rangle
			-\xi'(q)\int_0^1 \zeta([0,t])\dd t
			\right)^2
			-u \left(\sum_{i=1}^N h_\zeta(q,m_i)+N\mc{U}_\zeta(q)+Nf\right)
			+O_\ve(N).
		\end{multline*}
		Combining the last two displays and dropping the indicator
		$\indic_{(f-\ve,f+\ve)}(N^{-1}\mc R^{\rm ex}_\zeta(\bfm))$, we obtain again
		\eqref{eq:upperN}.
		
		Notice that
		\begin{equation*}
			\Vert \partial_m h_\zeta(q,\bfm)-u\xi'(q)\bfm\Vert^2=\Vert \partial_m h_\zeta(q,\bfm)\Vert^2-2u \xi'(q)\langle \bfm,\partial_m h_\zeta(q,\bfm)\rangle+\xi'(q)^2 u^2\Vert \bfm\Vert^2.
		\end{equation*}
		Hence, plugging this into \eqref{eq:upperN} and dropping the constraint that $\frac{1}{N}\Vert \bfm\Vert^2\in (q-\ve,q+\ve)$, we deduce that
		\begin{multline}\label{eq:boundN}
			\log \dE[\mc{N}_\ve(f,\zeta,q)]\leq  \frac{1}{2}N\xi(q)u^2 -uN\mc{U}_\zeta(q)-Nu f-\frac{N}{2}\log(2\pi \xi'(q))-\frac{1}{2}N\xi'(q)qu^2\\
			+\log \int_{[-1,1]^N} \prod_{i=1}^N \partial_{mm}h_\zeta(q,m_i)e^{-u\sum_{i=1}^N (h_\zeta(q,m_i)-\partial_m h_\zeta(q,m_i)m_i)-\frac{1}{2\xi'(q)}\Vert \partial_m h_\zeta(q,\bfm)\Vert^2}\dd \bfm+O_\ve(N).
		\end{multline}
		Let $g:(-1,1)\to \dR$ be strictly convex and differentiable and such that $\lim_{x\to -1} g'(x)=-\infty$ and $\lim_{x\to 1} g'(x)=\infty$. Recall that the Legendre transform of $g$ is given for every $y\in \dR$ by
		\begin{equation*}
			g^*(y)=\sup_{x\in (-1,1)}(yx-g(x)).
		\end{equation*}
		Moreover, the optimal $x$ is unique and given by $x=(g')^{-1}(y)$ (notice that $g'$ is strictly monotonic and surjective). Hence, $y=g'(x)$ and 
		\begin{equation}\label{eq:f Legendre}
			g(x)-xg'(x)=-g^*(g'(x)).
		\end{equation}
		
		Applying this to $g=h_\zeta(q,\cdot)$ (which satisfies the above assumption by Lemma \ref{lemma:convexity}), we obtain
		\begin{multline*}
			\int_{[-1,1]^N} \prod_{i=1}^N \partial_{mm}h_\zeta(q,m_i)e^{-u\sum_{i=1}^N (h_\zeta(q,m_i)-\partial_m h_\zeta(q,m_i)m_i)-\frac{1}{2\xi'(q)}\Vert \partial_m h_\zeta(q,\bfm)\Vert^2}\dd \bfm\\=\int_{[-1,1]^N} \prod_{i=1}^N \partial_{mm}h_\zeta(q,m_i)e^{u\sum_{i=1}^N (h_\zeta(q,\cdot))^*(\partial_m h_\zeta(q,m)) -\frac{1}{2\xi'(q)}\Vert \partial_m h_\zeta(q,\bfm)\Vert^2}\dd \bfm.  
		\end{multline*}
		Performing the change of variables $y=\partial_m h_\zeta(q,m)$ yields
		\begin{multline}\label{eq:comp1}
			\int_{[-1,1]^N} \prod_{i=1}^N \partial_{mm}h_\zeta(q,m_i)e^{-u\sum_{i=1}^N (h_\zeta(q,m_i)-\partial_m h_\zeta(q,m_i)m_i)-\frac{1}{2\xi'(q)}\Vert \partial_m h_\zeta(q,\bfm)\Vert^2}\dd \bfm\\= \left(\int_{\dR} e^{uh_\zeta(q,\cdot)^*(y)-\frac{1}{2\xi'(q)}y^2}\dd y\right)^N.
		\end{multline}
		Hence,
		\begin{multline*}
			\log \int_{[-1,1]^N} \prod_{i=1}^N \partial_{mm}h_\zeta(q,m_i)e^{-u\sum_{i=1}^N (h_\zeta(q,m_i)-\partial_m h_\zeta(q,m_i)m_i)-\frac{1}{2\xi'(q)}\Vert \partial_m h_\zeta(q,\bfm)\Vert^2}\dd \bfm-\frac{N}{2}\log(2\pi \xi'(q))\\
			= N\log \dE_{\mc{N}(0,\xi'(q))}[e^{u\Phi_\zeta(q,X)}].
		\end{multline*}
		Recall that for every $x\in \dR$,
		\begin{equation*}
			\Phi_\zeta(0,x)=\frac{1}{u}\log \dE_{X\sim\mc{N}(0,\xi'(q))}\Bigl[e^{u\Phi_\zeta(q,X+x)}\Bigr].
		\end{equation*}
		Hence, we find 
		\begin{multline}\label{eq:Phizeta}
			\log \int_{[-1,1]^N} \prod_{i=1}^N \partial_{mm}h_\zeta(q,m_i)e^{-u\sum_{i=1}^N (h_\zeta(q,m_i)-\partial_m h_\zeta(q,m_i)m_i)-\frac{1}{2\xi'(q)}\Vert \partial_m h_\zeta(q,\bfm)\Vert^2}\dd \bfm-\frac{N}{2}\log(2\pi \xi'(q))\\
			=uN \Phi_\zeta(0,0).
		\end{multline}
		Moreover, notice that 
		\begin{equation*}
			u(\xi(q)-\xi'(q)q)=-\int_0^q t\xi''(t) \zeta([0,t])\dd t,
		\end{equation*}
		since $\zeta([0,q))=\zeta(\{0\})=u$. Recalling that $\mc{U}_\zeta(q)=\frac{1}{2}\int_q^1 t\xi''(t)\zeta([0,t])\dd t$, we deduce that 
		\begin{equation*}
			\frac{1}{2} u^2 (\xi(q)-\xi'(q)q)-u\mc{U}_\zeta(q)=-\frac{u}{2}\int_0^1 t\xi''(t)\zeta([0,t])\dd t. 
		\end{equation*}
		Inserting this into \eqref{eq:Phizeta}, we obtain 
		\begin{multline}\label{eq:partial conclusion}
			\frac{1}{2}N\xi(q)u^2 -uN\mc{U}_\zeta(q)-Nu f-\frac{N}{2}\log(2\pi \xi'(q))-\frac{1}{2}N\xi'(q)qu^2\\
			+\log \int_{[-1,1]^N} \prod_{i=1}^N \partial_{mm}h_\zeta(q,m_i)e^{-u\sum_{i=1}^N (h_\zeta(q,m_i)-\partial_m h_\zeta(q,m_i)m_i)-\frac{1}{2\xi'(q)}\Vert \partial_m h_\zeta(q,\bfm)\Vert^2}\dd \bfm\\
			=Nu\left(\Phi_\zeta(0,0)-\frac{1}{2}\int_0^1 t\xi''(t)\zeta([0,t])\dd t-f\right).
		\end{multline}
		By \eqref{eq:boundN}, this proves the lemma.
	\end{proof}
	
	\subsection{Optimality consequences}
	
	In this subsection, we derive necessary conditions satisfied by the optimizer $\zeta$ of the Parisi functional over $\Prefix_2(u,q)$ and by the critical point $(u,q)$ of the complexity functional. These conditions, expressed in terms of the Auffinger--Chen process, are needed in the proof of the lower bound.
	
	\begin{lemma}\label{lemma:opt2prefix}
		Let $f\in \dR$, $q\in (0,1)$ and $u\in (0,1)$. Let $\zeta\in \Prefix_2(u,q)$ satisfy the assumptions of Proposition \ref{prop:annealed 0}. Let $(X_t)_{t\in [0,1]}$ be the Auffinger--Chen process
		\begin{equation*}
			\begin{cases}
				\dd X_t=\xi''(t)\zeta([0,t])\partial_x \Phi_\zeta(t,X_t)\dd t+\sqrt{\xi''(t)} \dd B_t\\
				X_0=0.
			\end{cases}
		\end{equation*}
		For every $s\in [0,1]$, set
		\begin{equation*}
			H_\zeta(s):=\frac{1}{2}\int_s^1 \xi''(r)\Bigl(\dE[(\partial_x\Phi_\zeta(r,X_r))^2]-r\Bigr)\dd r.
		\end{equation*}
		Then, the following hold:
		\begin{enumerate}[label=(\roman*)]
			\item \begin{equation}\label{eq:energyconstraint}
				-\dE[\Phi_\zeta(q,X_q)]+\int_0^q t\xi''(t)\zeta([0,t])\dd t +\frac{1}{2}\int_q^1 t\xi''(t)\zeta([0,t])\dd t=-f,
			\end{equation}
			\item \begin{equation}\label{eq:breakingpoint}
				\dE[(\partial_x\Phi_\zeta(q,X_q))^2]=q,
			\end{equation}
			\item  \begin{equation}\label{eq:Hzetaconstraint}
				\supp(\zeta)\cap [q,1]\subset \underset{[q,1]}{\mathrm{arg\,min}}(H_\zeta).
			\end{equation}
		\end{enumerate}
	\end{lemma}

	\begin{proof}
		We first record the first-variation formula for $\Pari$. By Remark~\ref{remark:delta0} and Lemma~\ref{lemma:first order mu}\ref{item:gateaumu} with $q=0$ and $\mu=\delta_0$, for every finite signed measure $\eta$ on $[0,1]$ with $\eta([0,1])=0$,
		\[
		D\Pari(\zeta)[\eta]
		=
		\frac12\int_0^1 \xi''(s)\Bigl(\E[(\partial_x\Phi_\zeta(s,X_s))^2]-s\Bigr)\eta([0,s])\,\dd s
		=
		\int_0^1 H_\zeta(s)\,\dd\eta(s),
		\]
		where the second equality is Fubini's theorem.
		
		Set
		\[
		\lambda:=\frac{1}{1-u}\,\zeta|_{[q,1]}\in\mc P([q,1]),
		\qquad
		\zeta=u\delta_0+(1-u)\lambda.
		\]
		Since $\zeta\in \Prefix_2(u,q)$, we have $q\in \supp(\lambda)$.
		
		\smallskip
		\noindent{\bf{Step 1: proof of \eqref{eq:Hzetaconstraint}.}}
		Let $\lambda'\in\mc P([q,1])$ with $q\in \supp(\lambda')$, and define
		\[
		\lambda_\theta:=(1-\theta)\lambda+\theta\lambda',
		\qquad
		\zeta_\theta:=u\delta_0+(1-u)\lambda_\theta,
		\qquad \theta\in[0,1].
		\]
		Then $\zeta_\theta\in \Prefix_2(u,q)$. Since $\zeta$ minimizes $\Pari$ over $\Prefix_2(u,q)$, the right derivative at $\theta=0$ is nonnegative:
		\[
		0\le \frac{\dd}{\dd\theta}\bigg|_{\theta=0^+}\Pari(\zeta_\theta)
		=(1-u)\int_q^1 H_\zeta(s)\,\dd(\lambda'-\lambda)(s).
		\]
		Hence
		\[
		\int_q^1 H_\zeta\,\dd\lambda \le \int_q^1 H_\zeta\,\dd\lambda'
		\]
		for every $\lambda'\in\mc P([q,1])$ with $q\in \supp(\lambda')$.
		
		Now let $\nu\in\mc P([q,1])$ be arbitrary and set
		\[
		\nu_\varepsilon:=(1-\varepsilon)\nu+\varepsilon\delta_q.
		\]
		Then $q\in \supp(\nu_\varepsilon)$, so
		\[
		\int_q^1 H_\zeta\,\dd\lambda \le \int_q^1 H_\zeta\,\dd\nu_\varepsilon.
		\]
		By Lemma~\ref{lemma:martingales}\ref{item:mart-square}, the map
		$s\mapsto \E[(\partial_x\Phi_\zeta(s,X_s))^2]$ is continuous, hence $H_\zeta$ is continuous. Letting $\varepsilon\downarrow0$, we obtain
		\[
		\int_q^1 H_\zeta\,\dd\lambda \le \int_q^1 H_\zeta\,\dd\nu
		\qquad\text{for every }\nu\in\mc P([q,1]).
		\]
		Therefore $\lambda$ minimizes the linear functional
		$\nu\mapsto \int_q^1 H_\zeta\,\dd\nu$ over $\mc P([q,1])$, and hence
		\[
		\supp(\lambda)\subset \underset{[q,1]}{\mathrm{arg\,min}}(H_\zeta).
		\]
		Since $\supp(\lambda)=\supp(\zeta)\cap[q,1]$, this is exactly
		\eqref{eq:Hzetaconstraint}.
		
		\smallskip
		\noindent{\bf{Step 2: proof of \eqref{eq:breakingpoint}.}}
		Set
		\[
		\mu:=\mathrm{Law}\bigl(\partial_x\Phi_\zeta(q,X_q)\bigr),
		\qquad
		\zeta_\star:=\zeta|_{(q,1]}+\zeta([0,q])\delta_q\in\mc P([q,1]).
		\]
		For every $t\in[q,1]$, one has $\zeta_\star([0,t])=\zeta([0,t])$, hence
		\[
		\Phi_{\zeta_\star}(t,\cdot)=\Phi_\zeta(t,\cdot),\qquad t\in[q,1].
		\]
		Let $(X_t^\mu)_{t\in[q,1]}$ be the Auffinger--Chen process from
		\eqref{eq:AC mu} associated with $(\mu,\zeta_\star)$.
		Because
		\[
		\mu=\mathrm{Law}(\partial_x\Phi_\zeta(q,X_q))
		\qquad\text{and}\qquad
		\partial_x\Phi_{\zeta_\star}(q,\cdot)=\partial_x\Phi_\zeta(q,\cdot),
		\]
		we may choose $X_q^\mu\stackrel d=X_q$. The SDEs on $[q,1]$ then coincide, so
		\[
		(X_t^\mu)_{t\in[q,1]}\stackrel d=(X_t)_{t\in[q,1]}.
		\]
		Consequently, for every $s\in[q,1]$,
		\[
		H_{\zeta_\star}^\mu(s)
		=
		\frac12\int_s^1 \xi''(r)\Bigl(\E[(\partial_x\Phi_{\zeta_\star}(r,X_r^\mu))^2]-r\Bigr)\dd r
		=
		H_\zeta(s).
		\]
		
		By \eqref{eq:Hzetaconstraint}, every point of $\supp(\zeta)\cap[q,1]$ minimizes
		$H_\zeta$ on $[q,1]$. Since $q\in \supp(\zeta)$, we also have
		\[
		H_\zeta(q)=\min_{[q,1]}H_\zeta.
		\]
		Therefore every point of
		\[
		\supp(\zeta_\star)=\{q\}\cup (\supp(\zeta)\cap(q,1])
		\]
		minimizes $H_{\zeta_\star}^\mu$ on $[q,1]$, and hence minimizes its extension
		$\hat H_{\zeta_\star}^\mu$ from \eqref{def:Hhatmu} on $[0,1]$.
		By Lemma~\ref{lemma:first order mu}\ref{item:firstordermu},
		$\zeta_\star$ is the unique minimizer of $\TAP(\mu,\cdot)$ on $\mc P([q,1])$.
		Applying Lemma~\ref{lemma:first order mu}\ref{item:optmu} at the point
		$q\in \supp(\zeta_\star)$, we get
		\[
		\E\bigl[(\partial_x\Phi_{\zeta_\star}(q,X_q^\mu))^2\bigr]=q.
		\]
		Since $\Phi_{\zeta_\star}(q,\cdot)=\Phi_\zeta(q,\cdot)$ and $X_q^\mu\stackrel d=X_q$,
		this is exactly \eqref{eq:breakingpoint}.
		
		\smallskip
		\noindent{\bf{Step 3: proof of \eqref{eq:energyconstraint}.}}
		For $v$ close to $u$, define
		\[
		\zeta_v:=v\delta_0+(1-v)\lambda.
		\]
		Because $q\in \supp(\lambda)$, we have $\zeta_v\in\Prefix_2(v,q)$. Set
		\[
		g(v):=v\bigl(\Pari(\zeta_v)-f\bigr).
		\]
		By definition of $\Ppar(v;q)$,
		\[
		\Cpx_f(v;q)=v\bigl(\Ppar(v;q)-f\bigr)\le g(v),
		\]
		and equality holds at $v=u$, because $\zeta_u=\zeta$ and
		$\zeta$ attains the infimum defining $\Ppar(u;q)$.
		Hence the function $g-\Cpx_f(\cdot;q)$ has a local minimum $0$ at $u$.
		Since $u$ is a critical point of $v\mapsto \Cpx_f(v;q)$, we obtain
		\[
		0=(g-\Cpx_f(\cdot;q))'(u)=g'(u).
		\]
		Using the first-variation formula above and the fact that
		$\supp(\lambda)\subset\mathrm{argmin}_{[q,1]}H_\zeta$, we obtain
		\[
		0
		=
		g'(u)
		=
		\Pari(\zeta)-f
		+
		u\,D\Pari(\zeta)[\delta_0-\lambda]
		=
		\Pari(\zeta)-f+u\bigl(H_\zeta(0)-H_\zeta(q)\bigr).
		\]
		
		Now Lemma~\ref{lemma:martingales}\ref{item:mart-pathwise} gives
		\[
		\dd\Phi_\zeta(t,X_t)
		=
		\frac12\,\xi''(t)\,\zeta([0,t])\bigl(\partial_x\Phi_\zeta(t,X_t)\bigr)^2\,\dd t
		+\sqrt{\xi''(t)}\,\partial_x\Phi_\zeta(t,X_t)\,\dd B_t.
		\]
		Taking expectations and integrating from $0$ to $q$,
		\[
		\E[\Phi_\zeta(q,X_q)]
		=
		\Phi_\zeta(0,0)
		+\frac12\int_0^q \xi''(t)\,\zeta([0,t])\,\E\bigl[(\partial_x\Phi_\zeta(t,X_t))^2\bigr]\dd t.
		\]
		Since $\zeta([0,t])=u$ for $t<q$, this implies
		\[
		\Phi_\zeta(0,0)+u\bigl(H_\zeta(0)-H_\zeta(q)\bigr)
		=
		\E[\Phi_\zeta(q,X_q)]
		-\frac12\int_0^q t\,\xi''(t)\,\zeta([0,t])\,\dd t.
		\]
		Recalling that
		\[
		\Pari(\zeta)=\Phi_\zeta(0,0)-\frac12\int_0^1 t\,\xi''(t)\,\zeta([0,t])\,\dd t,
		\]
		we deduce
		\[
		\Pari(\zeta)+u\bigl(H_\zeta(0)-H_\zeta(q)\bigr)
		=
		\E[\Phi_\zeta(q,X_q)]
		-\int_0^q t\,\xi''(t)\,\zeta([0,t])\,\dd t
		-\frac12\int_q^1 t\,\xi''(t)\,\zeta([0,t])\,\dd t.
		\]
		Substituting this into the identity $0=g'(u)$ yields
		\[
		-\E[\Phi_\zeta(q,X_q)]
		+\int_0^q t\,\xi''(t)\,\zeta([0,t])\,\dd t
		+\frac12\int_q^1 t\,\xi''(t)\,\zeta([0,t])\,\dd t
		=-f,
		\]
		which is \eqref{eq:energyconstraint}.
	\end{proof}

	\subsection{Lower bound on the complexity}
	
	In this subsection, we prove the matching lower bound on the annealed complexity, completing the proof of Proposition~\ref{prop:annealed 0}. The argument combines the Kac--Rice representation with a reduction to a ``good'' event, and uses the optimality conditions from the previous subsection to verify that this event has full probability.

	\begin{proof}[Proof of Proposition \ref{prop:annealed 0}]
		By Lemma \ref{lemma:upper bound 0}, it remains to prove the lower bound. Let $\zeta\in \Prefix_2(u,q)$. We write $C_\ve$ for a positive constant satisfying $C_\ve\to 0$ as $\ve\to 0$, 
		and $O_\ve(N)$ for a quantity bounded in absolute value by $C_\ve N$.
		
		Set
		\[
		\zeta_\star := \zeta|_{(q,1]} + \zeta([0,q])\,\delta_q.
		\]

		\paragraph{\bf Step 1: reduction to a good event}
		Set $\delta_\ve=\ve^{\frac{1}{2}}$. Let $\mathrm{Good}(\varepsilon)$ denote the set of $\mathbf{m} \in [-1+\delta_\ve,1-\delta_\ve ]^N$ such that
		\begin{gather*}
			\frac{1}{N}\|\mathbf{m}\|^2 \in [q-\varepsilon,\, q+\varepsilon], \\[6pt]
			\mathrm{dist}(\zeta_{\mathbf{m}},\, \zeta_\star) \leq \varepsilon,\\
			\left|
			\frac{1}{N}\sum_{i=1}^N \partial_m h_\zeta(q, m_i)\, m_i
			- \xi'(q)\int_0^1 \zeta([0,t])\,\mathrm{d}t
			\right| \leq \varepsilon, \\[6pt]
			\frac{1}{N}\sum_{i=1}^N
			\bigl(h_\zeta(q, m_i) - \partial_m h_\zeta(q, m_i)\, m_i\bigr)
			+ \int_0^q t\,\xi''(t)\,\zeta([0,t])\,\mathrm{d}t
			+ \tfrac{1}{2}\int_q^1 t\,\xi''(t)\,\zeta([0,t])\,\mathrm{d}t\\
			\in [-f-\varepsilon,-f+\varepsilon].
		\end{gather*}
		We claim that
		\begin{multline}\label{eq:dropgood}
			\log \dE[\mc{N}_\ve(f,\zeta,q)]\geq  \frac{1}{2}N\xi(q)u^2 -uN\mc{U}_\zeta(q)-Nu f-\frac{N}{2}\log(2\pi \xi'(q))-\frac{1}{2}N\xi'(q)qu^2\\
			+\log \int \prod_{i=1}^N \partial_{mm}h_\zeta(q,m_i)e^{-u\sum_{i=1}^N (h_\zeta(q,m_i)-\partial_m h_\zeta(q,m_i)m_i)-\frac{1}{2\xi'(q)}\Vert \partial_m h_\zeta(q,\bfm)\Vert^2}\indic_{\Good(\ve)}(\bfm) \dd \bfm\\+O_\ve(N).
		\end{multline}
		We prove this separately in the mixed and pure cases.
		
		\smallskip
		\noindent{\bf Mixed case.}
		By Lemma~\ref{lemma:KacRice},
		\begin{multline*}
			\mathbb{E}[\mathcal{N}_\varepsilon(f,\zeta,q)]
			= N \int_{[-1,1]^N} \int_{(f-\varepsilon,\, f+\varepsilon)}
			\mathbb{E}\bigl[|\det(\nabla^2 F_{\mathrm{TAP},\zeta}(\mathbf{m}))|
			\mid \nabla F_{\mathrm{TAP},\zeta}(\mathbf{m}) = 0,\,
			F_{\mathrm{TAP},\zeta}(\mathbf{m}) = Nf'\bigr] \\
			\times\, \varphi_{(\nabla F_{\mathrm{TAP},\zeta}(\mathbf{m}),\,
				F_{\mathrm{TAP},\zeta}(\mathbf{m}))}(0, Nf')\,
			\mathbf{1}_{\mathrm{SUSY}_1(\zeta,\varepsilon)}(\mathbf{m})\,
			\mathrm{d}\mathbf{m}\, \mathrm{d}f'.
		\end{multline*}
		Since $\Good(\ve)\subset \SUSY_1(\zeta,\ve)$,
		\begin{multline*}
			\dE[\mc{N}_\ve(f,\zeta,q)]\geq N\int_{[-1,1]^N}\int_{(f-\ve,f+\ve)} \dE[ |\det(\nabla^2 F_{\TAP,\zeta}(\bfm))|\mid \nabla F_{\TAP,\zeta}(\bfm)=0,F_{\TAP,\zeta}(\bfm)=Nf'] \\ \times \varphi_{(\nabla F_{\TAP,\zeta}(\bfm),F_{\TAP,\zeta}(\bfm))}(0,Nf')\indic_{\Good(\ve)}(\bfm)\dd \bfm\dd f'.
		\end{multline*}
		By Lemma \ref{lemma:det asymp}, for every $\bfm\in \Good(\ve)$, 
		\begin{multline*}
			\log \dE[ |\det(\nabla^2 F_{\TAP,\zeta}(\bfm))|\mid \nabla F_{\TAP,\zeta}(\bfm)=0,F_{\TAP,\zeta}(\bfm)=Nf']\\ \geq \log\prod_{i=1}^N\partial_{mm}h_\zeta(q,m_i)+{\frac{N}{2}\xi''(q)\left(\int_q^1 \zeta([0,t])\dd t\right)^2 }+
			O_\ve(N).
		\end{multline*}
		Moreover, by Lemma \ref{lemma:prob0 non deg}, there exists $C>0$ such that for every $\bfm\in \Good(\ve)$,
		\begin{multline*}
			\log \varphi_{(\nabla F_{\TAP,\zeta}(\bfm),F_{\TAP,\zeta}(\bfm))}(0,Nf')\\ =\frac{1}{2}N\xi(q)u^2-\frac{1}{2}N\xi''(q)\left(\int_q^1 \zeta([0,t])\dd t \right)^2 -\frac{1}{2\xi'(q)}\Vert \partial_m h_\zeta(q,\bfm)-u\xi'(q)\bfm\Vert^2-\frac{N}{2}\log (2\pi \xi'(q))\\
			-u \left(\sum_{i=1}^N h_\zeta(q,m_i)+N\mc{U}_\zeta(q)+Nf'\right)+O_\ve(N).
		\end{multline*}
		Inserting the last two displays into the previous lower bound, and using that $f'\in(f-\ve,f+\ve)$, yields \eqref{eq:dropgood}.
		
		\smallskip
		\noindent{\bf Pure case.}
		By Lemma~\ref{lemma:KacRice}, \eqref{eq:KacRice-pure},
		\begin{multline*}
			\dE[\mc N_\ve(f,\zeta,q)]
			=
			\int_{[-1,1]^N}
			\dE\!\left[
			|\det(\nabla^2 F_{\TAP,\zeta}(\bfm))|
			\,\middle|\,
			\nabla F_{\TAP,\zeta}(\bfm)=0
			\right]
			\varphi_{\nabla F_{\TAP,\zeta}(\bfm)}(0)
			\\
			\times
			\indic_{\SUSY_1(\zeta,\ve)}(\bfm)\,
			\indic_{(f-\ve,f+\ve)}
			\!\left(\frac1N\mc R^{\rm ex}_\zeta(\bfm)\right)
			\dd\bfm.
		\end{multline*}
		For $\bfm\in \Good(\ve)$, set
		\[
		T(\bfm):=\frac1p\langle \bfm,k_\zeta(q,\bfm)\rangle-\sum_{i=1}^N h_\zeta(q,m_i)-N\mc U_\zeta(q).
		\]
		The same algebra as in \eqref{eq:R-rewrite-pure} gives
		\begin{multline*}
			T(\bfm)
			=
			-\sum_{i=1}^N\bigl(h_\zeta(q,m_i)-\partial_m h_\zeta(q,m_i)m_i\bigr)
			-N\int_0^q t\xi''(t)\zeta([0,t])\,\dd t
			-\frac{N}{2}\int_q^1 t\xi''(t)\zeta([0,t])\,\dd t
			\\
			+\frac{p-1}{p}\!\left(
			N\xi'(q)\int_0^1 \zeta([0,t])\,\dd t
			-
			\langle \partial_m h_\zeta(q,\bfm),\bfm\rangle
			\right).
		\end{multline*}
		By the definition of $\Good(\ve)$, this implies
		\[
		T(\bfm)=Nf+O_\ve(N).
		\]
		Comparing the exact identity
		\[
		F_{\TAP,\zeta}(\bfm)
		=
		\frac1p\langle \bfm,\nabla F_{\TAP,\zeta}(\bfm)\rangle+\mc R^{\rm ex}_\zeta(\bfm)
		\]
		with \eqref{eq:F-R-pure}, we obtain
		\[
		\mc R^{\rm ex}_\zeta(\bfm)=T(\bfm)+O_\ve(N)=Nf+O_\ve(N)
		\]
		uniformly on $\Good(\ve)$. After decreasing $\ve$ if necessary (and relabelling it), we may therefore assume that
		\[
		\Good(\ve)\subset \left\{\bfm:\frac1N\mc R^{\rm ex}_\zeta(\bfm)\in(f-\ve,f+\ve)\right\}.
		\]
		Hence
		\begin{multline*}
			\dE[\mc N_\ve(f,\zeta,q)]
			\ge
			\int_{[-1,1]^N}
			\dE\!\left[
			|\det(\nabla^2 F_{\TAP,\zeta}(\bfm))|
			\,\middle|\,
			\nabla F_{\TAP,\zeta}(\bfm)=0
			\right]
			\varphi_{\nabla F_{\TAP,\zeta}(\bfm)}\!(0)
			\,\indic_{\Good(\ve)}(\bfm)\,\dd\bfm.
		\end{multline*}
		
		Now fix $\bfm\in \Good(\ve)$ and set
		\[
		f_\bfm:=\frac1N\mc R^{\rm ex}_\zeta(\bfm).
		\]
		By \eqref{eq:pure-critical-identity-KR}, on the event
		$\{\nabla F_{\TAP,\zeta}(\bfm)=0\}$ one has
		$F_{\TAP,\zeta}(\bfm)=Nf_\bfm$. Therefore
		\begin{multline*}
			\dE\!\left[
			|\det(\nabla^2 F_{\TAP,\zeta}(\bfm))|
			\,\middle|\,
			\nabla F_{\TAP,\zeta}(\bfm)=0
			\right]
			\\=
			\dE\!\left[
			|\det(\nabla^2 F_{\TAP,\zeta}(\bfm))|
			\,\middle|\,
			\nabla F_{\TAP,\zeta}(\bfm)=0,\,
			F_{\TAP,\zeta}(\bfm)=Nf_\bfm
			\right].
		\end{multline*}
		Hence, Lemma~\ref{lemma:det asymp} implies
		\begin{multline*}
			\log \dE\!\left[
			|\det(\nabla^2 F_{\TAP,\zeta}(\bfm))|
			\,\middle|\,
			\nabla F_{\TAP,\zeta}(\bfm)=0
			\right]\\
			\ge
			\log\prod_{i=1}^N\partial_{mm}h_\zeta(q,m_i)
			+\frac{N}{2}\xi''(q)\left(\int_q^1 \zeta([0,t])\dd t\right)^2
			+O_\ve(N).
		\end{multline*}
		Moreover, $\Good(\ve)$ implies \eqref{eq:pure-ass1} and \eqref{eq:pure-ass2}, so by
		Lemma~\ref{lemma:prob0 pure}, \eqref{eq:pure-prob0-simplified}, and $T(\bfm)=Nf+O_\ve(N)$,
		\begin{multline*}
			\log \varphi_{\nabla F_{\TAP,\zeta}(\bfm)}(0)\\
			=\frac{1}{2}N\xi(q)u^2-\frac{1}{2}N\xi''(q)\left(\int_q^1 \zeta([0,t])\dd t \right)^2 -\frac{1}{2\xi'(q)}\Vert \partial_m h_\zeta(q,\bfm)-u\xi'(q)\bfm\Vert^2-\frac{N}{2}\log (2\pi \xi'(q))\\
			-u \left(\sum_{i=1}^N h_\zeta(q,m_i)+N\mc{U}_\zeta(q)+Nf\right)+O_\ve(N).
		\end{multline*}
		Inserting the last two displays into the pure Kac--Rice lower bound above yields
		\eqref{eq:dropgood}.
		
		\paragraph{\bf Step 2: checking the constraints asymptotically}
		Consider the probability measure $\mu$ on $[-1,1]$ with density 
		\begin{equation*}
			\dd \mu(m)\propto \partial_{mm}h_\zeta(q,m)e^{-u (h_\zeta(q,m)-\partial_m h_\zeta(q,m)m)-\frac{1}{2\xi'(q)}\partial_m h_\zeta(q,m)^2}\dd m.
		\end{equation*}
		Let $\nu:=(\partial_m h_\zeta(q,m))\# \mu$. Observe that by the Legendre transform identity \eqref{eq:f Legendre}, $\nu$ is the probability measure on $\dR$ with density 
		\begin{equation*}
			\frac{\dd \nu}{\dd y}(y)\propto e^{-\frac{1}{2\xi'(q)}y^2+u\Phi_\zeta(q,y)}.
		\end{equation*}
		Hence, by Lemma \ref{lemma:law Xt s}, we get that
		\begin{equation*}
			\nu=\mathrm{Law}(X_q),
		\end{equation*}
		where $X_t$ is the Auffinger--Chen process
		\begin{equation*}
			\begin{cases}
				\dd X_t=\xi''(t)\zeta([0,t])\partial_x \Phi_\zeta(t,X_t)\dd t+\sqrt{\xi''(t)} \dd B_t,\\
				X_0=0.
			\end{cases}
		\end{equation*}
		
		$\bullet$ First,
		\begin{equation*}
			\dE_\mu[m^{2}]=\dE_{\nu}[\partial_x\Phi_\zeta(q,\cdot)^2]=\dE[\partial_x\Phi_\zeta(q,X_q)^2 ].
		\end{equation*}
		By Lemma \ref{lemma:opt2prefix}, (see \eqref{eq:breakingpoint}),
		\begin{equation*}
			\dE[\partial_x \Phi_\zeta(q,X_q)^2]=q.
		\end{equation*}
		Thus,
		\begin{equation*}
			\dE_\mu[m^{2}]=q.
		\end{equation*}
		
		$\bullet$ One has
		\begin{equation*}
			\dE_\mu[m\partial_m h_\zeta(q,m)]=\dE[\partial_{x}\Phi_\zeta(q,X_q)X_q]=\xi'(q)\int_0^1 \zeta([0,t])\dd t,
		\end{equation*}
		where we have used \eqref{eq:deltaM deltaX} in Lemma \ref{lemma:expectations} with $s=0$ in the last equality.
		
		$\bullet$ By the Legendre transform identity \eqref{eq:f Legendre}, notice that 
		\begin{equation*}
			\dE_\mu[h_\zeta(q,m)-\partial_m h_\zeta(q,m)m]=-\dE[\Phi_\zeta(q,X_q)].
		\end{equation*}
		By Lemma \ref{lemma:opt2prefix}, Equation \eqref{eq:energyconstraint},
		\begin{equation*}
			-\dE[\Phi_\zeta(q,X_q)]+\int_0^q t\xi''(t)\zeta([0,t])\dd t+\frac{1}{2}\int_q^1 t\xi''(t)\zeta([0,t])\dd t=-f. 
		\end{equation*}
		Therefore,
		\begin{equation*}
			\dE_\mu[h_\zeta(q,m)-\partial_m h_\zeta(q,m)m]+\int_0^q t\xi''(t)\zeta([0,t])\dd t+\frac{1}{2}\int_q^1 t\xi''(t)\zeta([0,t])\dd t=-f.
		\end{equation*}
		
		$\bullet$ By Lemma \ref{lemma:first order mu}, to prove that $\zeta_\star$ is the unique minimizer of $\TAP(\mu,\cdot)$, it is enough to check that for every $\zeta'\in\mc{P}([q,1])$, 
		\begin{equation}\label{eq:FOmu}
			\int_q^1 H^\mu_{\zeta_\star}(s)\dd \zeta'(s)\geq \int_q^1 H^\mu_{\zeta_\star}(s)\dd \zeta_\star(s),
		\end{equation}
		where for every $s\in [q,1]$,
		\begin{equation*}
			H^\mu_{\zeta_\star}(s)=\frac{1}{2}\int_s^1 \xi''(t)\Bigl(\dE[(\partial_x\Phi_{\zeta_\star}(t,X_t^\mu))^2]-t\Bigr)\dd t,
		\end{equation*}
		and $(X_t^{\mu})_{t\in [q,1]}$ is the Auffinger--Chen process
		\begin{equation*}
			\begin{cases}
				\dd X_t^\mu=\xi''(t)\zeta_\star([0,t])\partial_x \Phi_{\zeta_\star}(t,X^\mu_t)\dd t+\sqrt{\xi''(t)} \dd B_t,\\
				\mathrm{Law}(\partial_x\Phi_{\zeta_\star}(q,X^\mu_q))=\mu.
			\end{cases}
		\end{equation*}
		Since $\zeta_\star([0,t])=\zeta([0,t])$ for every $t\in [q,1]$, the Parisi PDE on $[q,1]$ is the same for $\zeta_\star$ and $\zeta$. Hence
		\begin{equation*}
			\Phi_{\zeta_\star}(t,\cdot)=\Phi_\zeta(t,\cdot)\qquad \text{for every }t\in [q,1].
		\end{equation*}
		Moreover, by the construction of $\mu$ and the identity $\partial_m h_\zeta(q,\cdot)=(\partial_x\Phi_\zeta(q,\cdot))^{-1}$, we have
		\begin{equation*}
			(\partial_x\Phi_\zeta(q,\cdot))^{-1}\#\mu=\mathrm{Law}(X_q).
		\end{equation*}
		Therefore, $(X_t^\mu)_{t\in [q,1]}$ has the same distribution as $(X_t)_{t\in [q,1]}$, and so
		\begin{equation*}
			H^\mu_{\zeta_\star}(s)=H_\zeta(s):=\frac{1}{2}\int_s^1 \xi''(t)\Bigl(\dE[(\partial_x\Phi_\zeta(t,X_t))^2]-t\Bigr)\dd t,
			\qquad s\in [q,1].
		\end{equation*}
		By Lemma \ref{lemma:opt2prefix}, Equation \eqref{eq:Hzetaconstraint}, every point of $\supp(\zeta)\cap [q,1]$ minimizes $H_\zeta$ on $[q,1]$. Since $\zeta\in\Prefix_2(u,q)$, we have $q\in \supp(\zeta)$, hence
		\begin{equation*}
			H_\zeta(q)=\min_{[q,1]} H_\zeta.
		\end{equation*}
		It follows that every point of
		\begin{equation*}
			\supp(\zeta_\star)=\{q\}\cup (\supp(\zeta)\cap (q,1])
		\end{equation*}
		also minimizes $H_\zeta$ on $[q,1]$. Consequently,
		\begin{equation*}
			\int_q^1 H_\zeta(s)\dd \zeta_\star(s)=H_\zeta(q)=\min_{[q,1]}H_\zeta.
		\end{equation*}
		Therefore, for every probability measure $\zeta'$ on $[q,1]$,
		\begin{equation*}
			\int_q^1 H_\zeta(s)\dd \zeta'(s)\geq \min_{[q,1]}H_\zeta
			=\int_q^1 H_\zeta(s)\dd \zeta_\star(s).
		\end{equation*}
		This proves \eqref{eq:FOmu}, and thus $\zeta_\star$ is the unique minimizer of $\TAP(\mu,\cdot)$ by strict convexity.
		
		\paragraph{\bf Step 3: control on the good event under the truncated product law}
		Set
		\[
		w_\zeta(m):=\partial_{mm}h_\zeta(q,m)e^{-u (h_\zeta(q,m)-\partial_m h_\zeta(q,m)m)-\frac{1}{2\xi'(q)}\partial_m h_\zeta(q,m)^2}
		\]
		and, for $\ve\in(0,1)$,
		\[
		Z_\ve:=\int_{-1+\delta_\ve}^{1-\delta_\ve} w_\zeta(m)\,\dd m,
		\qquad
		\dd\mu_\ve(m):=\frac{\indic_{[-1+\delta_\ve,1-\delta_\ve]}(m)\,w_\zeta(m)}{Z_\ve}\,\dd m.
		\]
		Since $w_\zeta\in L^1([-1,1])$ (this follows from the change of variables $m=\partial_x\Phi_\zeta(q,y)$, which identifies $w_\zeta\,\dd m$ with the density of $\mathrm{Law}(X_q)$ up to a normalizing constant; see Lemma~\ref{lemma:law Xt s}), we have $Z_\ve\uparrow Z_0:=\int_{-1}^1 w_\zeta(m)\,\dd m$ as $\ve\downarrow 0$, and $\mu_\ve\Rightarrow\mu$ weakly.
		
		Let
		\[
		q_\ve:=\int m^2\,\mu_\ve(\dd m)
		\]
		and let $\zeta_{\star,\ve}$ be the unique minimizer of $\TAP(\mu_\ve,\cdot)$.
		By Step~2, the functions
		\[
		m\mapsto m^2,\qquad
		m\mapsto m\partial_m h_\zeta(q,m),\qquad
		m\mapsto h_\zeta(q,m)-m\partial_m h_\zeta(q,m)
		\]
		are $\mu$-integrable, so their expectations under $\mu_\ve$ converge to the corresponding target values appearing in the definition of $\Good(\ve)$. Moreover, by Lemma~\ref{lemma:stableTapMin}, since $q_\ve\to q$ and $\mu_\ve\Rightarrow\mu$, we have
		\[
		\dist(\zeta_{\star,\ve},\zeta_\star)\xrightarrow[\ve\to 0]{} 0.
		\]
		After decreasing $\ve$ if necessary (and relabelling it), we may therefore assume that all four discrepancies are at most $\ve/4$.
		
		Now let $m_1,\ldots,m_N$ be i.i.d.\ with law $\mu_\ve$. By the law of large numbers, with probability tending to $1$ as $N\to\infty$,
		\[
		\left|\frac1N\sum_{i=1}^N m_i^2-q_\ve\right|<\ve/4,
		\]
		\[
		\left|
		\frac1N\sum_{i=1}^N m_i\partial_m h_\zeta(q,m_i)
		-\int m\partial_m h_\zeta(q,m)\,\mu_\ve(\dd m)
		\right|<\ve/4,
		\]
		and
		\[
		\left|
		\frac1N\sum_{i=1}^N\bigl(h_\zeta(q,m_i)-m_i\partial_m h_\zeta(q,m_i)\bigr)
		-\int \bigl(h_\zeta(q,m)-m\partial_m h_\zeta(q,m)\bigr)\,\mu_\ve(\dd m)
		\right|<\ve/4.
		\]
		Furthermore, by the i.i.d.\ part of Lemma~\ref{lemma:stableTapMin} applied to $\mu_\ve$,
		\[
		\mu_\ve^{\otimes N}\!\left(\dist(\zeta_\bfm,\zeta_{\star,\ve})>\ve/4\right)\xrightarrow[N\to\infty]{} 0.
		\]
		Using the bounds on the expectations and the triangle inequality, we conclude that
		\begin{equation*}
			\mu_\ve^{\otimes N}(\Good(\ve)^c)\xrightarrow[N\to\infty]{} 0.
		\end{equation*}
		
		\paragraph{\bf Step 4: conclusion}
		For every $N$ and every $\ve\in(0,1)$,
		\begin{multline*}
			\int_{[-1+\delta_\ve,1-\delta_\ve]^N} \prod_{i=1}^N \partial_{mm}h_\zeta(q,m_i)e^{-u\sum_{i=1}^N (h_\zeta(q,m_i)-\partial_m h_\zeta(q,m_i)m_i)-\frac{1}{2\xi'(q)}\Vert \partial_m h_\zeta(q,\bfm)\Vert^2}\indic_{\Good(\ve)}(\bfm) \dd \bfm\\
			= Z_\ve^N\,\mu_\ve^{\otimes N}(\Good(\ve)).
		\end{multline*}
		Hence, by Step~3,
		\begin{multline*}
			\liminf_{N\to\infty}\frac1N\log
			\int_{[-1+\delta_\ve,1-\delta_\ve]^N} \prod_{i=1}^N \partial_{mm}h_\zeta(q,m_i)\\
			\times e^{-u\sum_{i=1}^N (h_\zeta(q,m_i)-\partial_m h_\zeta(q,m_i)m_i)-\frac{1}{2\xi'(q)}\Vert \partial_m h_\zeta(q,\bfm)\Vert^2}\indic_{\Good(\ve)}(\bfm) \dd \bfm
			\ge \log Z_\ve.
		\end{multline*}
		Since $Z_\ve\uparrow Z_0$ as $\ve\downarrow 0$, we get
		\begin{multline*}
			\liminf_{\ve\to 0}\liminf_{N\to\infty}\frac1N\log
			\int_{[-1+\delta_\ve,1-\delta_\ve]^N} \prod_{i=1}^N \partial_{mm}h_\zeta(q,m_i)\\
			\times e^{-u\sum_{i=1}^N (h_\zeta(q,m_i)-\partial_m h_\zeta(q,m_i)m_i)-\frac{1}{2\xi'(q)}\Vert \partial_m h_\zeta(q,\bfm)\Vert^2}\indic_{\Good(\ve)}(\bfm) \dd \bfm
			\ge \log Z_0.
		\end{multline*}
		But
		\[
		Z_0
		=
		\int_{-1}^1 \partial_{mm}h_\zeta(q,m)e^{-u(h_\zeta(q,m)-\partial_m h_\zeta(q,m)m)-\frac{1}{2\xi'(q)}\partial_m h_\zeta(q,m)^2}\,\dd m,
		\]
		so
		\begin{multline*}
			\log Z_0
			=
			\lim_{N\to\infty}\frac1N\log \int_{[-1,1]^N} \prod_{i=1}^N \partial_{mm}h_\zeta(q,m_i)\\
			\times e^{-u\sum_{i=1}^N (h_\zeta(q,m_i)-\partial_m h_\zeta(q,m_i)m_i)-\frac{1}{2\xi'(q)}\Vert \partial_m h_\zeta(q,\bfm)\Vert^2}\dd \bfm.
		\end{multline*}
		Thus, by \eqref{eq:dropgood},
		\begin{multline*}
			\liminf_{\ve\to 0}\liminf_{N\to \infty} \frac{1}{N} \log \dE[\mc{N}_\ve(f,\zeta,q)]\geq  \frac{1}{2}\xi(q)u^2 -u\mc{U}_\zeta(q)-u f-\frac{1}{2}\log(2\pi \xi'(q))-\frac{1}{2}\xi'(q)qu^2\\
			+\liminf_{N\to \infty}\frac{1}{N}\log \int_{[-1,1]^N} \prod_{i=1}^N \partial_{mm}h_\zeta(q,m_i)e^{-u\sum_{i=1}^N (h_\zeta(q,m_i)-\partial_m h_\zeta(q,m_i)m_i)-\frac{1}{2\xi'(q)}\Vert \partial_m h_\zeta(q,\bfm)\Vert^2}\dd \bfm.
		\end{multline*}
		Inserting \eqref{eq:partial conclusion}, this concludes the proof of the proposition.
	\end{proof}
	
	\section{The conditional annealed computation}\label{section:multiple}

	The aim of the section is to perform conditional annealed computations of the complexity of the TAP free energy. We consider a hierarchy of ancestors and count the number of critical points in the band orthogonal to the last ancestor, conditionally on all the ancestors being critical points of the TAP free energy and conditionally on their free energy level. We will extend the computations of Section \ref{section:base}. We assume in this section that $\xi$ is a mixed model, i.e., at least two of the 
	coefficients $\beta_p^2$ are nonzero.

	\subsection{Statement of the conditional annealed estimate}
	
	In this subsection, we state the main result of Section~\ref{section:multiple}: the conditional annealed estimate for the TAP complexity. As in Section~\ref{section:base}, we first restrict our attention to ``SUSY'' states, imposing conditions on the hierarchy $(\bfm^{(1)},\ldots,\bfm^{(n)})$ that ensure the Ward identities are satisfied.

	\begin{definition}[SUSY hierarchy of states]\label{def:susyreplicas}
		Fix $n\geq 1$ and $\ve\in(0,1)$. Set $\bfm^{(0)}=0$. Let $(u_i)_{1\leq i\leq n}$ and
		$(q_i)_{1\leq i\leq n}$ be strictly increasing sequences in $(0,1)$, and let
		$\zeta\in\Prefix_{n+1}\bigl((u_i),(q_i)\bigr)$ as in Definition~\ref{def:nprefix}.
		We define $\SUSY_n(\zeta,\ve)$ as the set of
		$(\bfm^{(1)},\ldots,\bfm^{(n)})\in([-1+\delta_\ve,1-\delta_\ve]^N)^n$ where $\delta_\ve:=\ve^{\frac{1}{2}}$ and such that, for every $k\in[n]$
		and $j\in\{1,\ldots,k-1\}$,
		\begin{equation*}
			\left|\frac{1}{N}\|\bfm^{(k)}\|^2 - q_k\right| \leq \ve
			\qquad\text{and}\qquad
			\left|\frac{1}{N}\langle \bfm^{(k)}-\bfm^{(k-1)},\bfm^{(j)}\rangle\right| \leq \ve,
		\end{equation*}
		and, for every $i,k\in[n]$,
		\begin{equation}\label{eq:SUSYassumption}
			\left|\frac{1}{N}\langle\partial_m h_\zeta(q_k,\bfm^{(k)}),\bfm^{(i)}\rangle
			-\dE\bigl[X_{q_k}\,\partial_x\Phi_\zeta(q_i,X_{q_i})\bigr]\right|\leq\ve,
		\end{equation}
		and, for every $k\in[n]$,
		\begin{equation*}
			\dist\!\bigl(\zeta_{\bfm^{(k)}},\,\zeta|_{(q_k,1]}+\zeta([0,q_k])\,\delta_{q_k}\bigr)\leq\ve.
		\end{equation*}
	\end{definition}

	Note that the SUSY assumption is only \eqref{eq:SUSYassumption}. It would be interesting to know if one can remove it.

	Next, we make some assumptions on the ancestors \emph{only}. Roughly speaking, we want the empirical measure of the coordinates of the ancestor $\bfm^{(k)}$ to be distributed approximately according to the law of $\partial_x\Phi_\zeta(q_k,X_{q_k})$ where $(X_t)$ is the Auffinger--Chen process and $q_k:=\frac{1}{N}\Vert \bfm^{(k)}\Vert^2$. In fact, we only need these two distributions to agree on a \emph{finite number} of linear statistics.

	\begin{definition}[Ancestors matching the target law]\label{def:matchlaw}
		Fix $n\geq 2$. Let $\zeta$ be a probability measure on $[0,1]$ with
		$(n+1)$-atom prefix $\{0,q_1,\ldots,q_n\}$ as in Definition~\ref{def:nprefix};
		the tail $\supp(\zeta)\cap(q_n,1]$ may be nonempty.
		Let $(X_t)_{t\in[0,1]}$ be the Auffinger--Chen process~\eqref{eq:AuffingerChen},
		and write $U:=\partial_x\Phi_\zeta(q_{n-1},X_{q_{n-1}})$.
		Define $\Test:=\{\chi_1,\ldots,\chi_5\}$, where each $\chi_i:[-1,1]\to\dR$ is given by
		\begin{align}
			\chi_1(m) &:= \dE\!\left[(\partial_x\Phi_\zeta(q_n,X_{q_n}))^2 \;\middle|\; U=m\right], \label{def:chi1}\\
			\chi_2(m) &:= \dE\!\left[\partial_x\Phi_\zeta(q_n,X_{q_n})\,X_{q_n} \;\middle|\; U=m\right], \label{def:chi2}\\
			\chi_3(m) &:=\dE[\partial_x\Phi_\zeta(q_n,X_{q_n})\mid U=m]=m, \label{def:chi3}\\
			\chi_4(m) &:= \dE\!\left[X_{q_n} \;\middle|\; U=m\right], \label{def:chi4}\\
			\chi_5(m) &:= \dE\!\left[\Phi_\zeta(q_n,X_{q_n}) \;\middle|\; U=m\right]. \label{def:chi5}
		\end{align}
		For $\ve\in(0,1)$, let $\delta_\ve=\ve^{\frac{1}{2}}$ and define $\Lawmatch_{n-1}(\zeta,\ve)$ as the set of
		$(\bfm^{(1)},\ldots,\bfm^{(n-1)})\in([-1+\delta_\ve,1-\delta_\ve]^N)^{n-1}$ such that, for every
		$\chi\in\Test$ and $k\in[n-1]$,
		\begin{equation}\label{def:Linear}
			\left|\frac{1}{N}\sum_{i=1}^N\chi(m_i^{(n-1)})\,m_i^{(k)}
			-\dE\!\left[\chi(U)\,\partial_x\Phi_\zeta(q_k,X_{q_k})\right]\right|\leq\ve,
		\end{equation}
		\begin{equation}\label{def:Linear2}
			\left|\frac{1}{N}\sum_{i=1}^N\chi(m_i^{(n-1)})\,\partial_m h_\zeta(q_k,m_i^{(k)})
			-\dE\!\left[\chi(U)\,X_{q_k}\right]\right|\leq\ve,
		\end{equation}
		\begin{equation}\label{def:Linear3}
			\left|\frac{1}{N}\sum_{i=1}^N\chi(m_i^{(n-1)})
			-\dE\!\left[\chi(U)\right]\right|\leq\ve
		\end{equation}
		and, for every $k\in[n-1]$,
		\begin{equation}\label{def:Linear4}
			\left|\frac{1}{N}\sum_{i=1}^N \Phi_\zeta\!\Bigl(q_k,\partial_m h_\zeta(q_k,m_i^{(k)})\Bigr)
			-\dE\!\left[\Phi_\zeta(q_k,X_{q_k})\right]\right|\leq\ve.
		\end{equation}
	\end{definition}
	\begin{remark}
		The event $\Lawmatch_{n-1}$ enforces that, for each fixed $k\le n-1$ and for each test
		function $\chi$ in \eqref{def:chi1}--\eqref{def:chi5}, the $\chi(m^{(n-1)}_i)$--weighted
		empirical averages of the coordinates match the corresponding moments under the law of
		$(\partial_x\Phi_\zeta(q_1,X_{q_1}),\ldots,\partial_x\Phi_\zeta(q_{n-1},X_{q_{n-1}}))$, up to error $\varepsilon$
		as in \eqref{def:Linear}--\eqref{def:Linear4}.
	\end{remark}
	
	We now state the analogue of Proposition \ref{prop:annealed 0}.

	\begin{proposition}[Conditional annealed complexity]\label{prop:annealed cond}
		Let $n\ge2$. Let $\bm{u}=(u_1,\ldots,u_{n})$,
		$\bm{q}=(q_1,\ldots,q_n)$ be strictly increasing
		in~$(0,1)$, and $\zeta\in\Prefix_{n+1}(\bm{u};\bm{q})$.
		Define the \emph{prefix free energy}
		\begin{equation}\label{def:Ppar}
			\Ppar^{(n)}(\bm{u};\bm{q})
			\;:=\;
			\inf_{\zeta'\,\in\,\Prefix_{n+1}(\bm{u};\bm{q})}
			\Pari(\zeta'),
		\end{equation}
		and the \emph{$n$-step complexity functional at level $f$}
		\begin{equation}\label{def:Cfn}
			\Cpx^{(n)}_f(\bm{u};\bm{q})
			\;:=\;
			u_n\,\bigl(\Ppar^{(n)}(\bm{u};\bm{q})-f\bigr).
		\end{equation}
		
		We assume:
		\begin{enumerate}
			\item[\emph{(i)}]
			\emph{(Tail optimality)} The infimum $\Ppar^{(n)}(\bm{u};\bm{q})$ is
			attained at $\zeta$.
			\item[\emph{(ii)}]
			\emph{(Stationarity)} $(\bm{u};\bm{q})$ is a critical point of
			$\Cpx^{(n)}_f$.
		\end{enumerate}
		Let $\ve\in(0,1)$, let
		$(\bfm^{(1)},\ldots,\bfm^{(n-1)})\in\SUSY_{n-1}(\zeta,\ve)\cap\Lawmatch_{n-1}(\zeta,\ve)$. Let
		$\mc{N}_\ve(f,\zeta,q_n,\bfm^{(n-1)})$ be the number of critical points $\bfm$
		of $F_{\TAP,\zeta}$ with
		$(\bfm^{(1)},\ldots,\bfm^{(n-1)},\bfm)\in\SUSY_n(\zeta,\ve)$ and
		$\tfrac{1}{N}F_{\TAP,\zeta}(\bfm)\in(f-\ve,f+\ve)$. Set
		\begin{equation*}
			\Skel_{n-1}=\left(\nabla F_{\TAP,\zeta}(\bfm^{(1)}),F_{\TAP,\zeta}(\bfm^{(1)}),\ldots,\nabla F_{\TAP,\zeta}(\bfm^{(n-1)}), F_{\TAP,\zeta}(\bfm^{(n-1)})\right).
		\end{equation*} Then
		\begin{multline}\label{eq:fnfn}
			\lim_{\ve\to0}\lim_{N\to\infty}
			\frac{1}{N}\log\dE\Bigl[
			\mc{N}_\ve(f,\zeta,q_n,\bfm^{(n-1)})
			\;\Big|\;
			\Skel_{n-1}=(0,N\Pari(\zeta),\ldots,0,N\Pari(\zeta))
			\Bigr]
			=
			\Cpx_f^{(n)}(\bm{u};\bm{q})
			\\=
			\zeta([0,q_n))\,(\Pari(\zeta)-f).
		\end{multline}
	\end{proposition}

	\begin{remark}[Consistency with the base case]
		For $n=1$, one has $\bm{u}=(u)$, $\bm{q}=(q)$, $u_{n}=u=\zeta(\{0\})$, so $\Cpx^{(n)}_f(\bm{u};\bm{q})=\Cpx_f(u;q)$ and
		\eqref{eq:fnfn} recovers the conclusion of
		Proposition~\ref{prop:annealed 0}.
	\end{remark}

	\begin{remark}[On the choice of the free energy level of the ancestors]
		Differentiating \eqref{def:Cfn} with respect to $u_l=\zeta([0,q_l))$ for $l\in [n-1]$ forces the TAP free energy of the ancestors to equal $\Pari(\zeta)$ (see \eqref{eq:optfzeta}), while differentiating with respect to $u_n$ gives the last-state condition \eqref{eq:optf}, involving $f$. Moreover, taking $f=\Pari(\zeta)$ in \eqref{eq:fnfn} gives a sub-extensive complexity (i.e., not growing exponentially).
		
		At the thermodynamic equilibrium ($f=\Pari(\zeta)$), all TAP states have the same free energy, equal to the thermodynamic free energy, which is consistent with \cite[Theorem 4]{chen2018generalized}. 
	\end{remark}

	\subsection{Computing the conditional probability}
	
	In this subsection, we compute the joint Gaussian density of the gradients and values of $F_{\TAP,\zeta}$ along the hierarchy $(\bfm^{(1)},\ldots,\bfm^{(n)})$, extending Lemma~\ref{lemma:prob0 non deg} to $n$ states.

	\begin{lemma}\label{lemma:prob0 replica}
		Let $\ve>0$. Let $\zeta$ be a probability measure on $[0,1]$ with prefix $\{0,q_1,\ldots,q_n\}$ where $0<q_1<\ldots<q_n<1$. Let us emphasize that $\supp(\zeta)\cap (q_n,1]$ can be nonempty. For every $j\in [n-1]$, set $\Delta_j:=-\zeta(\{q_j\})$ and set $\Delta_n=\zeta([0,q_n))=\zeta([0,q_{n-1}])$. Let $f_1,\ldots,f_n\in \dR$. For every $\bfm\in [-1,1]^N$, denote $Z(\bfm):=(\nabla F_{\TAP,\zeta}(\bfm),F_{\TAP,\zeta}(\bfm))$.

		Let $(\bfm^{(1)},\ldots,\bfm^{(n)})\in \SUSY_n(\zeta,\ve)$ where $\SUSY_n(\zeta,\ve)$ is as in Definition \ref{def:susyreplicas}. Moreover, assume that for every $k\in [n-1]$, 
		\begin{multline}\label{eq:ass ancestors}
			\sum_{i=1}^N (h_\zeta(q_k,m_i^{(k)})-\partial_m h_\zeta(q_k,m_i^{(k)})m_i^{(k)})+\frac{1}{2}N\int_0^1 t\xi''(t)\zeta([0,t])\dd t+\frac{1}{2}N\int_{0}^{q_k} t\xi''(t)\zeta([0,t]) \dd t\\=-Nf_k+O(N\ve^{1/2}).
		\end{multline}
		For every $k\in [n]$, let $\Skel_k:=(Z(\bfm^{(1)}),\ldots,Z(\bfm^{(k)}))$.

		Then, 
		\begin{align}\label{eq:prob ineq}
			& \log \varphi_{\Skel_n\mid \Skel_{n-1}}(0,Nf_1,\ldots,0,Nf_n\mid 0,Nf_1,\ldots,0,Nf_{n-1})\\ \notag
			&\leq -\frac{1}{2}N\xi''(q_n)\left(\int_{q_n}^1 \zeta([0,t])\dd t\right)^2\\ \notag
			&-\frac{1}{2(\xi'(q_n)-\xi'(q_{n-1}))}\left\Vert \partial_m h_\zeta(q_n,\bfm^{(n)})-\partial_m h_\zeta(q_{n-1},\bfm^{(n-1)})\right\Vert^2\\ \notag
			& -\Delta_n\sum_{i=1}^N \bigl(h_\zeta(q_n,m_i^{(n)})-\partial_m h_\zeta(q_n,m_i^{(n)})m_i^{(n)}\bigr)\\ \notag
			&\quad +\Delta_n \sum_{i=1}^N \bigl(h_\zeta(q_{n-1},m_i^{(n-1)})- \partial_m h_\zeta(q_{n-1},m_i^{(n-1)})m_i^{(n-1)}\bigr)\\ \notag
			&- N\Delta_n(f_n\!-\!f_{n-1}) -\tfrac{N}{2} \log\bigl(2\pi (\xi'(q_n)\!-\!\xi'(q_{n-1}))\bigr) +O(N\ve^{1/2}).
		\end{align}

		Suppose in addition that 
		\begin{multline}\label{eq:addiass}
			\sum_{i=1}^N \bigl(h_\zeta(q_n,m_i^{(n)})-\partial_m h_\zeta(q_n,m_i^{(n)})m_i^{(n)}\bigr)\\
			+\tfrac{1}{2}N\!\int_0^1 t\xi''(t)\zeta([0,t])\dd t+\tfrac{1}{2}N \!\int_{0}^{q_n} t\xi''(t)\zeta([0,t]) \dd t=-Nf_n+O(N\ve^{1/2}).
		\end{multline}
		Then the inequality \eqref{eq:prob ineq} is an equality.
	\end{lemma}

	We begin with a short auxiliary computation.
	\begin{lemma}\label{lemma:compression2}
		Let $\zeta$ be a probability measure on $[0,1]$ with an $(n+1)$-atom prefix $\{0,q_1,\ldots,q_n\}$, where $0<q_1<\cdots<q_n<1$.
		Define $q_{ij}:=q_{\min(i,j)}$ for $i,j\in[n]$, and set
		\begin{equation*}
			\Delta_k :=
			\begin{cases}
				- \zeta(\{q_k\}) & \text{if } k \in [n-1],\\
				\zeta([0,q_n)) & \text{if } k = n.
			\end{cases}
		\end{equation*}
		Let $(\bfm^{(1)},\ldots,\bfm^{(n)})\in \SUSY_n(\zeta,\ve)$, where $\SUSY_n(\zeta,\ve)$ is as in Definition \ref{def:susyreplicas}, and define
		\begin{equation*}
			z^{(i)} :=- \sum_{k=1}^n \xi'(q_{ik})\, \Delta_k\, \bfm^{(k)},
			\qquad i \in [n].
		\end{equation*}
		Then for every $i,j\in [n]$,
		\begin{equation*}
			\langle z^{(i)}+\partial_m h_\zeta(q_i,\bfm^{(i)}),\,\bfm^{(j)}\rangle
			= N\,\xi'(q_{ij})\int_{q_j}^1 \zeta([0,t])\,\dd t+O(N\ve).
		\end{equation*}
	\end{lemma}
	\begin{proof}
		We verify the identity by treating the diagonal and off-diagonal cases separately.
		
		\medskip
		\noindent\textit{Diagonal case: $i = j$.}\enspace
		Observe that $\sum_{k=i}^n \Delta_k = \zeta([0,q_i))$.
		Expanding the inner product and using the overlap structure of $\SUSY_n(\zeta,\ve)$ yields
		\begin{multline*}
			\langle z^{(i)},\bfm^{(i)}\rangle
			= N\sum_{k=1}^{i-1}\xi'(q_k)\,q_k\,\zeta(\{q_k\})
			- N\,\xi'(q_i)\,q_i\,\zeta\bigl([0,q_i)\bigr)\\
			= N\!\int_0^{q_i}\xi'(t)\,t\,\dd\zeta(t)
			- N\,\xi'(q_i)\,q_i\,\zeta\bigl([0,q_i]\bigr)+O(N\ve).
		\end{multline*}
		Combining this with the definition of $\SUSY_n(\zeta,\ve)$, we obtain
		\begin{equation}\label{eq:diag}
			\langle z^{(i)}+\partial_m h_\zeta(q_i,\bfm^{(i)}),\,\bfm^{(i)}\rangle
			= N\xi'(q_i)\int_{q_i}^1 \zeta([0,t])\,\dd t+O(N\ve).
		\end{equation}
		
		\medskip
		\noindent\textit{Off-diagonal case: $j \le i$.}\enspace
		Since $q_{ik} = q_k$ for $k \le i$ and $q_{jk} = q_{\min(j,k)}$, the inner product evaluates to
		\begin{multline*}
			\langle z^{(i)},\bfm^{(j)}\rangle
			= -N\sum_{k=1}^{i}\xi'(q_k)\,\Delta_k\,q_{jk}
			- N\,\xi'(q_i)\,q_j\,\zeta\bigl([0,q_i]\bigr)
			\\ = N\int_0^{q_i}\xi'(u)\,\min(u,q_j)\,\dd\zeta(u)
			- N\,\xi'(q_i)\,q_j\,\zeta\bigl([0,q_i]\bigr)+O(N\ve).
		\end{multline*}
		Invoking the definition of $\SUSY_n(\zeta,\ve)$ then gives
		\begin{equation*}
			\langle z^{(i)}+\partial_m h_\zeta(q_i,\bfm^{(i)}),\,\bfm^{(j)}\rangle
			= N\,\xi'(q_j)\int_{q_j}^1 \zeta([0,t])\,\dd t+O(N\ve),
		\end{equation*}
		which is the desired identity since $q_{ij} = q_j$ when $j \le i$.
		
		\medskip
		\noindent\textit{Off-diagonal case: $j \ge i$.}\enspace
		We instead compute the difference
		\begin{equation*}
			\langle z^{(i)},\bfm^{(j)}-\bfm^{(i)}\rangle
			= N\,\xi'(q_i)\biggl(\sum_{k=i+1}^j
			\zeta(\{q_k\})(q_k-q_i) - (q_j-q_i)\,\zeta\bigl([0,q_j]\bigr)\biggr)+O(N\ve).
		\end{equation*}
		Applying Stieltjes integration by parts on $[q_i,q_j]$,
		\begin{equation*}
			\int_{(q_i,q_j]}(t-q_i)\,\dd\zeta(t)
			+ \int_{q_i}^{q_j}\zeta([0,t])\,\dd t
			= (q_j-q_i)\,\zeta\bigl([0,q_j]\bigr),
		\end{equation*}
		so the parenthetical expression above simplifies to
		$-\int_{q_i}^{q_j}\zeta([0,t])\,\dd t$. Therefore
		\begin{equation*}
			\langle z^{(i)},\bfm^{(j)}-\bfm^{(i)}\rangle
			= -N\,\xi'(q_i)\int_{q_i}^{q_j}\zeta([0,t])\,\dd t+O(N\ve),
		\end{equation*}
		and by the definition of $\SUSY_n(\zeta,\ve)$ and \eqref{eq:deltaM X},
		the same relation holds with $z^{(i)}$ replaced by
		$z^{(i)}+\partial_m h_\zeta(q_i,\bfm^{(i)})$.
		Adding this to the diagonal identity~\eqref{eq:diag} yields
		\begin{multline*}
			\langle z^{(i)}+\partial_m h_\zeta(q_i,\bfm^{(i)}),\,\bfm^{(j)}\rangle
			= N\,\xi'(q_i)\int_{q_i}^1 \zeta([0,t])\,\dd t
			- N\,\xi'(q_i)\int_{q_i}^{q_j}\zeta([0,t])\,\dd t
			\\ = N\,\xi'(q_i)\int_{q_j}^1 \zeta([0,t])\,\dd t+O(N\ve),
		\end{multline*}
		which is the claim since $q_{ij} = q_i$ when $i \le j$.
	\end{proof}

	\medskip
	
	We now turn to the proof of Lemma \ref{lemma:prob0 replica}.

	\begin{proof}[Proof of Lemma \ref{lemma:prob0 replica}]
		We prove the result assuming that all the constraints in $\SUSY_n(\zeta,\ve)$ hold exactly (i.e., with $\ve=0$); the general case follows by a continuity argument analogous to the one in the proof of Lemma~\ref{lemma:prob0 non deg}, which introduces errors of order $O(N\ve/\delta_\ve)=O(N\ve^{1/2})$ (since $\delta_\ve=\ve^{1/2}$).
		
		Set $q_0=0$ and, for $i,j\in[n]$, set $q_{ij}:=q_{\min(i,j)}$.
		
		\paragraph{\bf{Step 1: the convex dual identity}}
		By Lemma~\ref{lemma:gradient TAP}, \eqref{eq:gradTap-stability}, for every $k\in[n]$,
		\begin{equation*}
			\nabla F_{\TAP,\zeta}(\bfm^{(k)})=\nabla H(\bfm^{(k)})-k_\zeta(q_k,\bfm^{(k)})+O(\ve),
		\end{equation*}
		where $k_\zeta$ is as in \eqref{def:kqm} and the error term is understood componentwise. Moreover, recall that
		\begin{equation*}
			F_{\TAP,\zeta}(\bfm^{(k)})=H(\bfm^{(k)})-\sum_{i=1}^N h_\zeta(q_k,m_i^{(k)})-N\mc U_\zeta(q_k),
		\end{equation*}
		where $h_\zeta$ is as in \eqref{def:hqm} and $\mc U_\zeta$ as in \eqref{def:mcUzeta}. Set
		\begin{equation*}
			\Gamma_n:=\Cov\bigl[\nabla F_{\TAP,\zeta}(\bfm^{(1)}),\ldots,\nabla F_{\TAP,\zeta}(\bfm^{(n)})\bigr],
			\qquad
			\Gamma_n':=\Cov[\Skel_n].
		\end{equation*}
		Let $a_n=(a_n^{(1)},\ldots,a_n^{(n)})$ where, for every $k\in[n]$,
		\begin{equation*}
			a_n^{(k)}
			:=\Bigl(k_\zeta(q_k,\bfm^{(k)}),\;Nf_k+\sum_{i=1}^N h_\zeta(q_k,m_i^{(k)})+N\mc U_\zeta(q_k)\Bigr)\in\dR^{N+1}.
		\end{equation*}
		By the above, $\Skel_n=(Z(\bfm^{(1)}),\ldots,Z(\bfm^{(n)}))$ is a Gaussian vector of covariance matrix~$\Gamma_n'$ (positive definite since $\xi$ is a mixed model; see the proof of Lemma~\ref{lemma:prob0 non deg}); hence
		\begin{equation}\label{eq:log density n}
			\log \varphi_{\Skel_n}(0,Nf_1,\ldots,0,Nf_n)
			=-\tfrac12\langle a_n,(\Gamma_n')^{-1}a_n\rangle-\tfrac12\log\det(2\pi\Gamma_n').
		\end{equation}
		Using the convex dual identity
		\begin{equation*}
			-\langle a_n,(\Gamma_n')^{-1}a_n\rangle
			=\inf_{w\in\dR^{n(N+1)}}\bigl\{\langle w,\Gamma_n'w\rangle-2\langle w,a_n\rangle\bigr\},
		\end{equation*}
		we write $w=(x^{(1)},v_1,\ldots,x^{(n)},v_n)$ with $x^{(k)}\in\dR^N$ and $v_k\in\dR$, so that
		\begin{equation}\label{eq:infGamma'}
			-\langle a_n,(\Gamma_n')^{-1}a_n\rangle
			=\inf_{\substack{x^{(k)}\in\dR^N,\,v_k\in\dR\\k\in[n]}}
			\bigl(\langle w,\Gamma_n'w\rangle-2\langle w,a_n\rangle\bigr).
		\end{equation}
		Observe next that the covariance identities for $H$, together
		with the hierarchical orthogonality
		\[
		\langle \bfm^{(j)},\bfm^{(i+1)}-\bfm^{(i)}\rangle=0
		\quad\text{for }j\le i
		\qquad(\text{so that }\langle \bfm^{(i)},\bfm^{(j)}\rangle=Nq_{ij}),
		\]
		yield
		\begin{equation*}
			\langle w,\Gamma_n'w\rangle
			=\langle x,\Gamma_n x\rangle
			+N\sum_{i,j=1}^n \xi(q_{ij})v_iv_j
			+2\sum_{i,j=1}^n \xi'(q_{ij})v_j\langle x^{(i)},\bfm^{(j)}\rangle.
		\end{equation*}
		Hence, taking $v_i:=\Delta_i$ for every $i\in[n]$ in \eqref{eq:infGamma'} yields
		\begin{equation}\label{eq:boundzngn}
			\begin{aligned}
				-\langle a_n,(\Gamma_n')^{-1}a_n\rangle
				\le\;&
				N\sum_{i,j=1}^n \xi(q_{ij})\Delta_i\Delta_j
				-2\sum_{i=1}^n \Delta_i\Bigl(\sum_{k=1}^N h_\zeta(q_i,m_k^{(i)})+N\mc U_\zeta(q_i)+Nf_i\Bigr)\\
				&-N\sum_{i=1}^n \xi''(q_i)\Bigl(\int_{q_i}^1 \zeta([0,t])\,\dd t\Bigr)^2\\
				&+\inf_{x\in\mc D_n}\Bigl(
				\sum_{i,j=1}^n \xi'(q_{ij})\langle x^{(i)},x^{(j)}\rangle
				-2\sum_{i=1}^n\Bigl\langle \partial_m h_\zeta(q_i,\bfm^{(i)})-\sum_{j=1}^n \xi'(q_{ij})\Delta_j\bfm^{(j)},\,x^{(i)}\Bigr\rangle
				\Bigr),
			\end{aligned}
		\end{equation}
		where
		\begin{equation*}
			\mc D_n:=\Bigl\{x\in(\dR^N)^n:\ \forall\,i,j\in[n],\
			\langle x^{(i)},\bfm^{(j)}\rangle=N\mathbf 1_{\{i=j\}}\int_{q_i}^1 \zeta([0,t])\,\dd t\Bigr\}.
		\end{equation*}

		Define, for $i\in[n]$,
		\begin{equation}\label{def:y(k)}
			y_n^{(i)}:=\partial_m h_\zeta(q_i,\bfm^{(i)})-\sum_{j=1}^n \xi'(q_{ij})\Delta_j\bfm^{(j)},
			\qquad y_n^{(0)}:=0.
		\end{equation}
		Then the infimum in \eqref{eq:boundzngn} can be rewritten as
		\begin{equation}\label{eq:xDn}
			\inf_{x\in\mc D_n}\bigl(\langle x,(\xi'(Q)\otimes I_N)x\rangle-2\langle x,y_n\rangle\bigr),
			\qquad Q=(q_{ij})_{i,j\in[n]}.
		\end{equation}
		Notice that, since $\sum_{k=i}^n \Delta_k=\zeta([0,q_i))$,
		\begin{equation*}
			\sum_{i,j=1}^n\xi(q_{ij})\Delta_i\Delta_j
			=\sum_{i=1}^n(\xi(q_i)-\xi(q_{i-1}))\Bigl(\sum_{k=i}^n \Delta_k\Bigr)^2
			=\sum_{i=1}^n(\xi(q_i)-\xi(q_{i-1}))\zeta([0,q_i))^2.
		\end{equation*}
		
		\paragraph{\bf{Step 2: optimization over $x$}}
		Introduce the change of variables $z^{(k)}:=\sum_{i=k}^n x^{(i)}$. Then
		\begin{equation*}
			\langle x,(\xi'(Q)\otimes I_N)x\rangle
			=\sum_{k=1}^n (\xi'(q_k)\!-\!\xi'(q_{k-1})) \Vert z^{(k)}\Vert^2,
			\qquad
			\langle x,y_n\rangle
			=\sum_{k=1}^n \langle z^{(k)},y_n^{(k)}-y_n^{(k-1)}\rangle.
		\end{equation*}
		Hence, the unconstrained infimum satisfies
		\begin{equation*}
			\inf_{x\in(\dR^N)^n}\bigl(\langle x,(\xi'(Q)\otimes I_N)x\rangle-2\langle x,y_n\rangle\bigr)
			=-\sum_{k=1}^n \frac{\Vert y_n^{(k)}-y_n^{(k-1)}\Vert^2}{\xi'(q_k)-\xi'(q_{k-1})},
		\end{equation*}
		and the unique minimizer is given by
		\begin{equation}\label{def:zn}
			z_n^{(k)}:=\frac{y_n^{(k)}-y_n^{(k-1)}}{\xi'(q_k)-\xi'(q_{k-1})},\qquad k\in[n].
		\end{equation}
		Let $L\in\dR^{n\times n}$ be the lower-triangular matrix $L_{ik}:=\indic_{\{k\le i\}}$, so that $z^{(k)}=\sum_{i=k}^n x^{(i)}=(L^T x)^{(k)}$.
		Let $x_n:=((L^{-1})^T\otimes I_N)z_n$, i.e.\ $x_n^{(k)}=z_n^{(k)}-z_n^{(k+1)}$ with $z_n^{(n+1)}:=0$.
		We claim that $x_n\in\mc D_n$, so the restriction to $\mc D_n$ in \eqref{eq:xDn} is not lossy.
		
		By Lemma~\ref{lemma:compression2}, for every $i,j\in[n]$,
		\begin{equation}\label{eq:co}
			\langle y_n^{(i)},\bfm^{(j)}\rangle
			=N\xi'(q_{ij})\int_{q_j}^1 \zeta([0,t])\,\dd t.
		\end{equation}
		Fix $k\in[n]$.
		If $j<k$, then $q_{kj}=q_{k-1,j}=q_j$, so \eqref{eq:co} gives
		$\langle y_n^{(k)}-y_n^{(k-1)},\bfm^{(j)}\rangle=0$.
		If $j\ge k$, then $q_{kj}=q_k$ and $q_{k-1,j}=q_{k-1}$, so
		$\langle y_n^{(k)}-y_n^{(k-1)},\bfm^{(j)}\rangle
		=N(\xi'(q_k)-\xi'(q_{k-1}))\int_{q_j}^1 \zeta([0,t])\,\dd t$.
		Therefore, by \eqref{def:zn},
		\begin{equation*}
			\langle z_n^{(k)},\bfm^{(j)}\rangle
			=
			\begin{cases}
				0,& j<k,\\[2pt]
				N\displaystyle\int_{q_j}^1 \zeta([0,t])\,\dd t,& j\ge k.
			\end{cases}
		\end{equation*}
		Since $x_n^{(k)}=z_n^{(k)}-z_n^{(k+1)}$, it follows that for every $j\in[n]$,
		\begin{equation*}
			\langle x_n^{(k)},\bfm^{(j)}\rangle
			=
			\begin{cases}
				0,& j\ne k,\\[2pt]
				N\displaystyle\int_{q_k}^1 \zeta([0,t])\,\dd t,& j=k,
			\end{cases}
		\end{equation*}
		i.e.\ $x_n\in\mc D_n$, as claimed.
		
		Inserting the value of the infimum into \eqref{eq:boundzngn} then yields
		\begin{equation}\label{eq:u}
			\begin{aligned}
				-\langle a_n,(\Gamma_n')^{-1}a_n\rangle
				\le\;&
				N\sum_{i=1}^n(\xi(q_i)-\xi(q_{i-1}))\zeta([0,q_i))^2
				-2\sum_{i=1}^n \Delta_i\Bigl(\sum_{k=1}^N h_\zeta(q_i,m_k^{(i)})+N\mc U_\zeta(q_i)+Nf_i\Bigr)\\
				&-N\sum_{i=1}^n \xi''(q_i)\Bigl(\int_{q_i}^1 \zeta([0,t])\,\dd t\Bigr)^2
				-\sum_{i=1}^n\frac{\Vert y_n^{(i)}-y_n^{(i-1)}\Vert^2}{\xi'(q_i)-\xi'(q_{i-1})}.
			\end{aligned}
		\end{equation}
		
		\paragraph{\bf{Step 3: equality in \eqref{eq:u} under \eqref{eq:addiass}}}
		Suppose that \eqref{eq:addiass} holds. In view of the computations in Steps~1--2, the inequality \eqref{eq:u} is an equality if and only if $w=(x_n^{(1)},\Delta_1,\ldots,x_n^{(n)},\Delta_n)$ is the minimizer of $w'\mapsto \frac{1}{2}\langle w',\Gamma_n'w'\rangle-\langle w',a_n\rangle$. Since the objective in \eqref{eq:infGamma'} is strictly convex, this holds if and only if $\Gamma_n'w=a_n$, which is equivalent to
		\begin{equation}\label{eq:conditions}
			\begin{cases}
				\langle x_n^{(i)},\bfm^{(j)}\rangle=N\mathbf 1_{\{i=j\}}\displaystyle\int_{q_i}^1 \zeta([0,t])\,\dd t,\qquad & i,j\in[n],\\[6pt]
				\xi'(q_k)\langle x_n^{(k)},\bfm^{(k)}\rangle
				+N\displaystyle\sum_{j=1}^n\xi(q_{kj})\Delta_j
				-\displaystyle\sum_{i=1}^N h_\zeta(q_k,m_i^{(k)})-N\mc U_\zeta(q_k)=Nf_k,\qquad & k\in[n].
			\end{cases}
		\end{equation}
		The first condition is precisely $x_n\in\mc D_n$, which was established in Step~2. It remains to verify the second line. Fix $k\in[n]$ and write
		\begin{equation*}
			\sum_{i=1}^N h_\zeta(q_k,m_i^{(k)})
			=\sum_{i=1}^N\bigl(h_\zeta(q_k,m_i^{(k)})-\partial_m h_\zeta(q_k,m_i^{(k)})m_i^{(k)}\bigr)
			+\langle \partial_m h_\zeta(q_k,\bfm^{(k)}),\bfm^{(k)}\rangle.
		\end{equation*}
		By $x_n\in\mc D_n$ we have $\xi'(q_k)\langle x_n^{(k)},\bfm^{(k)}\rangle=N\xi'(q_k)\int_{q_k}^1\zeta([0,t])\,\dd t$. Moreover, taking $i=j=k$ in \eqref{eq:co} and subtracting
		$\bigl\langle \sum_{j=1}^n\xi'(q_{kj})\Delta_j\bfm^{(j)},\,\bfm^{(k)}\bigr\rangle$ from both sides of \eqref{eq:co} gives
		\begin{equation*}
			\langle \partial_m h_\zeta(q_k,\bfm^{(k)}),\bfm^{(k)}\rangle
			=N\xi'(q_k)\int_{q_k}^1\zeta([0,t])\,\dd t
			-N\sum_{j=1}^{k-1}\xi'(q_j)q_j\zeta(\{q_j\})
			+N\xi'(q_k)q_k\zeta([0,q_k)).
		\end{equation*}
		Observe next that
		\begin{equation*}
			\sum_{j=1}^n \xi(q_{kj})\Delta_j
			=-\sum_{j=1}^{k-1}\xi(q_j)\zeta(\{q_j\})+\xi(q_k)\sum_{j=k}^n \Delta_j
			=-\sum_{j=1}^{k-1}\xi(q_j)\zeta(\{q_j\})+\xi(q_k)\zeta([0,q_k)).
		\end{equation*}
		Combining the last three displays yields
		\begin{multline*}
			\xi'(q_k)\langle x_n^{(k)},\bfm^{(k)}\rangle
			+N\sum_{j=1}^n \xi(q_{kj})\Delta_j
			-\sum_{i=1}^N h_\zeta(q_k,m_i^{(k)})-N\mc U_\zeta(q_k)\\
			\qquad=
			-\sum_{i=1}^N\bigl(h_\zeta(q_k,m_i^{(k)})-\partial_m h_\zeta(q_k,m_i^{(k)})m_i^{(k)}\bigr)
			-N\zeta([0,q_k))\!\int_0^{q_k}\!t\xi''(t)\,\dd t
			+N\sum_{j=1}^{k-1}\zeta(\{q_j\})\!\int_0^{q_j}\!t\xi''(t)\,\dd t\\
			-N\mc U_\zeta(q_k).
		\end{multline*}
		Using Stieltjes integration by parts,
		\begin{equation*}
			\sum_{j=1}^{k-1}\zeta(\{q_j\})\int_0^{q_j} t\xi''(t)\,\dd t
			=\zeta([0,q_k))\int_0^{q_k} t\xi''(t)\,\dd t-\int_0^{q_k}\zeta([0,t])\,t\xi''(t)\,\dd t,
		\end{equation*}
		we obtain
		\begin{multline*}
			\xi'(q_k)\langle x_n^{(k)},\bfm^{(k)}\rangle
			+N\sum_{j=1}^n \xi(q_{kj})\Delta_j
			-\sum_{i=1}^N h_\zeta(q_k,m_i^{(k)})-N\mc U_\zeta(q_k)
			\\=
			-\sum_{i=1}^N\bigl(h_\zeta(q_k,m_i^{(k)})-\partial_m h_\zeta(q_k,m_i^{(k)})m_i^{(k)}\bigr)
			-\frac N2\!\int_0^1\! t\xi''(t)\zeta([0,t])\,\dd t
			-\frac N2\!\int_0^{q_k}\! t\xi''(t)\zeta([0,t])\,\dd t.
		\end{multline*}
		For $k\le n-1$ this equals $Nf_k$ by \eqref{eq:ass ancestors}, and for $k=n$ it equals $Nf_n$ by \eqref{eq:addiass}. Thus the second condition in \eqref{eq:conditions} holds, and therefore $\Gamma_n'w=a_n$, so that \eqref{eq:u} is an equality:
		\begin{equation}\label{eq:Goodn}
			\begin{aligned}
				-\langle a_n,(\Gamma_n')^{-1}a_n\rangle
				=\;& N\sum_{i=1}^n(\xi(q_i)-\xi(q_{i-1}))\zeta([0,q_i))^2
				-2\sum_{i=1}^n \Delta_i\Bigl(\sum_{k=1}^N h_\zeta(q_i,m_k^{(i)})+N\mc U_\zeta(q_i)+Nf_i\Bigr)\\
				&-N\sum_{i=1}^n \xi''(q_i)\Bigl(\int_{q_i}^1 \zeta([0,t])\,\dd t\Bigr)^2
				-\sum_{i=1}^n\frac{\Vert y_n^{(i)}-y_n^{(i-1)}\Vert^2}{\xi'(q_i)-\xi'(q_{i-1})}.
			\end{aligned}
		\end{equation}
		
		\paragraph{\bf{Step 4: conclusion}}
		Let $a_{n-1}$, $\Gamma_{n-1}'$ be defined as above with $\Skel_{n-1}$ in place of $\Skel_n$.
		Since the conditional density of a Gaussian vector is Gaussian,
		\begin{multline*}
			\log \varphi_{\Skel_n\mid\Skel_{n-1}}(0,Nf_1,\ldots,0,Nf_n\mid 0,Nf_1,\ldots,0,Nf_{n-1})
			\\=\log\varphi_{\Skel_n}(0,Nf_1,\ldots,0,Nf_n)-\log\varphi_{\Skel_{n-1}}(0,Nf_1,\ldots,0,Nf_{n-1}).
		\end{multline*}
		In view of \eqref{eq:log density n}, it suffices to compare the quadratic forms and the determinants at levels $n$ and $n-1$.
		
		We begin with the quadratic forms. For the $(n\!-\!1)$-level computation, the relevant weights are
		\begin{equation*}
			\tilde \Delta_j:=-\zeta(\{q_j\})\quad (1\le j\le n-2),
			\qquad
			\tilde \Delta_{n-1}:=\zeta([0,q_{n-1}))=\Delta_{n-1}+\Delta_n,
		\end{equation*}
		and we define $y_{n-1}^{(i)}$ by \eqref{def:y(k)} with $n$ replaced by $n-1$ and $\Delta$ replaced by $\tilde \Delta$.
		Since \eqref{eq:ass ancestors} holds up to $k=n-1$, Step~3 applied at level $n-1$ yields the equality
		\begin{equation}\label{eq:Goodn-1}
			\begin{aligned}
				-\langle a_{n-1},(\Gamma_{n-1}')^{-1}a_{n-1}\rangle
				=\;&
				N\sum_{i=1}^{n-1}(\xi(q_i)-\xi(q_{i-1}))\zeta([0,q_i))^2
				-2\sum_{i=1}^{n-1} \tilde \Delta_i\Bigl(\sum_{k=1}^N h_\zeta(q_i,m_k^{(i)})+N\mc U_\zeta(q_i)+Nf_i\Bigr)\\
				&-N\sum_{i=1}^{n-1} \xi''(q_i)\Bigl(\int_{q_i}^1 \zeta([0,t])\,\dd t\Bigr)^{\!2}\\
				&-\sum_{i=1}^{n-1}
				\frac{\| y_{n-1}^{(i)}\!-\!y_{n-1}^{(i-1)}\|^2}
				{\xi'(q_i)\!-\!\xi'(q_{i-1})}.
			\end{aligned}
		\end{equation}
		Combining the inequality \eqref{eq:u} at level $n$ with the equality \eqref{eq:Goodn-1} at level $n-1$, we get
		\begin{equation*}
			-\frac12\langle a_n,(\Gamma_n')^{-1}a_n\rangle+\frac12\langle a_{n-1},(\Gamma_{n-1}')^{-1}a_{n-1}\rangle
			\le \frac12\bigl(\mathrm{RHS}_n-\mathrm{RHS}_{n-1}\bigr),
		\end{equation*}
		where $\mathrm{RHS}_n$ and $\mathrm{RHS}_{n-1}$ denote the right-hand sides of \eqref{eq:u} and \eqref{eq:Goodn-1} respectively.
		All terms with index $\le n-2$ cancel. Using $\tilde \Delta_{n-1}=\Delta_{n-1}+\Delta_n$, we obtain
		\begin{multline*}
			\mathrm{RHS}_n-\mathrm{RHS}_{n-1}
			= N(\xi(q_n)-\xi(q_{n-1}))\zeta([0,q_n))^2
			- N\xi''(q_n)\Bigl(\int_{q_n}^1\zeta([0,t])\,\dd t\Bigr)^2\\
			-2\Delta_n\Bigl(\sum_{i=1}^N h_\zeta(q_n,m_i^{(n)})+N\mc U_\zeta(q_n)+Nf_n
			-\sum_{i=1}^N h_\zeta(q_{n-1},m_i^{(n-1)})-N\mc U_\zeta(q_{n-1})-Nf_{n-1}\Bigr)\\
			-\Biggl[\sum_{i=1}^n\frac{\Vert y_n^{(i)}-y_n^{(i-1)}\Vert^2}{\xi'(q_i)\!-\!\xi'(q_{i-1})}
			-\sum_{i=1}^{n-1}\frac{\Vert y_{n-1}^{(i)}-y_{n-1}^{(i-1)}\Vert^2}{\xi'(q_i)\!-\!\xi'(q_{i-1})}\Biggr].
		\end{multline*}
		
		We now compute the difference of the squared-norm sums. From the definitions of $y_n$ and $y_{n-1}$ and the identity $\tilde \Delta_{n-1}=\Delta_{n-1}+\Delta_n$, for every $i\le n-1$,
		\begin{multline*}
			y_n^{(i)}=y_{n-1}^{(i)}-\xi'(q_i)\Delta_n\,(\bfm^{(n)}-\bfm^{(n-1)}),
			\qquad\text{hence}\\
			y_n^{(i)}-y_n^{(i-1)}=\bigl(y_{n-1}^{(i)}-y_{n-1}^{(i-1)}\bigr)-(\xi'(q_i)\!-\!\xi'(q_{i-1})) \Delta_n\,(\bfm^{(n)}-\bfm^{(n-1)}).
		\end{multline*}
		Using $\sum_{i=1}^{n-1}(y_{n-1}^{(i)}-y_{n-1}^{(i-1)})=y_{n-1}^{(n-1)}$,
		the telescoping identity $\sum_{i=1}^{n-1}(\xi'(q_i)-\xi'(q_{i-1}))=\xi'(q_{n-1})$,
		and $\|\bfm^{(n)}-\bfm^{(n-1)}\|^2=N(q_n-q_{n-1})$, we compute
		\begin{multline*}
			\sum_{i=1}^{n-1}\frac{\Vert y_n^{(i)}-y_n^{(i-1)}\Vert^2-\Vert y_{n-1}^{(i)}-y_{n-1}^{(i-1)}\Vert^2}{\xi'(q_i)-\xi'(q_{i-1})}\\
			=-2\Delta_n\langle y_{n-1}^{(n-1)},\bfm^{(n)}-\bfm^{(n-1)}\rangle+N\xi'(q_{n-1})\Delta_n^2(q_n-q_{n-1}).
		\end{multline*}
		Moreover, by definition,
		\begin{equation*}
			y_n^{(n)}-y_n^{(n-1)}
			=\partial_m h_\zeta(q_n,\bfm^{(n)})-\partial_m h_\zeta(q_{n-1},\bfm^{(n-1)})-(\xi'(q_n)\!-\!\xi'(q_{n-1})) \Delta_n\,\bfm^{(n)}.
		\end{equation*}
		Since $\bfm^{(n)}-\bfm^{(n-1)}\perp \Span\{\bfm^{(1)},\ldots,\bfm^{(n-1)}\}$, the definition \eqref{def:y(k)} at level $n-1$ gives
		$\langle y_{n-1}^{(n-1)},\bfm^{(n)}-\bfm^{(n-1)}\rangle=\langle \partial_m h_\zeta(q_{n-1},\allowbreak\bfm^{(n-1)}),\bfm^{(n)}-\bfm^{(n-1)}\rangle$.
		Expanding the squares and using
		$\lVert\bfm^{(n)}\rVert^2=Nq_n$, we obtain
		\begin{multline*}
			\sum_{i=1}^n\frac{\Vert y_n^{(i)}-y_n^{(i-1)}\Vert^2}{\xi'(q_i)\!-\!\xi'(q_{i-1})}
			-\sum_{i=1}^{n-1}\frac{\Vert y_{n-1}^{(i)}-y_{n-1}^{(i-1)}\Vert^2}{\xi'(q_i)\!-\!\xi'(q_{i-1})}
			=\frac{\bigl\Vert\partial_m h_\zeta(q_n,\bfm^{(n)})-\partial_m h_\zeta(q_{n-1},\bfm^{(n-1)})\bigr\Vert^2}{\xi'(q_n)-\xi'(q_{n-1})}
			\\-2\Delta_n\Bigl(\langle \partial_m h_\zeta(q_n,\bfm^{(n)}),\bfm^{(n)}\rangle-\langle \partial_m h_\zeta(q_{n-1},\bfm^{(n-1)}),\bfm^{(n-1)}\rangle\Bigr)
			+N \Delta_n^2\bigl(\xi'(q_n)q_n-\xi'(q_{n-1})q_{n-1}\bigr).
		\end{multline*}
		Plugging this into $\mathrm{RHS}_n-\mathrm{RHS}_{n-1}$ and grouping terms yields
		\begin{align*}
			&-\frac12\langle a_n,(\Gamma_n')^{-1}a_n\rangle+\frac12\langle a_{n-1},(\Gamma_{n-1}')^{-1}a_{n-1}\rangle\\
			&\qquad \le
			-\frac{N}{2}\xi''(q_n)\Bigl(\int_{q_n}^1\zeta([0,t])\,\dd t\Bigr)^2
			-\frac{\bigl\Vert\partial_m h_\zeta(q_n,\bfm^{(n)})-\partial_m h_\zeta(q_{n-1},\bfm^{(n-1)})\bigr\Vert^2}{2(\xi'(q_n)-\xi'(q_{n-1}))}\\
			&\qquad\quad
			-\Delta_n\!\sum_{i=1}^N\bigl(h_\zeta(q_n,m_i^{(n)})\!-\!\partial_m h_\zeta(q_n,m_i^{(n)})m_i^{(n)}\bigr)\\
			&\qquad\quad
			+\Delta_n\!\sum_{i=1}^N\bigl(h_\zeta(q_{n-1},m_i^{(n-1)})\!-\!\partial_m h_\zeta(q_{n-1},m_i^{(n-1)})m_i^{(n-1)}\bigr)\\
			&\qquad\quad
			-N\Delta_n(f_n-f_{n-1})
			-N \Delta_n\bigl(\mc U_\zeta(q_n)-\mc U_\zeta(q_{n-1})\bigr)\\
			&\qquad\quad
			+\tfrac{N\Delta_n^2}{2}\bigl(\xi(q_n)\!-\!q_n\xi'(q_n)\\
			&\hspace{8em}
			-\xi(q_{n-1})\!+\!q_{n-1}\xi'(q_{n-1})\bigr).
		\end{align*}
		Since $\supp(\zeta)\cap(q_{n-1},q_n)=\varnothing$, we have $\zeta([0,t])=\Delta_n$ for all $t\in[q_{n-1},q_n)$. Hence
		\begin{equation*}
			\mc U_\zeta(q_n)-\mc U_\zeta(q_{n-1})
			=-\frac12\int_{q_{n-1}}^{q_n} t\xi''(t)\zeta([0,t])\,\dd t
			=-\frac{\Delta_n}{2}\int_{q_{n-1}}^{q_n} t\xi''(t)\,\dd t.
		\end{equation*}
		Moreover, using $(\xi(q)-q\xi'(q))'=-q\xi''(q)$,
		\begin{equation*}
			\xi(q_n)-q_n\xi'(q_n)-\xi(q_{n-1})+q_{n-1}\xi'(q_{n-1})
			=-\int_{q_{n-1}}^{q_n} t\xi''(t)\,\dd t.
		\end{equation*}
		These two identities show that the last two terms cancel, yielding exactly the right-hand side of \eqref{eq:prob ineq} without the determinant contribution.
		
		We now turn to the determinant. Write the covariance of $\Skel_n$ in block form
		\[
		\Gamma_n'
		=\begin{pmatrix}
			G_n & C_n\\
			C_n^T & E_n
		\end{pmatrix},
		\]
		where $G_n$ is the covariance of the $n$ gradients (dimension $nN\times nN$) and $E_n$ is the covariance of the $n$ energies (dimension $n\times n$).
		
		Recall that $\Cov(H(\bfm),H(\bfm'))=N\xi(\langle \bfm,\bfm'\rangle/N)$, hence for $i,j\in[n]$ and coordinates $a,b\in[N]$,
		\[
		\Cov(\partial_a H(\bfm^{(i)}),\partial_b H(\bfm^{(j)}))
		=\xi'(q_{ij})\mathbf 1_{\{a=b\}}+\frac{1}{N}\xi''(q_{ij})\,m^{(j)}_a\,m^{(i)}_b,
		\qquad q_{ij}=\frac1N\langle \bfm^{(i)},\bfm^{(j)}\rangle.
		\]
		Equivalently, the $(i,j)$ block of $G_n$ is
		\[
		(G_n)_{ij}=\xi'(q_{ij})I_N+\frac{1}{N}\xi''(q_{ij})\,\bfm^{(j)}(\bfm^{(i)})^T.
		\]
		Therefore,
		\[
		G_n = (\xi'(Q)\otimes I_N) + \frac{1}{N}R_n,
		\qquad
		R_n:=\sum_{i,j=1}^n \xi''(q_{ij})\,(e_i\otimes \bfm^{(j)})(e_j\otimes \bfm^{(i)})^T,
		\]
		so $R_n$ has rank at most $n^2$ (hence $G_n$ is a finite-rank perturbation of $\xi'(Q)\otimes I_N$).
		
		Let $A:=\xi'(Q)\otimes I_N$. Since $\xi'(Q)\succ 0$, $A$ is invertible and
		\[
		\det G_n = \det A \cdot \det\!\bigl(I_{nN}+A^{-1}(N^{-1}R_n)\bigr).
		\]
		Now $A^{-1}=(\xi'(Q)^{-1}\otimes I_N)$ and $\|\bfm^{(k)}\|^2=\langle \bfm^{(k)},\bfm^{(k)}\rangle=Nq_k$, so each rank-one term
		$\frac{1}{N}\xi''(q_{ij})(e_i\otimes \bfm^{(j)})(e_j\otimes \bfm^{(i)})^T$
		has operator norm $O(1)$, and $\mathrm{rank}(R_n)\le n^2$. Hence $A^{-1}(N^{-1}R_n)$ has rank $\le n^2$ and bounded operator norm, so
		\[
		\log\det G_n
		=\log\det A + O(1)
		= N\log\det \xi'(Q) + O(1),
		\]
		since $\det(\xi'(Q)\otimes I_N)=\det(\xi'(Q))^N$.
		
		Applying the Schur complement formula,
		\[
		\det\Gamma_n'=\det G_n\cdot \det\bigl(E_n-C_n^T G_n^{-1}C_n\bigr),
		\]
		where the second determinant is over an $n\times n$ matrix. Its entries are $O(N)$ (indeed $E_n$ has entries $N\xi(q_{ij})$ and
		$C_n$ has entries of order $\|\bfm^{(j)}\|=O(\sqrt N)$), so it contributes $O(\log N)$ to $\log\det\Gamma_n'$.
		Consequently,
		\[
		\log\det\Gamma_n'=N\log\det\xi'(Q)+O(\log N),
		\qquad
		\log\det\Gamma_{n-1}'=N\log\det\xi'(Q_{n-1})+O(\log N).
		\]
		Since $\xi'(Q)$ has the $LDL^T$ factorization with $L_{ik}=\mathbf 1_{\{k\le i\}}$, we have
		$\det\xi'(Q)=\prod_{i=1}^n (\xi'(q_i)-\xi'(q_{i-1}))$ and $\det\xi'(Q_{n-1})=\prod_{i=1}^{n-1}(\xi'(q_i)-\xi'(q_{i-1}))$.
		Therefore,
		\[
		-\frac12\log\frac{\det(2\pi\Gamma_n')}{\det(2\pi\Gamma_{n-1}')}
		=-\frac{N}{2}\log\bigl(2\pi(\xi'(q_n)-\xi'(q_{n-1}))\bigr)+O(\log N),
		\]
		and combining with the quadratic-form bound yields \eqref{eq:prob ineq}.
		
		Finally, if in addition \eqref{eq:addiass} holds, then \eqref{eq:u} is an equality by Step~3; since \eqref{eq:Goodn-1} is already an equality, the bound \eqref{eq:prob ineq} holds with equality. This concludes the proof.
	\end{proof}
	
	\subsection{Deriving optimality conditions}
	
	In this subsection, we derive necessary conditions satisfied by the optimizer $\zeta$ in the assumption of Proposition~\ref{prop:annealed cond}, extending the optimality consequences of Lemma~\ref{lemma:opt2prefix} to the hierarchical setting.
	
	\begin{lemma}\label{lemma:optnprefix}
		Let $n\geq 2$ and let $(u_i)_{1\leq i\leq n}$ and $(q_i)_{1\leq i\leq n}$ be strictly increasing sequences of numbers in $(0,1)$. Let $\zeta\in \Prefix_{n+1}((u_i)_{1\leq i\leq n},(q_i)_{1\leq i\leq n})$ where $\Prefix_{n+1}$ is as in Definition \ref{def:nprefix}. Suppose that $\zeta$ satisfies the assumptions of Proposition \ref{prop:annealed cond}. Let $(X_t)_{t\in [0,1]}$ be the Auffinger--Chen process
		\begin{equation*}
			\begin{cases}
				\dd X_t=\xi''(t)\zeta([0,t])\partial_x \Phi_\zeta(t,X_t)\dd t+\sqrt{\xi''(t)} \dd B_t\\
				X_0=0.
			\end{cases}
		\end{equation*}
		For every $s\in [0,1]$, set
		\begin{equation*}
			H_\zeta(s):=\frac{1}{2}\int_s^1 \xi''(r)\Bigl(\dE[(\partial_x\Phi_\zeta(r,X_r))^2]-r\Bigr)\dd r.
		\end{equation*}
		Then, the following hold:
		\begin{enumerate}
			\item For every $k\in [n-1]$,
			\begin{equation}\label{eq:optfzeta}
				-\dE[\Phi_\zeta(q_k,X_{q_k})]+\frac{1}{2}\int_0^1 t\xi''(t)\zeta([0,t])\dd t +\frac{1}{2}\int_0^{q_k}t\xi''(t) \zeta([0,t])\dd t=-\Pari(\zeta),
			\end{equation}
			\item   \begin{equation}\label{eq:optf}
				-\dE[\Phi_\zeta(q_n,X_{q_n})]+\frac{1}{2}\int_0^1 t\xi''(t)\zeta([0,t])\dd t +\frac{1}{2}\int_0^{q_n}t\xi''(t) \zeta([0,t])\dd t=-f,
			\end{equation}
			\item For every $k\in [n]$,
			\begin{equation}\label{eq:statqkk}
				\dE[(\partial_x\Phi_\zeta(q_k,X_{q_k}))^2]=q_k,
			\end{equation}
			\item  \begin{equation}\label{eq:suppzzz}
				\supp(\zeta)\cap [q_n,1]\subset \underset{[q_n,1]}{\mathrm{arg\,min}}(H_\zeta).
			\end{equation}
		\end{enumerate}
	\end{lemma}

	\begin{proof}
		The proof follows the same strategy as the proof of Lemma~\ref{lemma:opt2prefix}; we only sketch the differences.
		
		Write $\zeta=u_1\delta_0+(1-u_1)\lambda$ with $\lambda\in\mc P([q_1,1])$ and $q_1\in\supp(\lambda)$. Set $\lambda':=(1-\theta)\lambda+\theta\lambda_0$ for a perturbation $\lambda_0\in\mc P([q_1,1])$ with $q_1\in\supp(\lambda_0)$, and define $\zeta_\theta:=u_1\delta_0+(1-u_1)\lambda'$. Then $\zeta_\theta\in\Prefix_{n+1}(\bm u;\bm q)$, and the tail optimality assumption gives $\frac{\dd}{\dd\theta}\big|_{\theta=0^+}\Pari(\zeta_\theta)\geq 0$. By the first-variation formula from Lemma~\ref{lemma:first order mu}\ref{item:gateaumu}, this yields \eqref{eq:suppzzz} exactly as in the proof of~\eqref{eq:Hzetaconstraint}.
		
		For the stationarity conditions, differentiating $\Cpx_f^{(n)}(\bm u;\bm q)$ with respect to $u_l$ for $l\in[n-1]$ gives $\frac{\partial}{\partial u_l}\Pari(\zeta)=0$; combined with the first-variation formula in $u_l$, this yields \eqref{eq:optfzeta} by the same computation as for~\eqref{eq:energyconstraint}. Differentiating with respect to $u_n$ gives $\Pari(\zeta)+u_n\frac{\partial}{\partial u_n}\Pari(\zeta)=f$, which together with the first-variation formula yields \eqref{eq:optf}. Finally, differentiating with respect to $q_k$ and using the stationarity condition $\partial_{q_k}\Cpx_f^{(n)}=0$ gives \eqref{eq:statqkk}, by the same argument as for~\eqref{eq:breakingpoint}.
	\end{proof}

	\subsection{Proof of the conditional annealed estimate}
	We are now in a position to prove the section's main result, following the proof of Proposition \ref{prop:annealed 0}.

	\begin{proof}[Proof of Proposition \ref{prop:annealed cond}]
		Set
		\[
		u_n:=\zeta([0,q_n)),\qquad
		\zeta_\star:=\zeta|_{(q_n,1]}+\zeta([0,q_n])\delta_{q_n}.
		\]
		As in the proof of Proposition \ref{prop:annealed 0}, we write $C_\ve$ for a
		positive quantity such that $C_\ve\to 0$ as $\ve\to 0$, and $O_\ve(N)$ for any
		quantity bounded in absolute value by $C_\ve N$.
		
		\paragraph{\bf Step 1: upper bound.}
		By Lemma \ref{lemma:KacRice},
		\begin{align*}
			&\dE\Bigl[\mc N_\ve(f,\zeta,q_n,\bfm^{(n-1)})\mid
			\Skel_{n-1}=(0,N\Pari(\zeta),\ldots,0,N\Pari(\zeta))\Bigr]\\
			&\qquad=
			N\int_{[-1,1]^N}\int_{f-\ve}^{f+\ve}
			\dE\Bigl[
			|\det(\nabla^2F_{\TAP,\zeta}(\bfm))|
			\,\Big|\,
			\Skel_n=(0,N\Pari(\zeta),\ldots,0,N\Pari(\zeta),0,Nf')
			\Bigr]\\
			&\hspace{8em}\times
			\varphi_{Z(\bfm)\mid\Skel_{n-1}}
			\Bigl(
			0,Nf'
			\,\Big|\,
			0,N\Pari(\zeta),\ldots,0,N\Pari(\zeta)
			\Bigr)\\
			&\hspace{8em}\times
			\indic_{(\bfm^{(1)},\ldots,\bfm^{(n-1)},\bfm)\in \SUSY_n(\zeta,\ve)}
			\,\dd\bfm\,\dd f'.
		\end{align*}
		For every $\bfm$ in the domain of integration, the conditional law of $\nabla^2 H(\bfm)$ given $\Skel_n$ differs from its law given $(Z(\bfm))$ by a finite-rank (at most $2(n-1)$-rank) perturbation arising from the conditioning on $\Skel_{n-1}$. Since Lemma~\ref{lemma:det app} computes the determinant via the free convolution of $T_\zeta\#\mu_N^{(\bfm)}$ with a semicircle law, and a finite-rank perturbation changes the Stieltjes transform by $O(1/N)$, the same asymptotic holds. More precisely, computations analogous to those in the proof of Lemma~\ref{lemma:det asymp} give
		\begin{multline*}
			\log \dE\Bigl[
			|\det(\nabla^2F_{\TAP,\zeta}(\bfm))|
			\,\Big|\,
			\Skel_n=(0,N\Pari(\zeta),\ldots,0,N\Pari(\zeta),0,Nf')
			\Bigr]\\
			=
			\sum_{i=1}^N\log \partial_{mm}h_\zeta(q_n,m_i)
			+\frac N2\xi''(q_n)\Bigl(\int_{q_n}^1\zeta([0,t])\,\dd t\Bigr)^2
			+O_\ve(N).
		\end{multline*}
		
		Next, for every $k\le n-1$, \eqref{eq:optfzeta} and \eqref{def:Linear4} give
		\begin{multline*}
			\sum_{i=1}^N\bigl(h_\zeta(q_k,m_i^{(k)})-\partial_m h_\zeta(q_k,m_i^{(k)})\,m_i^{(k)}\bigr)
			+\frac N2\int_0^1 t\xi''(t)\zeta([0,t])\,\dd t\\
			+\frac N2\int_0^{q_k} t\xi''(t)\zeta([0,t])\,\dd t
			=-N\Pari(\zeta)+O_\ve(N).
		\end{multline*}
		Hence the hypothesis \eqref{eq:ass ancestors} of Lemma \ref{lemma:prob0 replica}
		holds with $f_1=\cdots=f_{n-1}=\Pari(\zeta)$, and therefore
		\begin{align*}
			&\log
			\varphi_{Z(\bfm)\mid\Skel_{n-1}}
			\Bigl(
			0,Nf'
			\,\Big|\,
			0,N\Pari(\zeta),\ldots,0,N\Pari(\zeta)
			\Bigr)\\
			&\quad\le
			-\frac N2\xi''(q_n)\Bigl(\int_{q_n}^1\zeta([0,t])\,\dd t\Bigr)^2\\
			&\qquad
			-\frac{\bigl\|
				\partial_m h_\zeta(q_n,\bfm)-\partial_m h_\zeta(q_{n-1},\bfm^{(n-1)})
				\bigr\|^2}{2(\xi'(q_n)-\xi'(q_{n-1}))}\\
			&\qquad
			-u_n\sum_{i=1}^N\bigl(h_\zeta(q_n,m_i)-m_i\,\partial_m h_\zeta(q_n,m_i)\bigr)\\
			&\qquad
			+u_n\sum_{i=1}^N\bigl(h_\zeta(q_{n-1},m_i^{(n-1)})
			-m_i^{(n-1)}\,\partial_m h_\zeta(q_{n-1},m_i^{(n-1)})\bigr)\\
			&\qquad
			-Nu_n(f-\Pari(\zeta))
			-\frac N2\log\bigl(2\pi(\xi'(q_n)-\xi'(q_{n-1}))\bigr)
			+O_\ve(N),
		\end{align*}
		where we used $|f'-f|\le \ve$ and absorbed the resulting error into $O_\ve(N)$.
		
		Combining the last two displays and integrating coordinatewise, we obtain
		\begin{align*}
			&\log\dE\Bigl[\mc N_\ve(f,\zeta,q_n,\bfm^{(n-1)})\mid
			\Skel_{n-1}=(0,N\Pari(\zeta),\ldots,0,N\Pari(\zeta))\Bigr]\\
			&\quad\le
			u_n\sum_{i=1}^N\bigl(h_\zeta(q_{n-1},m_i^{(n-1)})
			-m_i^{(n-1)}\,\partial_m h_\zeta(q_{n-1},m_i^{(n-1)})\bigr)\\
			&\qquad
			-\frac N2\log\bigl(2\pi(\xi'(q_n)-\xi'(q_{n-1}))\bigr)
			+Nu_n(\Pari(\zeta)-f)+O_\ve(N)\\
			&\qquad
			+\sum_{i=1}^N
			\log \int_{-1}^1
			\partial_{mm}h_\zeta(q_n,m)\,
			\exp\!\biggl(
			-\frac{(\partial_m h_\zeta(q_n,m)-\partial_m h_\zeta(q_{n-1},m_i^{(n-1)}))^2}
			{2(\xi'(q_n)-\xi'(q_{n-1}))}\\
			&\hspace{14em}
			-u_n\bigl(h_\zeta(q_n,m)-m\,\partial_m h_\zeta(q_n,m)\bigr)
			\biggr)\,\dd m.
		\end{align*}
		By \eqref{eq:f Legendre},
		\[
		h_\zeta(q_{n-1},a)-a\,\partial_m h_\zeta(q_{n-1},a)
		=
		-\Phi_\zeta\bigl(q_{n-1},\partial_m h_\zeta(q_{n-1},a)\bigr),
		\]
		and the change of variables $y=\partial_m h_\zeta(q_n,m)$ together with
		Lemma \ref{lemma:law Xt s} gives
		\begin{multline*}
			\int_{-1}^1
			\partial_{mm}h_\zeta(q_n,m)\,
			\exp\!\biggl(
			-\frac{(\partial_m h_\zeta(q_n,m)-x)^2}{2(\xi'(q_n)-\xi'(q_{n-1}))}\\
			-u_n\bigl(h_\zeta(q_n,m)-m\,\partial_m h_\zeta(q_n,m)\bigr)
			\biggr)\,\dd m
			=
			\sqrt{2\pi(\xi'(q_n)-\xi'(q_{n-1}))}\,e^{u_n\Phi_\zeta(q_{n-1},x)}.
		\end{multline*}
		The two terms cancel exactly, and we conclude that
		\[
		\log\dE\Bigl[\mc N_\ve(f,\zeta,q_n,\bfm^{(n-1)})\mid
		\Skel_{n-1}=(0,N\Pari(\zeta),\ldots,0,N\Pari(\zeta))\Bigr]
		\le
		Nu_n(\Pari(\zeta)-f)+O_\ve(N).
		\]
		
		\paragraph{\bf Step 2: reduction to a good event.}
		Set $\delta_\ve:=\ve^{1/2}$ and let $\Good(\ve)\subset[-1+\delta_\ve,1-\delta_\ve]^N$
		be the set of $\bfm$ such that
		\[
		(\bfm^{(1)},\ldots,\bfm^{(n-1)},\bfm)\in \SUSY_n(\zeta,\ve),
		\]
		and
		\begin{multline*}
			\frac1N\sum_{i=1}^N\bigl(h_\zeta(q_n,m_i)-m_i\,\partial_m h_\zeta(q_n,m_i)\bigr)\\
			+\frac12\int_0^1 t\xi''(t)\zeta([0,t])\,\dd t
			+\frac12\int_0^{q_n} t\xi''(t)\zeta([0,t])\,\dd t
			\in[-f-\ve,-f+\ve].
		\end{multline*}
		In particular, every $\bfm\in \Good(\ve)$ satisfies
		\[
		\dist(\zeta_\bfm,\zeta_\star)\le \ve.
		\]
		
		Restricting the Kac--Rice integral to $\Good(\ve)$ and using the equality case
		of Lemma \ref{lemma:prob0 replica} (the last energy constraint holds with $f'$
		in place of $f$ because $|f'-f|\le \ve$), we get
		\begin{align*}
			&\log\dE\Bigl[\mc N_\ve(f,\zeta,q_n,\bfm^{(n-1)})\mid
			\Skel_{n-1}=(0,N\Pari(\zeta),\ldots,0,N\Pari(\zeta))\Bigr]\\
			&\quad\ge
			u_n\sum_{i=1}^N\bigl(h_\zeta(q_{n-1},m_i^{(n-1)})
			-m_i^{(n-1)}\,\partial_m h_\zeta(q_{n-1},m_i^{(n-1)})\bigr)\\
			&\qquad
			-\frac N2\log\bigl(2\pi(\xi'(q_n)-\xi'(q_{n-1}))\bigr)
			+Nu_n(\Pari(\zeta)-f)
			+\log I_{N,\ve}
			+O_\ve(N),
		\end{align*}
		where
		\begin{multline*}
			I_{N,\ve}:=
			\int_{[-1,1]^N}
			\prod_{i=1}^N
			\partial_{mm}h_\zeta(q_n,m_i)\,
			\exp\!\biggl(
			-\frac{(\partial_m h_\zeta(q_n,m_i)-\partial_m h_\zeta(q_{n-1},m_i^{(n-1)}))^2}
			{2(\xi'(q_n)-\xi'(q_{n-1}))}\\
			-u_n\bigl(h_\zeta(q_n,m_i)-m_i\,\partial_m h_\zeta(q_n,m_i)\bigr)
			\biggr)
			\indic_{\Good(\ve)}(\bfm)\,\dd\bfm.
		\end{multline*}
		
		\paragraph{\bf Step 3: control of the good event under a truncated conditional product law.}
		For $a\in[-1+\delta_\ve,1-\delta_\ve]$, define the truncated normalizing constant
		\begin{multline*}
			Z_\ve(a):=
			\int_{-1+\delta_\ve}^{1-\delta_\ve}
			\partial_{mm}h_\zeta(q_n,m)\,
			\exp\!\biggl(
			-\frac{(\partial_m h_\zeta(q_n,m)-\partial_m h_\zeta(q_{n-1},a))^2}
			{2(\xi'(q_n)-\xi'(q_{n-1}))}\\
			-u_n\bigl(h_\zeta(q_n,m)-m\,\partial_m h_\zeta(q_n,m)\bigr)
			\biggr)\,\dd m,
		\end{multline*}
		and let $\pi_\ve(\dd m\mid a)$ be the probability measure on
		$[-1+\delta_\ve,1-\delta_\ve]$ proportional to the integrand above.
		Let
		\[
		\mathbb P_\ve^{(N)}:=\bigotimes_{i=1}^N \pi_\ve(\cdot\mid m_i^{(n-1)}).
		\]
		Then
		\[
		I_{N,\ve}
		=
		\Bigl(\prod_{i=1}^N Z_\ve(m_i^{(n-1)})\Bigr)\,
		\mathbb P_\ve^{(N)}\bigl(\Good(\ve)\bigr).
		\]
		
		We now show that
		\begin{equation}\label{eq:good-prob-cond-rewrite}
			\mathbb P_\ve^{(N)}\bigl(\Good(\ve)^c\bigr)\longrightarrow 0
			\qquad\text{as }N\to\infty.
		\end{equation}
		To this end, let $\pi(\cdot\mid a)$ denote the same kernel without truncation
		(i.e., integrating over $[-1,1]$ instead of $[-1+\delta_\ve,1-\delta_\ve]$).
		If
		\[
		U:=\partial_x\Phi_\zeta(q_{n-1},X_{q_{n-1}}),
		\]
		then by \eqref{eq:f Legendre} and Lemma \ref{lemma:law Xt s},
		\[
		\partial_m h_\zeta(q_n,\cdot)\#\pi(\cdot\mid a)
		=
		\Law(X_{q_n}\mid U=a).
		\]
		Consequently,
		\begin{align*}
			\int m^2\,\pi(\dd m\mid a)&=\chi_1(a),\\ 
			\int m\,\pi(\dd m\mid a)&=\chi_3(a)=a,\\
			\int m\,\partial_m h_\zeta(q_n,m)\,\pi(\dd m\mid a)&=\chi_2(a),\\
			\int \partial_m h_\zeta(q_n,m)\,\pi(\dd m\mid a)&=\chi_4(a),\\
			\int \bigl(h_\zeta(q_n,m)-m\,\partial_m h_\zeta(q_n,m)\bigr)\,\pi(\dd m\mid a)&=-\chi_5(a).
		\end{align*}
		
		For each fixed $\ve$, the corresponding truncated expectations
		\[
		a\mapsto \int \psi(m)\,\pi_\ve(\dd m\mid a)
		\]
		are continuous on $[-1+\delta_\ve,1-\delta_\ve]$ for every bounded continuous
		$\psi$, and differ from the untruncated ones by $O_\ve(1)$ uniformly on this
		compact set. Using the weak law-matching input
		\eqref{def:Linear}--\eqref{def:Linear4}, we therefore obtain the following target
		means, all up to an $O_\ve(N)$ error:
		\begin{align*}
			\dE_{\mathbb P_\ve^{(N)}}\!\Bigl[\sum_{i=1}^N m_i^2\Bigr]
			&=Nq_n+O_\ve(N),\\
			\dE_{\mathbb P_\ve^{(N)}}\!\Bigl[\sum_{i=1}^N (m_i-m_i^{(n-1)})m_i^{(k)}\Bigr]
			&=O_\ve(N),
			\qquad 1\le k\le n-1,\\
			\dE_{\mathbb P_\ve^{(N)}}\!\Bigl[\sum_{i=1}^N \partial_m h_\zeta(q_n,m_i)\,m_i^{(k)}\Bigr]
			&=
			N\,\dE\bigl[X_{q_n}\partial_x\Phi_\zeta(q_k,X_{q_k})\bigr]
			+O_\ve(N),\\
			&\hspace{12em}
			1\le k\le n,\\
			\dE_{\mathbb P_\ve^{(N)}}\!\Bigl[\sum_{i=1}^N \partial_m h_\zeta(q_k,m_i^{(k)})\,m_i\Bigr]
			&=
			N\,\dE\bigl[X_{q_k}\partial_x\Phi_\zeta(q_n,X_{q_n})\bigr]
			+O_\ve(N),\\
			&\hspace{12em}
			1\le k\le n-1,\\
			\dE_{\mathbb P_\ve^{(N)}}\!\Bigl[\sum_{i=1}^N \bigl(h_\zeta(q_n,m_i)
			-m_i\,\partial_m h_\zeta(q_n,m_i)\bigr)\Bigr]
			&=
			-N\,\dE[\Phi_\zeta(q_n,X_{q_n})]+O_\ve(N).
		\end{align*}
		In the last line, \eqref{eq:optf} turns the right-hand side into
		\[
		-Nf-\frac N2\int_0^1 t\xi''(t)\zeta([0,t])\,\dd t
		-\frac N2\int_0^{q_n} t\xi''(t)\zeta([0,t])\,\dd t
		+O_\ve(N).
		\]
		
		Since the coordinates are independent and uniformly bounded under
		$\mathbb P_\ve^{(N)}$, Chebyshev's inequality gives concentration around these
		means. Because there are only finitely many scalar and bilinear constraints, it
		follows that all of them hold with $\mathbb P_\ve^{(N)}$-probability tending to
		$1$.
		
		It remains to control the metric constraint on $\zeta_\bfm$.
		Recall that $\zeta_\bfm$ is the unique minimizer of $\TAP(\mu_N^{(\bfm)},\cdot)$
		on $\mc P([q_\bfm,1])$, and that $\TAP(\mu,\zeta')$ depends on $\mu$ only through
		$q_\mu=\int m^2\,\dd\mu$ and $\int h_{\zeta'}(q_\mu,m)\,\dd\mu(m)$.
		Set
		\[
		\mu_\ve:=\int_{-1+\delta_\ve}^{1-\delta_\ve} \pi_\ve(\cdot\mid a)\,\Law(U)(\dd a).
		\]
		We show that $\TAP(\mu_N^{(\bfm)},\cdot)$ is uniformly close
		to $\TAP(\mu_\ve,\cdot)$, using only finitely many linear statistics of
		the ancestors.
		
		By the Legendre duality~\eqref{eq:f Legendre} and Lemma~\ref{lemma:law Xt s},
		the pushforward of $\pi(\cdot\mid a)$ by $\partial_m h_\zeta(q_n,\cdot)$ is
		$\Law(X_{q_n}\mid U=a)$; hence
		\[
		\int h_\zeta(q_n,m)\,\pi(\dd m\mid a)
		=\chi_2(a)-\chi_5(a),
		\qquad
		\int m^2\,\pi(\dd m\mid a)=\chi_1(a).
		\]
		The truncated kernels $\pi_\ve(\cdot\mid a)$ differ from $\pi(\cdot\mid a)$
		by $O_\ve(1)$ uniformly on $[-1+\delta_\ve,1-\delta_\ve]$.
		Under $\mathbb P_\ve^{(N)}$, the coordinates $m_i$ are independent and
		bounded, so for every fixed continuous $g$,
		\[
		\frac1N\sum_{i=1}^N g(m_i)
		-\frac1N\sum_{i=1}^N \int g\,\dd\pi_\ve(\cdot\mid m_i^{(n-1)})
		\xrightarrow[N\to\infty]{} 0
		\qquad\text{in }\mathbb P_\ve^{(N)}\text{-probability}
		\]
		by Chebyshev's inequality ($\mathrm{Var}=O(1/N)$).
		Applying this with $g(m)=m^2$ and $g(m)=h_{\zeta'}(q_n,m)$, it suffices
		to compare the \emph{deterministic} averages
		$\frac1N\sum\int g\,\dd\pi_\ve(\cdot\mid m_i^{(n-1)})$ with $\int g\,\dd\mu_\ve$.
		
		For the reference measure $\zeta$ and $g=h_\zeta(q_n,\cdot)$:
		\[
		\frac1N\sum_{i=1}^N\int h_\zeta(q_n,m)\,\pi_\ve(\dd m\mid m_i^{(n-1)})
		=\frac1N\sum_{i=1}^N\bigl(\chi_2(m_i^{(n-1)})-\chi_5(m_i^{(n-1)})\bigr)+O_\ve(1).
		\]
		By the law-matching condition~\eqref{def:Linear3} applied to $\chi_2$ and $\chi_5$,
		this equals $\dE[\chi_2(U)-\chi_5(U)]+O_\ve(1)=\int h_\zeta(q_n,m)\,\dd\mu_\ve(m)+O_\ve(1)$.
		Similarly, \eqref{def:Linear3} applied to $\chi_1$ gives
		$\frac1N\sum\chi_1(m_i^{(n-1)})=q_n+O(\ve)$.
		
		For a nearby measure $\zeta'\in\mc P([q_n,1])$, the stability of the
		Parisi PDE gives
		$\sup_{m\in[-1+\delta_\ve,1-\delta_\ve]}|h_{\zeta'}(q_n,m)-h_\zeta(q_n,m)|
		\le C\,\dist(\zeta',\zeta)$. Since this bound is uniform in the ancestor
		coordinates, we obtain
		\[
		\sup_{\zeta'\in\mc P([q_n,1])}
		\bigl|\TAP(\mu_N^{(\bfm)},\zeta')-\TAP(\mu_\ve,\zeta')\bigr|
		=O_\ve(1)
		\qquad\text{in }\mathbb P_\ve^{(N)}\text{-probability.}
		\]
		Since $\mc P([q_n,1])$ is compact and $\zeta'\mapsto\TAP(\mu_\ve,\zeta')$ is
		strictly convex with unique minimizer $\zeta_{\star,\ve}$, uniform
		convergence of the functionals implies convergence of the minimizers:
		\[
		\dist(\zeta_\bfm,\zeta_{\star,\ve})\xrightarrow[N\to\infty]{} 0
		\qquad\text{in }\mathbb P_\ve^{(N)}\text{-probability.}
		\]
		
		Finally, as in Step~2 of the proof of Proposition~\ref{prop:annealed 0},
		the unique minimizer of $\TAP(\mu,\cdot)$ for
		$\mu:=\Law(\partial_x\Phi_\zeta(q_n,X_{q_n}))$ is $\zeta_\star$.
		Since $\mu_\ve\Rightarrow\mu$ and $\int m^2\,\dd\mu_\ve\to q_n$ as $\ve\to0$
		(by dominated convergence), Lemma~\ref{lemma:stableTapMin} gives
		$\dist(\zeta_{\star,\ve},\zeta_\star)\to 0$.
		After decreasing $\ve$ if necessary,
		$\dist(\zeta_{\star,\ve},\zeta_\star)\le \ve/2$, and hence
		\[
		\mathbb P_\ve^{(N)}\bigl(\dist(\zeta_\bfm,\zeta_\star)>\ve\bigr)\xrightarrow[N\to\infty]{} 0.
		\]
		Combining this with the concentration estimates above proves
		\eqref{eq:good-prob-cond-rewrite}.
		
		\paragraph{\bf Step 4: conclusion.}
		For $a\in[-1+\delta_\ve,1-\delta_\ve]$, let $Z(a)$ denote the normalization
		without truncation:
		\begin{multline*}
			Z(a):=
			\int_{-1}^1
			\partial_{mm}h_\zeta(q_n,m)\,
			\exp\!\biggl(
			-\frac{(\partial_m h_\zeta(q_n,m)-\partial_m h_\zeta(q_{n-1},a))^2}
			{2(\xi'(q_n)-\xi'(q_{n-1}))}\\
			-u_n\bigl(h_\zeta(q_n,m)-m\,\partial_m h_\zeta(q_n,m)\bigr)
			\biggr)\,\dd m.
		\end{multline*}
		By the integral identity from Step~1,
		\begin{multline*}
			\log Z(a)
			=
			\tfrac12\log\bigl(2\pi(\xi'(q_n)-\xi'(q_{n-1}))\bigr)
			+u_n\Phi_\zeta\bigl(q_{n-1},\partial_m h_\zeta(q_{n-1},a)\bigr)\\
			=
			\tfrac12\log\bigl(2\pi(\xi'(q_n)-\xi'(q_{n-1}))\bigr)
			-u_n\bigl(h_\zeta(q_{n-1},a)-a\,\partial_m h_\zeta(q_{n-1},a)\bigr).
		\end{multline*}
		Moreover, the truncation error satisfies
		\[
		\frac1N\sum_{i=1}^N \bigl(\log Z_\ve(m_i^{(n-1)})-\log Z(m_i^{(n-1)})\bigr)=O_\ve(1),
		\]
		since $\log Z_\ve(a)-\log Z(a)=O_\ve(1)$ uniformly on $[-1+\delta_\ve,1-\delta_\ve]$. Therefore,
		\begin{multline*}
			\sum_{i=1}^N \log Z_\ve(m_i^{(n-1)})
			=
			\frac N2\log\bigl(2\pi(\xi'(q_n)-\xi'(q_{n-1}))\bigr)\\
			-u_n\sum_{i=1}^N
			\bigl(h_\zeta(q_{n-1},m_i^{(n-1)})
			-m_i^{(n-1)}\,\partial_m h_\zeta(q_{n-1},m_i^{(n-1)})\bigr)
			+O_\ve(N).
		\end{multline*}
		Since $\mathbb P_\ve^{(N)}(\Good(\ve))\to 1$, we have
		\[
		\log I_{N,\ve}
		=
		\sum_{i=1}^N \log Z_\ve(m_i^{(n-1)})+o_N(N)
		\]
		for every fixed $\ve$. Inserting the last two displays into the lower bound from
		Step~2 yields
		\[
		\log\dE\Bigl[\mc N_\ve(f,\zeta,q_n,\bfm^{(n-1)})\mid
		\Skel_{n-1}=(0,N\Pari(\zeta),\ldots,0,N\Pari(\zeta))\Bigr]
		\ge
		Nu_n(\Pari(\zeta)-f)+O_\ve(N).
		\]
		Together with Step~1, this gives
		\[
		\lim_{\ve\to0}\lim_{N\to\infty}
		\frac1N\log\dE\Bigl[\mc N_\ve(f,\zeta,q_n,\bfm^{(n-1)})\mid
		\Skel_{n-1}=(0,N\Pari(\zeta),\ldots,0,N\Pari(\zeta))\Bigr]
		=
		u_n(\Pari(\zeta)-f).
		\]
		Finally, assumption \emph{(i)} says that
		\[
		\Ppar^{(n)}(\bm u;\bm q)=\Pari(\zeta),
		\]
		and by definition $u_n=\zeta([0,q_n))$, so the right-hand side is
		\[
		u_n\bigl(\Ppar^{(n)}(\bm u;\bm q)-f\bigr)=\Cpx_f^{(n)}(\bm u;\bm q).
		\]
		This completes the proof.
	\end{proof}
	
	\appendix

	\section{The supersymmetric ansatz}\label{section:SUSY}

	This appendix collects the heuristic supersymmetric computations that motivate the rigorous proofs of Sections \ref{section:base} and \ref{section:multiple}. The discussion in this section is purely formal and is not intended to be rigorous.

	\subsection{Ansatz for the annealed complexity}\label{sub:SUSY}

	\paragraph{\bf{Writing the complexity as a partition function}}
	Let $f\in \dR$, $q\in (0,1)$, and $u\in (0,1)$. Let $\zeta\in \Prefix_2(u,q)$ satisfy the assumptions of Proposition \ref{prop:annealed 0}, and note that $u=\zeta(\{0\})$. Recall $\mc{N}_\ve(f,\zeta,q)$ from the statement of Proposition \ref{prop:annealed 0}. Set $\zeta_\star=\zeta|_{(q,1]}+\zeta([0,q])\delta_q$.

	Replacing the condition $F_{\TAP,\zeta}(\bfm)\in N(f-\ve,f+\ve)$ by the sharper constraint $F_{\TAP,\zeta}(\bfm)=Nf$, the Kac--Rice formula yields (formally)
	\begin{equation*}
		\dE[\mc{N}_0(f,\zeta,q)]\approx\dE\left[\int_{[-1,1]^N} |\det(\nabla^2 F_{\TAP,\zeta}(\bfm))|\,\delta_{\nabla F_{\TAP,\zeta}}(0)\,\delta_{F_{\TAP,\zeta}(\bfm)}(Nf)\,\indic_{\SUSY_1(\zeta,0)}(\bfm) \dd \bfm\right].
	\end{equation*}
	We introduce a simplified quantity obtained by dropping the absolute value of the determinant: formally set
	\begin{equation*}
		\tilde{\mc{N}}_0(f,\zeta,q)=\int_{[-1,1]^N} \det(\nabla^2 F_{\TAP,\zeta}(\bfm))\,\delta_{\nabla F_{\TAP,\zeta}(\bfm)}(0)\,\delta_{F_{\TAP,\zeta}(\bfm)}(Nf)\,\indic_{\SUSY_1(\zeta,0)}(\bfm)\delta_{\zeta_\bfm-\zeta_\star}(0) \dd \bfm.
	\end{equation*}
	(As we shall see, the condition $\zeta_\bfm=\zeta_\star$ implies that the bulk of the spectrum of the Hessian has a fixed sign (positive in our convention), which justifies this simplification.) For a matrix $A\in \mc{M}_N(\dR)$, we use the standard Grassmann representation
	\begin{equation}\label{eq:detA}
		\det A=\int e^{-\bar\psi^{\top}A\psi}\dd\bar{\psi}\dd\psi,
	\end{equation}
	where $\dd \bar{\psi}\dd \psi=\prod_{i=1}^N \dd \bar{\psi}_i\dd \psi_i$ denotes the Berezin measure. The Dirac delta function admits the integral representation
	\begin{equation}\label{eq:delta0}
		\delta_0(y)=\frac{1}{2i\pi}\int_{i\dR}e^{xy}\dd x.
	\end{equation}
	Applying \eqref{eq:detA} to $\nabla^2 F_{\TAP,\zeta}(\bfm)$ and \eqref{eq:delta0} to $\partial_i F_{\TAP,\zeta}(\bfm)$ for each $i\in [N]$, as well as to $F_{\TAP,\zeta}(\bfm)-Nf$, and performing the change of variables $x\mapsto -x$ (which preserves the contour $(i\dR)^N$), we obtain
	\begin{equation}\label{eq:tildeNf}
		\dE[\tilde{\mc{N}}_0(f,\zeta,q)]=\frac{1}{(2i\pi)^{N+1}} \dE\left[\int e^{S(\bfm,x,v,\bar\psi,\psi)}\indic_{\SUSY_1(\zeta,0)}(\bfm) \dd \bfm \dd x \dd v \dd \bar\psi \dd \psi\right],
	\end{equation}
	where $S$ is the action
	\begin{equation*}
		S(\bfm,x,v,\bar\psi,\psi)= \,x\cdot\nabla F_{\TAP,\zeta}(\bfm)-\bar{\psi}^{\top}\nabla^2 F_{\TAP,\zeta}(\bfm)\psi+v(F_{\TAP,\zeta}(\bfm)-Nf),
	\end{equation*}
	with $x\in (i\dR)^N$ and $v\in i\dR$.

	\paragraph{\bf{The BRST supersymmetry}}
	Consider the infinitesimal fermionic transformation with odd parameter $\varepsilon$:
	\begin{equation*}
		\delta m_i=\varepsilon\psi_i,\quad
		\delta x_i=-\varepsilon v\psi_i,\quad
		\delta\bar\psi_i=\varepsilon x_i,\quad
		\delta\psi_i=0.
	\end{equation*}
	A standard computation shows that the Berezinian (super-Jacobian) of this change of variables equals $1$ and that
	\begin{equation*}
		\delta S=0.
	\end{equation*}
	Consequently, both the integrand and the measure are invariant under this transformation. Hence, for any sufficiently regular observable $O=O(\bfm,x,\bar{\psi},\psi)$,
	\begin{equation*}
		0=\int \delta(Oe^S)=\int (\delta O) e^S,
	\end{equation*}
	which yields the Ward identities
	\begin{equation*}
		\int (\delta O)\, e^S=0.
	\end{equation*}
	Taking
	\(
	O=\sum_{i=1}^N\bar\psi_i m_i
	\), we obtain
	\begin{equation}\label{eq:Ward1}
		\int \langle\bar\psi,\psi\rangle\, e^S=\int \langle x, \bfm\rangle\, e^S.
	\end{equation}

	A supersymmetric saddle is a critical point of the (disorder-averaged) effective action that satisfies all Ward identities simultaneously.

	\paragraph{\bf{Computing the annealed complexity}}
	Differentiating the TAP correction (using $q=\frac{1}{N}\Vert \bfm\Vert^2$ and $\zeta_\bfm=\zeta|_{(q,1]}+\zeta([0,q])\delta_q$, see Lemma \ref{lemma:gradient TAP}) gives
	\begin{equation}\label{eq:grad co}
		\partial_i \left(\sum_{j=1}^N \Phi_\zeta^*(q,\cdot)(m_j)+\frac{1}{2}N\int_q^1 t\xi''(t)\zeta([0,t])\dd t\right)=k_\zeta(q,m_i),
	\end{equation}
	where $k_\zeta(q,m_i)$ is as in \eqref{def:kqm}. Furthermore, for every $i,j\in [N]$,
	\begin{multline}\label{eq:Hess co}
		\partial_{ij} \left(\sum_{k=1}^N \Phi_\zeta^*(q,\cdot)(m_k)+\frac{1}{2}N\int_q^1 t\xi''(t)\zeta([0,t])\dd t\right)=\left(\partial_{mm}h_\zeta(q,m_i)+\xi''(q)\int_q^1 \zeta([0,t])\dd t\right)\indic_{i=j}\\+O\left(\frac{1}{N}\right).
	\end{multline}
	Taking expectations in \eqref{eq:tildeNf} and applying \eqref{eq:grad co}--\eqref{eq:Hess co} yields
	\begin{equation}\label{eq:tildeNf2}
		\dE[\tilde{\mc{N}}_0(f,\zeta,q)]=\frac{1}{(2i\pi)^{N+1}}\int e^{\dE[S(\bfm,x,v,\bar\psi,\psi)]+\frac{1}{2}\Var[S(\bfm,x,v,\bar\psi,\psi)] }\indic_{\SUSY_1(\zeta,0)}(\bfm)\dd \bfm \dd x\dd v \dd \bar\psi\dd \psi.
	\end{equation}
	Explicitly,
	\begin{multline*}
		\dE[S(\bfm,x,v,\bar\psi,\psi)]=-\sum_{i=1}^N x_i k_\zeta(q,m_i)-\sum_{i}\bar{\psi}_i\psi_i \Bigl(\partial_{mm}h_\zeta(q,m_i)+\xi''(q)\int_q^1 \zeta([0,t])\dd t\Bigr)\\+v\left(-\sum_{i=1}^N h_\zeta(q,m_i)-N\mc{U}_\zeta(q)-Nf\right)+O(1),
	\end{multline*}
	and
	\begin{multline*}
		\Var[S(\bfm,x,v,\bar\psi,\psi)]= v^2 N \xi(q) + \xi'(q)\,\|x\|^2 + \frac{\xi''(q)}{N}\,\langle \bfm, x\rangle^2
		- \frac{\xi''(q)}{N}\,\left(\sum_{i=1}^N \bar{\psi}_i\psi_i\right)^2 +2v \xi'(q)\sum_{i=1}^N x_i m_i
		\\+ \frac{2\xi''(q)}{N}\!\left[\langle \bar\psi, x\rangle \langle \bfm,\psi\rangle
		+ \langle\bar\psi, \bfm\rangle \langle x,\psi\rangle + v\langle\bar{\psi}, \bfm\rangle \langle \bfm,\psi\rangle\right]
		+ O\Bigl(\frac{1}{N^2}\Bigr).
	\end{multline*}

	\paragraph{\bf{Connection to random matrix theory}}
	Let $t>0$. Let $G$ be an $N\times N$ GOE matrix with parameter $t$, i.e.\ $G$ is real symmetric and
	\[
	(G_{ij})_{1\le i<j\le N}\ \text{i.i.d. } \mathcal N(0,t/N),\qquad (G_{ii})_{1\le i\le N}\ \text{i.i.d. } \mathcal N(0,2t/N),
	\]
	independent across $i\le j$.
	Equivalently, $G$ has density proportional to $\exp\{-\frac{N}{4t}\tr(G^2)\}$. Let $D_N$ be a diagonal matrix. By \eqref{eq:detA},
	\begin{equation*}
		\det(G+D_N)=\int e^{-\sum_{i,j=1}^N(G_{ij}+D_{ii}\delta_{i,j})\bar{\psi}_i \psi_j}\prod_{i=1}^N \dd\bar\psi_i \dd\psi_i.
	\end{equation*}
	Taking expectations yields
	\begin{equation}\label{eq:detE}
		\dE[\det(G+D_N)]=\int  e^{-\sum_{i=1}^N D_{ii}\bar{\psi}_i \psi_i-\frac{t}{2N}(\sum_{i=1}^N\bar{\psi}_i \psi_i)^2}\prod_{i=1}^N \dd\bar\psi_i \dd\psi_i.
	\end{equation}
	The quadratic term in the exponent can be linearized via the Hubbard--Stratonovich transform:
	\begin{equation}\label{eq:take saddle}
		e^{-\frac{t}{2N}(\sum_{i=1}^N\bar{\psi}_i \psi_i)^2}=\sqrt \frac{N}{2\pi t}\int_{\dR} e^{-\frac{N}{2t}z^2+iz\sum_{i=1}^N\bar{\psi}_i \psi_i}\dd z.
	\end{equation}
	Substituting \eqref{eq:take saddle} into \eqref{eq:detE} and integrating out the fermionic variables gives
	\begin{equation*}
		\dE[\det(G+D_N)]=\sqrt \frac{N}{2\pi t}\int_{\dR}e^{-\frac{N}{2t}z^2} \prod_{i=1}^N (d_i-iz) \dd z.
	\end{equation*}
	Suppose now that $\frac{1}{N}\sum_{i=1}^N \delta_{D_{ii}}$ converges to a probability measure $\mu$ on $\dR$ with support bounded away from $0$. The integral representation shows that $\dE[\det(G+D_N)]>0$, and by the steepest descent method one can establish that
	\begin{equation*}
		\frac{1}{N} \log \dE[\det(G+D_N)]=\frac{\omega_{\mu,t}^2(0)}{2t}+\int \log|x-\omega_{\mu,t}(0)|\dd \mu(x)+o_N(1),
	\end{equation*}
	where $\omega_{\mu,t}(0)\in \mathbb{C}^+\cup \dR$ is the subordination function of $\mu$ at time $t$ and location $0$; its definition is recalled in Lemma \ref{lemma:free co} in the Appendix. In particular, taking the saddle point in \eqref{eq:take saddle}, we deduce that whenever $\ell(\mu\boxplus\sigma_t)>0$, the typical value of $\langle \bar{\psi},\psi\rangle$ in \eqref{eq:detE} satisfies
	\begin{equation}\label{eq:approxbarpsipsi}
		\frac{1}{N}\langle \bar{\psi},\psi\rangle \approx \frac{1}{t}\omega_{\mu,t}(0).
	\end{equation}

	The value of $\omega_{\mu,t}(0)$ admits a variational characterization. Recall that the left edge of the free convolution is given by
	\begin{equation*}
		\ell(\mu\boxplus \sigma_t)=\inf\left\{x+tG_\mu(x):\int \frac{\dd \mu(\lambda)}{(x-\lambda)^2}\geq \frac{1}{t}\right\}.
	\end{equation*}
	If $\ell(\mu\boxplus\sigma_t)>0$, then $\omega_{\mu,t}(0)\in \dR$ and
	\begin{equation*}
		\frac{1}{N} \log \dE[\det(G+D_N)]=\inf_{\lambda\leq \omega_{\mu,t}(\ell(\mu\boxplus \sigma_t))}\left\{ \frac{\lambda^2}{2t}+\int\log(x-\lambda)\dd\mu(x)\right\}+o_N(1).
	\end{equation*}

	Since $\zeta|_{(q,1]}+\zeta([0,q])\delta_q$ is the minimizer of $\TAP(\mu_N^{(\bfm)},\cdot)$, Lemma \ref{lemma:first order mu} gives
	\begin{equation*}
		\xi''(q)\dE[\partial_{mm} h_\zeta(q,\cdot)^{-2}]\leq 1.
	\end{equation*}
	Since $\Vert \bfm\Vert^2=Nq$, we will show in Lemma \ref{lemma:det asymp} that this implies
	\begin{equation*}
		\omega_{T_\zeta\#\mu_N^{(\bfm)},\xi''(q)}(0)=\xi''(q)\int_q^1 \zeta([0,t])\dd t,
	\end{equation*}
	where $T_\zeta={\partial_{mm} h_\zeta}+\xi''(q)\int_q^1 \zeta([0,t])\dd t$. In view of \eqref{eq:approxbarpsipsi}, we therefore expect
	\begin{equation}\label{eq:barpsipsi}
		\frac{1}{N} \langle \bar{\psi},\psi\rangle\approx \int_q^1 \zeta([0,t])\dd t.
	\end{equation}

	\paragraph{\bf{Identifying the supersymmetric saddle via the Ward identities}}
	Combining the Ward identity \eqref{eq:Ward1} with \eqref{eq:barpsipsi}, we seek a saddle satisfying
	\begin{equation}\label{eq:Ward}
		\langle \bar{\psi},\psi\rangle=\langle x,\bfm\rangle=N \int_q^1 \zeta([0,t])\dd t.
	\end{equation}
	The remaining Ward identities further enforce $\langle \bar\psi, x\rangle=0$, $\langle \bfm,\psi\rangle=0$, $\langle\bar\psi, \bfm\rangle=0$, and $\langle x,\psi\rangle=0$. We also impose the constraint $\Vert \bfm\Vert^2=Nq$. Let $\mc{F}$ denote the subspace of $(\bfm,x,v,\bar{\psi},\psi)$ satisfying all these conditions, and define the restricted partition function
	\begin{equation*}
		I:=\int e^{\dE[S(\bfm,x,v,\bar\psi,\psi)]+\frac{1}{2}\Var[S(\bfm,x,v,\bar\psi,\psi)] }\indic_{\mc{F}}(\bfm,x,v,\bar\psi,\psi)\,\indic_{\SUSY_1(\zeta,0)}(\bfm) \dd \bfm \dd x\dd v \dd \bar{\psi} \dd \psi.
	\end{equation*}
	On $\mc{F}$, the exponent simplifies to $\mc{S}(\bfm,x,v,\bar{\psi},\psi)$, where
	\begin{multline*}
		\mc{S}(\bfm,x,v,\bar{\psi},\psi)=-\sum_{i=1}^N x_i k_\zeta(q,m_i)-\sum_{i}\bar{\psi}_i\psi_i \left(\partial_{mm}h_\zeta(q,m_i)+\xi''(q)\int_q^1 \zeta([0,t])\dd t\right)\\-v\left(\sum_{i=1}^N h_\zeta(q,m_i)+N\mc{U}_\zeta(q)+Nf\right)+\frac{1}{2}v^2 N \xi(q) + \frac{1}{2}\xi'(q)\,\|x\|^2  +v \xi'(q)\sum_{i=1}^N x_i m_i.
	\end{multline*}
	The constraints are enforced via Lagrange multipliers. One can verify that
	\begin{equation*}
		I\approx \int e^{\mc{S}(\bfm,x,v,\bar{\psi},\psi)}e^{\xi''(q)(\int_q^1 \zeta([0,t])\dd t)\left(\sum_{i=1}^N \bar{\psi}_i \psi_i+\sum_{i=1}^N x_i m_i \right)}\indic_{\SUSY_1(\zeta,0)}(\bfm) \dd \bfm \dd x\dd v \dd \bar{\psi}\dd \psi:=I'.
	\end{equation*}
	Integrating out the fermionic variables $\bar{\psi}$ and $\psi$ produces the determinant of the diagonal matrix with entries $\partial_{mm} h_\zeta(q,m_i)$, $i\in [N]$, so that
	\begin{multline*}
		I'=\int  \prod_{i=1}^N \partial_{mm} h_\zeta(q,m_i)\,\exp\biggl\{\sum_{i=1}^N \bigl(-x_i \partial_m h_\zeta(q,m_i)+v\xi'(q)x_im_i\bigr)\\
		+\tfrac{1}{2}\xi'(q)\Vert x\Vert^2-v\Bigl(\sum_{i=1}^N h_\zeta(q,m_i)+N\mc{U}_\zeta(q)+Nf\Bigr)+\tfrac{1}{2}Nv^2 \xi(q)\biggr\}\dd \bfm \dd x\dd v,
	\end{multline*}
	where we have used the identity $k_\zeta(q,m)=\partial_m h_\zeta(q,m)+m\xi''(q)\int_q^1 \zeta([0,t])\dd t$. Integrating over $x$ then gives
	\begin{equation*}
		I'=\int e^{\tilde{\mc{S}}(\bfm,v)}\dd \bfm \dd v,
	\end{equation*}
	where
	\begin{multline*}
		\tilde{\mc{S}}(\bfm,v)=\sum_{i=1}^N\log \partial_{mm} h_\zeta(q,m_i)-v\left(\sum_{i=1}^N h_\zeta(q,m_i)+N\mc{U}_\zeta(q)+Nf\right)+\frac{1}{2}v^2 N\xi(q)\\-\frac{1}{2\xi'(q)}\Vert \partial_m h_\zeta(q,\bfm)-v\xi'(q)\bfm\Vert^2.
	\end{multline*}
	A direct computation identifies the saddle as $v=u$ (i.e., $v=\zeta(\{0\})$), and yields
	\begin{equation}\label{eq:infuq}
		\log I'=uN\left(\Phi_{\zeta}(0,0)-\frac{1}{2}\int_0^1 t\xi''(t)\zeta([0,t])\dd t-f\right)+o(N).
	\end{equation}

	\paragraph{\bf{Consistency check}}
	Denote by $\mu$ the candidate measure identified above:
	\begin{equation*}
		\dd \mu(m) \propto \partial_{mm}h_\zeta(q,m)\,e^{-\frac{1}{2\xi'(q)}\partial_m h_\zeta(q,m)^2-u (h_\zeta(q,m)-\partial_m h_\zeta(q,m)m) }\indic_{[-1,1]}(m)\,\dd m.
	\end{equation*}
	Let $\nu$ be the pushforward of $\mu$ under the map $\partial_m h_\zeta(q,\cdot)$. A key observation is that $\nu=\mathrm{Law}(X_q)$, where $X_t$ is the Auffinger--Chen process
	\begin{equation*}
		\begin{cases}
			\dd X_t=\xi''(t)\zeta([0,t])\partial_x \Phi_\zeta(t,X_t)\dd t+\sqrt{\xi''(t)} \dd B_t,\\
			X_0=0.
		\end{cases}
	\end{equation*}
	Using properties of the Auffinger--Chen process (see Section \ref{section:Auffinger}) together with the optimality conditions on $\zeta$ (see the statement of Proposition \ref{prop:annealed 0}), one can verify that $\mu$ is indeed a consistent solution. We omit the details here, as a complete treatment is given in Section \ref{section:base}.

	Throughout the remainder of the paper, the Ward identity \eqref{eq:Ward} serves as a central guide for the computations.

	\subsection{Ansatz for the conditional annealed complexity}\label{sub:SUSYhier}
	The supersymmetric computation of Section \ref{sub:SUSY} can be extended to compute the number of critical points of a hierarchical system of replicas. For brevity, we record only the resulting ansatz.

	\paragraph{\bf{Partition function}}

	Let $f\in \dR$. Let $n\geq 2$ and let $(u_i)_{1\leq i\leq n}$ and $(q_i)_{1\leq i\leq n}$ be strictly increasing sequences of numbers in $(0,1)$. Let $\zeta\in \Prefix_{n+1}((u_i)_{1\leq i\leq n},(q_i)_{1\leq i\leq n})$ where $\Prefix_{n+1}$ is as in Definition \ref{def:nprefix}. Suppose that $\zeta$ satisfies the assumptions of Proposition \ref{prop:annealed cond}.

	Set $f_1=\Pari(\zeta),\ldots,f_{n-1}=\Pari(\zeta)$ and $f_n=f$. We count the number of critical points $\bfm^{(n)}$ of free energy level $f$ of square radius $q_n$ that are in the band orthogonal to some $\bfm^{(n-1)}$ critical point of level $\Pari(\zeta)$ of square radius $q_{n-1}$, which is the band orthogonal to some $\bfm^{(n-2)}$ critical point of level $\Pari(\zeta)$ of square radius $q_{n-2}$, etc. Dropping the absolute value of the determinant, a good guess for the number of critical points of this hierarchical subsystem is
	\begin{multline*}
		\dE\left[\int_{([-1,1]^N)^n } \prod_{k=1}^n \det(\nabla^2 F_{\TAP,\zeta}(\bfm^{(k)}))\right. \\ \times \left. \prod_{k=1}^n \left(\delta_{\nabla F_{\TAP,\zeta}(\bfm^{(k)})}(0)\delta_{F_{\TAP,\zeta}(\bfm^{(k)})}(Nf_k)\indic_{(\bfm^{(1)},\ldots,\bfm^{(n)})\in \SUSY_n(\zeta,0)} \right)\prod_{k=1}^n \dd \bfm^{(k)} \right].
	\end{multline*}
	As in Section \ref{sub:SUSY}, one can introduce Lagrange fields $x^{(k)}\in \dR^N$ and Lagrange parameters $\Delta_k\in \dR$, for $k\in [n]$ and fermionic fields $\psi^{(k)}$, $\bar{\psi}^{(k)}$.

	\paragraph{\bf{Summarizing the ansatz}}

	The Ward identities suggest that for the optimum,
	\begin{equation}\label{eq:an1}
		\langle x^{(i)},\bfm^{(j)}\rangle=0 \quad \text{and}\quad   \langle \bar{\psi}^{(i)},{\psi}^{(j)}\rangle=0\quad \text{for every $i\neq j \in [n]$,}
	\end{equation}
	and
	\begin{equation*}
		\langle x^{(i)},\bfm^{(i)}\rangle=\langle \bar{\psi}^{(i)},\psi^{(i)}\rangle\quad \text{for every $i\in [n]$}.
	\end{equation*}
	Moreover, as in Section \ref{sub:SUSY}, the fact that $\zeta_{\bfm^{(k)}}=\zeta|_{(q_k,1]}+\zeta([0,q_k])\delta_{q_k}$ gives
	\begin{equation}\label{eq:an2}
		\langle \bar{\psi}^{(i)},\psi^{(i)}\rangle=N\int_{q_i}^1 \zeta([0,t])\dd t=\langle x^{(i)},\bfm^{(i)}\rangle.
	\end{equation}
	Moreover, the Lagrange multipliers are expected to be given by
	\begin{equation*}
		\Delta_n=\zeta([0,q_n))\quad \text{and for every $i\in [n-1]$,}\quad \Delta_i=-\zeta(\{q_i\}).
	\end{equation*}
	Observe that for every $i\in [n]$, one has $\Delta_i+\cdots+\Delta_n=\zeta([0,q_i))$.

	\section{Properties of the Parisi PDE}\label{section:Auffinger}

	\subsection{The Auffinger--Chen SDE}
	Recall from Assumption \ref{ass:xi} that $\xi$ is smooth, strictly convex, $\xi(0)=0$ and $\xi'(0)=0$.

	The Parisi PDE~\eqref{eq:ParisiPDE} is a semilinear parabolic equation
	whose quadratic gradient nonlinearity places it in the class of
	Hamilton--Jacobi--Bellman equations arising in stochastic optimal control.
	We now make this connection precise.
	
	A \emph{feedback control} is a Borel-measurable map
	$\alpha\colon[0,1]\times\mathbb{R}\to\mathbb{R}$ such that,
	for every initial condition $(t,x)\in[0,1]\times\mathbb{R}$,
	the stochastic differential equation
	\begin{equation}\label{eq:AuffingerChen b}
		\dd X_s
		= \zeta\bigl([0,s]\bigr)\,\xi''(s)\,\alpha(s,X_s)\,\dd s
		+ \sqrt{\xi''(s)}\,\dd B_s,
		\qquad X_t = x,
	\end{equation}
	admits a unique strong solution $(X_s)_{s\in[t,1]}$ satisfying
	the square-integrability condition
	\begin{equation}\label{eq:admissibility}
		\mathbb{E}\!\left[
		\int_t^1 \xi''(s)\,\zeta\bigl([0,s]\bigr)\,\alpha(s,X_s)^2\,\dd s
		\right] < \infty.
	\end{equation}
	Denote the class of all such feedback controls by~$\mathcal{A}$.
	Consider the stochastic optimal control problem of maximizing the
	payoff functional
	\begin{equation}\label{eq:control_rep}
		\Phi_\zeta(t,x)
		= \sup_{\alpha\in\mathcal{A}}\;
		\mathbb{E}\!\left[
		\log\bigl(2\cosh(X_1)\bigr)
		- \frac{1}{2}\int_t^1
		\xi''(s)\,\zeta\bigl([0,s]\bigr)\,\alpha(s,X_s)^2
		\,\dd s
		\;\middle|\; X_t = x
		\right],
	\end{equation}
	where $(X_s)_{s\in[t,1]}$ solves~\eqref{eq:AuffingerChen b}
	under the feedback control~$\alpha$.
	By the dynamic programming principle, the value
	function~$\Phi_\zeta$ satisfies the Hamilton--Jacobi--Bellman equation
	\[
	\partial_t \Phi_\zeta
	+ \frac{\xi''(t)}{2}\,\partial_{xx}\Phi_\zeta
	+ \sup_{a\in\mathbb{R}}\Bigl\{
	\zeta\bigl([0,t]\bigr)\,\xi''(t)\,a\;\partial_x\Phi_\zeta
	- \tfrac{1}{2}\,\xi''(t)\,\zeta\bigl([0,t]\bigr)\,a^2
	\Bigr\}
	= 0.
	\]
	Optimizing the Hamiltonian pointwise in~$a$ yields a maximizer
	$a^\star = \partial_x\Phi_\zeta(t,x)$ (which is unique for time $t$ such that $\zeta([0,t])>0$), and substitution recovers the
	Parisi PDE
	\begin{equation}\label{eq:ParisiPDE}
		\begin{cases}
			\partial_t \Phi_\zeta(t,x)
			= -\dfrac{\xi''(t)}{2}\Bigl(
			\partial_{xx}\Phi_\zeta(t,x)
			+ \zeta\bigl([0,t]\bigr)\bigl(\partial_x\Phi_\zeta(t,x)\bigr)^2
			\Bigr), \\[6pt]
			\Phi_\zeta(1,x) = \log\bigl(2\cosh x\bigr).
		\end{cases}
	\end{equation}
	A standard verification argument then confirms that~$\Phi_\zeta$
	coincides with the value function~\eqref{eq:control_rep} and that the
	optimal feedback control is
	\begin{equation}\label{eq:optimal_control}
		\alpha^\star(t,x) = \partial_x\Phi_\zeta(t,x).
	\end{equation}
	Substituting~\eqref{eq:optimal_control} into~\eqref{eq:AuffingerChen b}
	gives the closed-loop dynamics
	\begin{equation}\label{eq:AuffingerChen}
		\dd X_t
		= \zeta\bigl([0,t]\bigr)\,\xi''(t)\,\partial_x\Phi_\zeta(t,X_t)\,\dd t+\sqrt{\xi''(t)}\,\dd B_t, \qquad X_0 = 0,
	\end{equation}
	which is called the \emph{Auffinger--Chen SDE}.
	In particular, the following differential identities hold.
	
	\begin{lemma}[Differential identities]\label{lemma:martingales}
		Let $X_t$ denote the Auffinger--Chen process \eqref{eq:AuffingerChen}. Then:
		\begin{enumerate}[label=(\roman*)]
			\item\label{item:mart-pathwise} Define
			\begin{equation*}
				Y_t:=\Phi_\zeta(t,X_t)-\Phi_\zeta(0,X_0)-\frac12\int_0^t \xi''(s)\,\zeta([0,s])(\partial_x\Phi_\zeta(s,X_s))^2\,\dd s-\int_0^t \partial_x\Phi_\zeta(s,X_s)\sqrt{\xi''(s)}\,\dd B_s .
			\end{equation*}
			Then $\dd Y_t=0$, i.e.\ $Y_t$ is a.s.\ constant in $t$.
			
			\item\label{item:mart-mg} The process $\bigl(\partial_x\Phi_\zeta(t,X_t)\bigr)_{t\in[0,1]}$ is a martingale.
			
			\item\label{item:mart-second} The second derivative along the optimal trajectory satisfies
			\begin{equation*}
				\dd\bigl(\partial_{xx}\Phi_\zeta(t,X_t)\bigr)
				=-\zeta([0,t])\,\xi''(t)\bigl(\partial_{xx}\Phi_\zeta(t,X_t)\bigr)^2\,\dd t
				+\sqrt{\xi''(t)}\,\partial_{xxx}\Phi_\zeta(t,X_t)\,\dd B_t.
			\end{equation*}
			
			\item\label{item:mart-square} Moreover,
			\begin{equation*}
				\dd\bigl(\partial_x\Phi_\zeta(t,X_t)\bigr)^2
				=2\sqrt{\xi''(t)}\,\partial_x\Phi_\zeta(t,X_t)\,\partial_{xx}\Phi_\zeta(t,X_t)\,\dd B_t
				+\xi''(t)\bigl(\partial_{xx}\Phi_\zeta(t,X_t)\bigr)^2\,\dd t .
			\end{equation*}
		\end{enumerate}
	\end{lemma}
	
	\begin{proof}
		We first establish~\ref{item:mart-pathwise}. By It\^o's formula, for a smooth $f(t,x)$,
		\[
		\dd f(t,X_t)=\partial_t f(t,X_t)\,\dd t+\partial_x f(t,X_t)\,\dd X_t+\frac12\,\partial_{xx}f(t,X_t)\,\dd[X]_t.
		\]
		Applying this to $f=\Phi_\zeta$ and using $\dd[X]_t=\xi''(t)\,\dd t$ together with the optimal dynamics
		\[
		\dd X_t=\zeta([0,t])\,\xi''(t)\,\partial_x\Phi_\zeta(t,X_t)\,\dd t+\sqrt{\xi''(t)}\,\dd B_t,
		\]
		we obtain
		\begin{align*}
			\dd \Phi_\zeta(t,X_t)
			&=\partial_t\Phi_\zeta(t,X_t)\,\dd t
			+\partial_x\Phi_\zeta(t,X_t)\,\dd X_t
			+\frac12\,\partial_{xx}\Phi_\zeta(t,X_t)\,\xi''(t)\,\dd t\\
			&=\Bigl(\partial_t\Phi_\zeta
			+\zeta([0,t])\,\xi''(t)\bigl(\partial_x\Phi_\zeta\bigr)^2
			+\frac12\,\xi''(t)\partial_{xx}\Phi_\zeta\Bigr)(t,X_t)\,\dd t
			+\sqrt{\xi''(t)}\,\partial_x\Phi_\zeta(t,X_t)\,\dd B_t.
		\end{align*}
		Invoking the Parisi PDE \eqref{eq:ParisiPDE} to substitute for $\partial_t\Phi_\zeta$, the drift simplifies to
		\[
		\frac12\,\zeta([0,t])\,\xi''(t)\bigl(\partial_x\Phi_\zeta(t,X_t)\bigr)^2\,\dd t,
		\]
		and therefore
		\[
		\dd \Phi_\zeta(t,X_t)
		=\frac12\,\zeta([0,t])\,\xi''(t)(\partial_x\Phi_\zeta(t,X_t))^2\,\dd t+\sqrt{\xi''(t)}\partial_x\Phi_\zeta(t,X_t) \,\dd B_t.
		\]
		By construction of $Y_t$, this is exactly equivalent to $\dd Y_t=0$.
		
		We next prove~\ref{item:mart-mg}. Applying It\^o's formula to $f=\partial_x\Phi_\zeta$ yields
		\begin{align*}
			\dd\bigl(\partial_x\Phi_\zeta(t,X_t)\bigr)
			&=\partial_{tx}\Phi_\zeta(t,X_t)\,\dd t
			+\partial_{xx}\Phi_\zeta(t,X_t)\,\dd X_t
			+\frac12\,\partial_{xxx}\Phi_\zeta(t,X_t)\,\xi''(t)\,\dd t\\
			&=\Bigl(\partial_{tx}\Phi_\zeta
			+\zeta([0,t])\,\xi''(t)\,\partial_{xx}\Phi_\zeta\,\partial_x\Phi_\zeta
			+\tfrac12\,\xi''(t)\partial_{xxx}\Phi_\zeta\Bigr)(t,X_t)\,\dd t\\
			&\quad+\sqrt{\xi''(t)}\,\partial_{xx}\Phi_\zeta(t,X_t)\,\dd B_t.
		\end{align*}
		Differentiating \eqref{eq:ParisiPDE} in $x$ gives
		\[
		\partial_{tx}\Phi_\zeta(t,x)
		=-\frac{\xi''(t)}{2}\Bigl(\partial_{xxx}\Phi_\zeta(t,x)
		+2\,\zeta([0,t])\,\partial_{xx}\Phi_\zeta(t,x)\,\partial_x\Phi_\zeta(t,x)\Bigr).
		\]
		Substituting this identity into the drift term above cancels it identically, leaving
		\[
		\dd\bigl(\partial_x\Phi_\zeta(t,X_t)\bigr)=\sqrt{\xi''(t)}\,\partial_{xx}\Phi_\zeta(t,X_t)\,\dd B_t.
		\]
		Hence $\partial_x\Phi_\zeta(t,X_t)$ is a local martingale; under the standing integrability conditions ensuring square-integrability of the stochastic integral, it is a martingale.
		
		We now prove~\ref{item:mart-second}. Set
		\[
		V(t,x):=\partial_{xx}\Phi_\zeta(t,x).
		\]
		Applying It\^o's formula to $V(t,X_t)$ gives
		\begin{align*}
			\dd V(t,X_t)
			&=\partial_t V(t,X_t)\,\dd t+\partial_x V(t,X_t)\,\dd X_t
			+\frac12\,\partial_{xx}V(t,X_t)\,\dd[X]_t\\
			&=\partial_{txx}\Phi_\zeta(t,X_t)\,\dd t
			+\partial_{xxx}\Phi_\zeta(t,X_t)\,\dd X_t
			+\frac12\,\partial_{xxxx}\Phi_\zeta(t,X_t)\,\xi''(t)\,\dd t.
		\end{align*}
		Using the optimal dynamics
		\[
		\dd X_t=\zeta([0,t])\,\xi''(t)\,\partial_x\Phi_\zeta(t,X_t)\,\dd t
		+\sqrt{\xi''(t)}\,\dd B_t,
		\]
		we obtain
		\begin{align*}
			\dd V(t,X_t)
			&=\Bigl(\partial_{txx}\Phi_\zeta
			+\zeta([0,t])\,\xi''(t)\,\partial_x\Phi_\zeta\,\partial_{xxx}\Phi_\zeta
			+\frac12\,\xi''(t)\,\partial_{xxxx}\Phi_\zeta\Bigr)(t,X_t)\,\dd t\\
			&\qquad\qquad +\sqrt{\xi''(t)}\,\partial_{xxx}\Phi_\zeta(t,X_t)\,\dd B_t.
		\end{align*}
		Next differentiate \eqref{eq:ParisiPDE} twice in $x$. Since
		\[
		\partial_{xx}\bigl(\partial_x\Phi_\zeta\bigr)^2
		=2\bigl(\partial_{xx}\Phi_\zeta\bigr)^2
		+2\,\partial_x\Phi_\zeta\,\partial_{xxx}\Phi_\zeta,
		\]
		we have
		\[
		\partial_{txx}\Phi_\zeta(t,x)
		=-\frac{\xi''(t)}{2}\Bigl(
		\partial_{xxxx}\Phi_\zeta(t,x)
		+2\,\zeta([0,t])\Bigl((\partial_{xx}\Phi_\zeta(t,x))^2
		+\partial_x\Phi_\zeta(t,x)\,\partial_{xxx}\Phi_\zeta(t,x)\Bigr)\Bigr).
		\]
		Substituting this into the drift above, the $\partial_{xxxx}\Phi_\zeta$ terms cancel,
		and the $\partial_x\Phi_\zeta\,\partial_{xxx}\Phi_\zeta$ terms cancel as well, leaving
		\[
		\dd V(t,X_t)
		=-\zeta([0,t])\,\xi''(t)\,\bigl(\partial_{xx}\Phi_\zeta(t,X_t)\bigr)^2\,\dd t
		+\sqrt{\xi''(t)}\,\partial_{xxx}\Phi_\zeta(t,X_t)\,\dd B_t,
		\]
		which is exactly the claimed identity in~\ref{item:mart-second}.
		
		Finally,~\ref{item:mart-square} is an immediate consequence of It\^o's formula applied to the square of the martingale in~\ref{item:mart-mg}, namely
		\(
		\dd(u_t^2)=2u_t\,\dd u_t+\dd[u]_t
		\)
		with $u_t=\partial_x\Phi_\zeta(t,X_t)$.
	\end{proof}

	\begin{lemma}\label{lemma:xxphi}
		Let $\zeta\in\mc P([0,1])$ and fix $q\in[0,1]$.
		Let $(X_t)_{t\in[q,1]}$ solve the Auffinger--Chen SDE 
		\[
		\dd X_t=\xi''(t)\,\zeta([0,t])\,\partial_x\Phi_\zeta(t,X_t)\dd t+\sqrt{\xi''(t)}\,\dd B_t.
		\]
		Write $u(t,y):=\partial_x\Phi_\zeta(t,y)$. Then for every $x\in\dR$,
		\begin{equation}\label{eq:xxphi-general}
			\partial_{xx}\Phi_\zeta(q,x)
			=1-\zeta\bigl([0,q)\bigr)\,u(q,x)^2
			-\int_{[q,1]}\E\!\left[u(t,X_t)^2\,\middle|\,X_q=x\right]\zeta(\dd t).
		\end{equation}
		Consequently,
		\[
		\E\big[\partial_{xx}\Phi_\zeta(q,X_q)\big]
		=1-\zeta\bigl([0,q)\bigr)\,\E\big[u(q,X_q)^2\big]
		-\int_{[q,1]}\E\big[u(t,X_t)^2\big]\,\zeta(\dd t).
		\]
		Moreover, if $\E[u(t,X_t)^2]=t$ for $\zeta$-a.e.\ $t\in[q,1]$
		and also $\E[u(q,X_q)^2]=q$,
		then
		\[
		\E\big[\partial_{xx}\Phi_\zeta(q,X_q)\big]
		=1-\zeta\bigl([0,q)\bigr)\,q-\int_{[q,1]}t\,\zeta(\dd t)
		=\int_q^1 \zeta([0,t])\dd t.
		\]
	\end{lemma}
	
	\begin{proof}
		Let $u(t,y)=\partial_x\Phi_\zeta(t,y)$ and set
		$V(t):=\E\!\left[u(t,X_t)^2\mid X_q=x\right]$.
		From Lemma~\ref{lemma:martingales}\ref{item:mart-second},
		\[
		\dd\big(\partial_{xx}\Phi_\zeta(t,X_t)\big)
		=-\zeta([0,t])\,\xi''(t)\big(\partial_{xx}\Phi_\zeta(t,X_t)\big)^2\dd t
		+\sqrt{\xi''(t)}\,\partial_{xxx}\Phi_\zeta(t,X_t)\dd B_t.
		\]
		Taking conditional expectation given $X_q=x$ and integrating from $q$ to $1$ yields
		\begin{equation}\label{eq:xxphi-step1-gen}
			\E\!\left[\partial_{xx}\Phi_\zeta(1,X_1)\mid X_q=x\right]-\partial_{xx}\Phi_\zeta(q,x)
			=-\int_q^1 \zeta([0,t])\,\xi''(t)\,
			\E\!\left[\big(\partial_{xx}\Phi_\zeta(t,X_t)\big)^2\mid X_q=x\right]\dd t.
		\end{equation}
		From Lemma~\ref{lemma:martingales}\ref{item:mart-square},
		\[
		\dd\big(u(t,X_t)^2\big)
		=2\sqrt{\xi''(t)}\,u(t,X_t)\,\partial_{xx}\Phi_\zeta(t,X_t)\dd B_t
		+\xi''(t)\big(\partial_{xx}\Phi_\zeta(t,X_t)\big)^2\dd t,
		\]
		so $V$ is absolutely continuous on $[q,1]$ with
		\begin{equation}\label{eq:Vprime}
			V'(t)=\xi''(t)\,\E\!\left[\big(\partial_{xx}\Phi_\zeta(t,X_t)\big)^2\mid X_q=x\right].
		\end{equation}
		Substituting \eqref{eq:Vprime} into \eqref{eq:xxphi-step1-gen} gives
		\begin{equation}\label{eq:xxphi-step2-gen}
			\E\!\left[\partial_{xx}\Phi_\zeta(1,X_1)\mid X_q=x\right]-\partial_{xx}\Phi_\zeta(q,x)
			=-\int_q^1 \zeta([0,t])\,V'(t)\dd t
			=-\int_{(q,1]} \zeta([0,t])\dd V(t).
		\end{equation}
		Since $V$ is continuous, Stieltjes integration by parts yields
		\[
		\int_{(q,1]} \zeta([0,t])\dd V(t)
		=\zeta([0,1])\,V(1)-\zeta([0,q])\,V(q)-\int_{(q,1]} V(t)\,\zeta(\dd t).
		\]
		Using $\zeta([0,1])=1$ and decomposing
		$\zeta([0,q])=\zeta([0,q))+\zeta(\{q\})$, we obtain
		\begin{align}
			\E\!\left[\partial_{xx}\Phi_\zeta(1,X_1)\mid X_q=x\right]-\partial_{xx}\Phi_\zeta(q,x)
			&=-V(1)+\zeta([0,q))\,V(q)+\zeta(\{q\})\,V(q)+\int_{(q,1]} V(t)\,\zeta(\dd t) \notag\\
			&=-V(1)+\zeta([0,q))\,V(q)+\int_{[q,1]} V(t)\,\zeta(\dd t),
			\label{eq:xxphi-step3-gen}
		\end{align}
		where the last step recombines the atom $\zeta(\{q\})\,V(q)$ with the integral
		over $(q,1]$ into a single integral over $[q,1]$.
		
		Now $\partial_{xx}\Phi_\zeta(1,y)=1-\tanh^2(y)=1-u(1,y)^2$, so
		$\E[\partial_{xx}\Phi_\zeta(1,X_1)\mid X_q=x]=1-V(1)$.
		Also $V(q)=u(q,x)^2$.
		Substituting into \eqref{eq:xxphi-step3-gen} and rearranging:
		\begin{align*}
			\partial_{xx}\Phi_\zeta(q,x)
			&=\bigl(1-V(1)\bigr)+V(1)-\zeta([0,q))\,u(q,x)^2-\int_{[q,1]} V(t)\,\zeta(\dd t)\\
			&=1-\zeta\bigl([0,q)\bigr)\,u(q,x)^2-\int_{[q,1]} V(t)\,\zeta(\dd t),
		\end{align*}
		which is \eqref{eq:xxphi-general}.
		
		The second statement follows by taking expectation over $X_q$.
		For the final claim, substitute $\E[u(t,\allowbreak X_t)^2]=t$ and $\E[u(q,\allowbreak X_q)^2]=q$ to get
		\[
		\E\big[\partial_{xx}\Phi_\zeta(q,X_q)\big]
		=1-\zeta\bigl([0,q)\bigr)\,q-\int_{[q,1]}t\,\zeta(\dd t).
		\]
		To verify this equals $\int_q^1\zeta([0,t])\dd t$, write
		\begin{equation*}
			1-\zeta\bigl([0,q)\bigr)\,q-\int_{[q,1]}t\,\zeta(\dd t)
			=1-\int_{[0,q)}q\,\zeta(\dd t)-\int_{[q,1]}t\,\zeta(\dd t).
		\end{equation*}
		More directly, note that $\int_{[0,1]}1\,\zeta(\dd t)=1$ and therefore
		\[
		1-\zeta([0,q))\,q-\int_{[q,1]}t\,\zeta(\dd t)
		=\int_{[0,q)}(1-q)\,\zeta(\dd t)+\int_{[q,1]}(1-t)\,\zeta(\dd t).
		\]
		Using the layer-cake identity
		$\int_A(1-t)\,\zeta(\dd t)=\int_A\int_t^1\dd s\,\zeta(\dd t)$
		and Fubini's theorem,
		\begin{align*}
			\int_{[0,q)}(1-q)\,\zeta(\dd t)+\int_{[q,1]}(1-t)\,\zeta(\dd t)
			&=\int_{[0,q)}\int_q^1\dd s\,\zeta(\dd t)+\int_{[q,1]}\int_t^1\dd s\,\zeta(\dd t)\\
			&=\int_q^1\zeta([0,q))\dd s+\int_q^1\zeta([q,s])\dd s\\
			&=\int_q^1\bigl(\zeta([0,q))+\zeta([q,s])\bigr)\dd s\\
			&=\int_q^1\zeta([0,s])\dd s,
		\end{align*}
		as claimed.
	\end{proof}

	\begin{lemma}\label{lemma:law Xt s}
		Let $\zeta\in\mc P([0,1])$ and let $(X_t)_{t\in[0,1]}$ solve the Auffinger--Chen SDE
		\[
		\dd X_t=\xi''(t)\zeta([0,t])\,\partial_x \Phi_\zeta(t,X_t)\,\dd t+\sqrt{\xi''(t)}\,\dd B_t,
		\qquad X_0=0.
		\]
		Fix $0\le s<t\le 1$ and assume that
		\[
		\zeta((s,t))=0,
		\]
		so that $\zeta([0,u])=\zeta([0,s])$ for every $u\in[s,t)$. Set $\alpha:=\zeta([0,s])$.
		Then, conditionally on $X_s=x$, the law of $X_t$ is absolutely continuous with respect to Lebesgue measure and satisfies
		\[
		\mathrm{Law}(X_t\mid X_s=x)(\dd y)\ \propto\
		\exp\!\left(-\frac{(y-x)^2}{2(\xi'(t)-\xi'(s))}+\alpha\,\Phi_\zeta(t,y)\right)\dd y.
		\]
		Equivalently, writing $p^{BM}_{s,t}(x,y)$ for the transition density of the inhomogeneous Brownian motion
		$\dd Y_u=\sqrt{\xi''(u)}\,\dd W_u$ (so that $\Var(Y_t-Y_s)=\xi'(t)-\xi'(s)$), one has
		\[
		p_{s,t}(x,y)=\frac{e^{\alpha \Phi_\zeta(t,y)}}{e^{\alpha \Phi_\zeta(s,x)}}\,p^{BM}_{s,t}(x,y),
		\qquad
		p^{BM}_{s,t}(x,y)=\frac{1}{\sqrt{2\pi(\xi'(t)-\xi'(s))}}
		\exp\!\left(-\frac{(y-x)^2}{2(\xi'(t)-\xi'(s))}\right).
		\]
	\end{lemma}

	\begin{proof}
		Set $\alpha:=\zeta([0,s])$ and define the Hopf--Cole transform
		\[
		H(u,y):=\exp\bigl(\alpha\,\Phi_\zeta(u,y)\bigr).
		\]
		Since $\zeta((s,t))=0$, we have $\zeta([0,u])=\alpha$ for all $u\in[s,t)$. Multiplying the Parisi PDE
		by $\alpha$ and using the identities
		\[
		\partial_u H=\alpha H\,\partial_u\Phi_\zeta,
		\qquad
		\partial_{yy}H=\alpha H\,\partial_{yy}\Phi_\zeta+\alpha^2 H\,(\partial_y\Phi_\zeta)^2,
		\]
		one checks that $H$ solves the backward heat equation on $[s,t]$:
		\begin{equation}\label{eq:H-backward-heat}
			\partial_u H(u,y)+\frac12\,\xi''(u)\,\partial_{yy}H(u,y)=0,
			\qquad u\in[s,t).
		\end{equation}
		
		Let $(Y_u)_{u\in[s,t]}$ be the inhomogeneous Brownian motion started at $Y_s=x$,
		\[
		\dd Y_u=\sqrt{\xi''(u)}\,\dd W_u,\qquad Y_s=x,
		\]
		defined on a filtered probability space with law denoted $P^{BM}_{s,x}$. By It\^o's formula and
		\eqref{eq:H-backward-heat}, the process $(H(u,Y_u))_{u\in[s,t]}$ is a positive $P^{BM}_{s,x}$--martingale. 
		Define a new probability measure $Q_{s,x}$ on $\mc F_t$ by
		\begin{equation}\label{eq:Doob-h-transform}
			\frac{\dd Q_{s,x}}{\dd P^{BM}_{s,x}}\Big|_{\mc F_t}=\frac{H(t,Y_t)}{H(s,x)}.
		\end{equation}
		Standard Doob $h$--transform (equivalently, Girsanov's theorem applied to \eqref{eq:Doob-h-transform})
		implies that under $Q_{s,x}$, the process $(Y_u)_{u\in[s,t]}$ satisfies
		\[
		\dd Y_u=\xi''(u)\,\partial_y\log H(u,Y_u)\,\dd u+\sqrt{\xi''(u)}\,\dd \widetilde W_u
		=\alpha\,\xi''(u)\,\partial_y\Phi_\zeta(u,Y_u)\,\dd u+\sqrt{\xi''(u)}\,\dd \widetilde W_u,
		\]
		where $(\widetilde W_u)$ is a $Q_{s,x}$--Brownian motion. Since $\zeta([0,u])=\alpha$ for $u\in[s,t]$,
		this is precisely the Auffinger--Chen dynamics on the plateau $[s,t]$. Consequently, the conditional law of
		$(X_u)_{u\in[s,t]}$ given $X_s=x$ coincides with the law of $(Y_u)_{u\in[s,t]}$ under $Q_{s,x}$.
		
		In particular, for any bounded measurable $f$,
		\[
		\E\bigl[f(X_t)\mid X_s=x\bigr]
		=\E_{Q_{s,x}}\bigl[f(Y_t)\bigr]
		=\E_{P^{BM}_{s,x}}\!\left[\frac{H(t,Y_t)}{H(s,x)}\,f(Y_t)\right]
		=\frac{1}{H(s,x)}\int_{\mathbb R} H(t,y)\,f(y)\,p^{BM}_{s,t}(x,y)\,\dd y.
		\]
		Therefore the conditional transition density satisfies
		\[
		p_{s,t}(x,y)=\frac{H(t,y)}{H(s,x)}\,p^{BM}_{s,t}(x,y)
		=\frac{e^{\alpha\Phi_\zeta(t,y)}}{e^{\alpha\Phi_\zeta(s,x)}}\,p^{BM}_{s,t}(x,y),
		\]
		and substituting the explicit Gaussian form of $p^{BM}_{s,t}$ yields the stated proportionality formula.
	\end{proof}

	\begin{lemma}\label{lemma:expectations}
		Let $(X_t)$ solve the Auffinger--Chen SDE with $X_0=0$ and set $M_t:=\partial_x\Phi_\zeta(t,X_t)$.
		Let $0\le s<t\le 1$.
		\begin{enumerate}[label=(\roman*)]
			\item\label{item:exp-deltaXM}
			\begin{equation}\label{eq:deltaX M}
				\dE[(X_t-X_s)M_s]
				=\dE[M_s^2]\int_s^t\zeta([0,u])\,\xi''(u)\,\dd u.
			\end{equation}
			
			\item\label{item:exp-deltaMX}
			\begin{equation}\label{eq:deltaM X}
				\dE[(M_t-M_s)\,X_s]=0.
			\end{equation}
			
			\item\label{item:exp-XtMt}
			\begin{equation}\label{eq:XtMt-master}
				\dE[X_t M_t]
				=\xi'(t)-\int_{[0,1]}\xi'\bigl(\min(t,u)\bigr)\,\dE[M_u^2]\;\zeta(\dd u).
			\end{equation}
			
			Consequently, for all $s,t\in [0,1]$,
			\begin{equation}\label{eq:XsMt-general}
				\dE[X_s\, M_t]
				=\begin{cases}
					\dE[X_s M_s], & s\le t,\\[4pt]
					\dE[X_t M_t]+\dE[M_t^2]\displaystyle\int_t^s\zeta([0,u])\,\xi''(u)\,\dd u, & s\ge t.
				\end{cases}
			\end{equation}
			
			\item\label{item:exp-product} Suppose that $\dE[\partial_x\Phi_\zeta(t,X_t)^2]=t$  for $\zeta$-a.e. $t\in [0,1]$. Suppose $\zeta((s,t))=0$ and $s,t\in\supp(\zeta)$. Then
			\begin{equation}\label{eq:deltaM deltaX}
				\dE[(M_t-M_s)(X_t-X_s)]
				=\bigl(\xi'(t)-\xi'(s)\bigr)\int_s^1\zeta([0,u])\,\dd u.
			\end{equation}
		\end{enumerate}
	\end{lemma}
	
	\begin{proof}
		Throughout, $(M_t)$ is a martingale by Lemma~\ref{lemma:martingales}\ref{item:mart-mg}.
		
		\medskip\noindent\textit{Proof of~\ref{item:exp-deltaXM}.}
		From the SDE, $X_t-X_s=\int_s^t\zeta([0,u])\,\xi''(u)\,M_u\,\dd u+\int_s^t\sqrt{\xi''(u)}\,\dd B_u$.
		Multiplying by $M_s$ and taking expectations, the stochastic integral vanishes
		(being independent of~$\mc F_s$), and the martingale property
		$\dE[M_u\mid\mc F_s]=M_s$ for $u\ge s$ gives
		\[
		\dE[(X_t-X_s)M_s]
		=\int_s^t\zeta([0,u])\,\xi''(u)\,\dE[M_s\,M_u]\,\dd u
		=\dE[M_s^2]\int_s^t\zeta([0,u])\,\xi''(u)\,\dd u.
		\]
		
		\medskip\noindent\textit{Proof of~\ref{item:exp-deltaMX}.}
		Since $X_s$ is $\mc F_s$-measurable and $\dE[M_t-M_s\mid\mc F_s]=0$, this is immediate.
		
		\medskip\noindent\textit{Proof of~\ref{item:exp-XtMt}.}
		By It\^o's product rule,
		\[
		\dd(X_t M_t)
		=M_t\,\dd X_t+X_t\,\dd M_t+\dd[X,M]_t.
		\]
		Using $\dd M_t=\sqrt{\xi''(t)}\,\partial_{xx}\Phi_\zeta(t,X_t)\,\dd B_t$
		and $\dd[X,M]_t=\xi''(t)\,\partial_{xx}\Phi_\zeta(t,X_t)\,\dd t$,
		the drift of $\dd(X_t M_t)$ is
		\[
		\xi''(t)\Bigl(\zeta([0,t])\,M_t^2+\partial_{xx}\Phi_\zeta(t,X_t)\Bigr)\dd t.
		\]
		Taking expectations and integrating from $0$ to $t$ (with $X_0=0$):
		\begin{equation}\label{eq:XtMt-ODE}
			\dE[X_t M_t]
			=\int_0^t\xi''(s)\Bigl(\zeta([0,s])\,\dE[M_s^2]
			+\dE[\partial_{xx}\Phi_\zeta(s,X_s)]\Bigr)\dd s.
		\end{equation}
		By Lemma~\ref{lemma:xxphi},
		\[
		\dE[\partial_{xx}\Phi_\zeta(s,X_s)]
		=1-\zeta\bigl([0,s)\bigr)\,\dE[M_s^2]-\int_{[s,1]}\dE[M_u^2]\,\zeta(\dd u).
		\]
		Since $\zeta([0,s])-\zeta([0,s))=\zeta(\{s\})=0$ for Lebesgue-a.e.\ $s$,
		the integrand in \eqref{eq:XtMt-ODE} simplifies $\dd s$-a.e.\ to
		\[
		\xi''(s)\Bigl(1-\int_{s}^1\dE[M_u^2]\,\zeta(\dd u)\Bigr).
		\]
		Hence
		\[
		\dE[X_t M_t]
		=\xi'(t)-\int_0^t\xi''(s)\int_{s}^1\dE[M_u^2]\,\zeta(\dd u)\,\dd s.
		\]
		Exchanging the order of integration, the inner Lebesgue integral
		over $s\in[0,\min(t,u)]$ evaluates to $\xi'(\min(t,u))$, yielding
		\eqref{eq:XtMt-master}.
		
		The formula \eqref{eq:XsMt-general} follows: the case $s\le t$ is
		\ref{item:exp-deltaMX}; the case $s\ge t$ combines $\dE[X_s M_t]=\dE[X_t M_t]+\dE[(X_s-X_t)M_t]$
		with~\ref{item:exp-deltaXM}.
		
		\medskip\noindent\textit{Proof of~\ref{item:exp-product}.}
		By~\ref{item:exp-deltaXM} and~\ref{item:exp-deltaMX},
		\begin{multline}
			\dE[(M_t-M_s)(X_t-X_s)]
			=\dE[M_t(X_t-X_s)]-\dE[M_s(X_t-X_s)]
			\\=\bigl(\dE[X_t M_t]-\dE[X_s M_s]\bigr)-\dE[M_s^2]\int_s^t\zeta([0,u])\,\xi''(u)\,\dd u,
		\end{multline}
		where we used $\dE[M_t X_s]=\dE[M_s X_s]$ from~\ref{item:exp-deltaMX}.
		Since $\zeta((s,t))=0$, we have $\zeta([0,u])=\zeta([0,s])=:m$ for $u\in[s,t)$ and
		$\int_{[r,1]}\dE[M_u^2]\,\zeta(\dd u)$ is constant in $r\in(s,t)$.
		From the proof of~\ref{item:exp-XtMt}, the difference reduces to
		\[
		\dE[X_t M_t]-\dE[X_s M_s]
		=\bigl(\xi'(t)-\xi'(s)\bigr)
		\Bigl(1-\int_{[t,1]}\dE[M_u^2]\,\zeta(\dd u)\Bigr).
		\]
		Subtracting $\dE[M_s^2]\cdot m\bigl(\xi'(t)-\xi'(s)\bigr)$ and
		using $s,t\in\supp(\zeta)$ so that $\dE[M_s^2]=s$ (by assumption and since $s\mapsto \dE[M_s^2]$ is continuous), we get
		\[
		1-\int_{t}^1\dE[M_u^2]\,\zeta(\dd u)
		=\dE[\partial_{xx}\Phi_\zeta(t,X_t)]+\zeta([0,t))\,t
		=\int_t^1\zeta([0,u])\,\dd u+mt.
		\]
		Substituting, we obtain
		\[
		\dE[(M_t-M_s)(X_t-X_s)]
		=\bigl(\xi'(t)-\xi'(s)\bigr)\Bigl(\int_t^1\zeta([0,u])\,\dd u+mt-sm\Bigr)
		=\bigl(\xi'(t)-\xi'(s)\bigr)\int_s^1\zeta([0,u])\,\dd u,
		\]
		since $\zeta([0,u])=m$ on $[s,t)$.
	\end{proof}

	\begin{lemma}[Parisi flow preserves convexity]\label{lemma:convexity}
		Let $g:\dR\to\dR$ be $C^2$ with $g''\ge \kappa$ for some $\kappa\in\dR$, and let
		$\Phi_\zeta$ be the solution of the Parisi PDE run backward on $[t,1]$ with terminal
		condition $\Phi_\zeta(1,\cdot)=g$, where $\xi''\ge 0$ and $\zeta\ge 0$.
		Then for every $s\in[t,1]$,
		\[
		\inf_{x\in\dR}\,\partial_{xx}\Phi_\zeta(s,x)
		\;\ge\;
		\inf_{x\in\dR}\,g''(x)
		\;\ge\;
		\kappa.
		\]
		In particular, $x\mapsto\Phi_\zeta(s,x)$ is everywhere at least as convex as $g$.
	\end{lemma}
	
	\begin{proof}
		\textbf{Step 1: step-function $\zeta$.}
		On a layer $[a,b]$ where $t\mapsto\zeta([0,t])$ is constant equal to $m$, set $\sigma^2:=\int_a^b\xi''(r)\,dr$
		and let $Z\sim \mc{N}(0,1)$. The corresponding backward map is
		\[
		\mathcal{T}_{m,\sigma}(f)(x)
		:=
		\begin{cases}
			\E\,f(x+\sigma Z), & m=0,\\[4pt]
			\dfrac{1}{m}\log\E\exp\!\bigl\{mf(x+\sigma Z)\bigr\}, & m>0.
		\end{cases}
		\]
		Suppose $f''\ge\kappa$. For $m=0$ the claim is immediate:
		$(\mathcal{T}_{0,\sigma}f)''(x)=\E f''(x+\sigma Z)\ge\kappa$.
		For $m>0$, let $Q_x$ denote the probability measure whose density (with respect to the
		law of $Z$) is proportional to $e^{mf(x+\sigma Z)}$. A direct computation gives
		\[
		(\mathcal{T}_{m,\sigma}f)''(x)
		=\E_{Q_x}\!\bigl[f''(x+\sigma Z)\bigr]
		+m\,\mathrm{Var}_{Q_x}\!\bigl(f'(x+\sigma Z)\bigr)
		\;\ge\;\kappa,
		\]
		where the inequality uses $f''\ge\kappa$ for the first term and nonnegativity of
		variance for the second. Hence every layer preserves the lower curvature bound $\kappa$.
		For a step function $\zeta$, the solution satisfies
		$\Phi_\zeta(s,\cdot)=\mathcal{T}_{m_1,\sigma_1}\circ\cdots\circ\mathcal{T}_{m_p,\sigma_p}(g)$
		over the layers covering $[s,1]$, so $\inf_x\,\partial_{xx}\Phi_\zeta(s,x)\ge\kappa$
		follows by induction.
		
		\textbf{Step 2: general $\zeta$.}
		For general $\zeta$, choose probability measures $\zeta_n$ with finitely many atoms such that $\zeta_n\to\zeta$.
		The corresponding solutions $\Phi_{\zeta_n}$ converge to $\Phi_\zeta$ locally uniformly
		(see, e.g., \cite[Proposition~2.3]{jagannath2017some}).
		Since each $\Phi_{\zeta_n}(s,\cdot)-\frac{\kappa}{2}(\cdot)^2$ is convex, so is its locally uniform limit $\Phi_\zeta(s,\cdot)-\frac{\kappa}{2}(\cdot)^2$, which gives $\partial_{xx}\Phi_\zeta(s,\cdot)\ge\kappa$.
	\end{proof}
	
	\subsection{Optimality conditions for the TAP correction}\label{sub:TAP correction}
	
	In this subsection, we study the variational problem $\zeta\mapsto \TAP(\mu,\zeta)$.
	
	Recall that we abuse notation and write $\zeta([0,t])$ for $\zeta([q,t])$ when $\zeta\in \mc{P}([q,1])$ and $t\in [q,1]$.
	
	\begin{lemma}\label{lemma:first order mu}
		Let $q\in [0,1)$, $\bfm\in (-1,1)^N$, and set $\mu:=\frac{1}{N}\sum_{i=1}^N \delta_{m_i}\in\mc{P}((-1,1))$. Let $\zeta\in \mc{P}([q,1])$ and let $(X_t^{\mu})_{t\in [q,1]}$ be the Auffinger--Chen process with law-matching condition
		\begin{equation}\label{eq:AC mu}
			\begin{cases}    \dd X_t^\mu=\xi''(t)\,\zeta([0,t])\,\partial_x \Phi_\zeta(t,X^\mu_t)\,\dd t+\sqrt{\xi''(t)}\, \dd B_t,\\
				\mathrm{Law}\bigl(\partial_x\Phi_\zeta(q,X^\mu_q)\bigr)=\mu.
			\end{cases}
		\end{equation}
		Recall $\TAP(\mu,\zeta)$ from \eqref{def:TAPmuzeta}. The following assertions hold.
		\begin{enumerate}[label=(\roman*)]
			\item\label{item:convexmu} The map $\zeta \mapsto \TAP(\mu,\zeta)$ is strictly convex on $\mc{P}([q,1])$.
			\item\label{item:gateaumu} For every finite signed measure $\eta$ on $[q,1]$ with $\eta([q,1])=0$, the Gâteaux derivative of $\TAP(\mu,\cdot)$ at $\zeta$ in the direction $\eta$ is given by 
			\begin{equation}\label{eq:gateau2}
				D\TAP(\mu,\cdot)(\zeta)[\eta]=\frac{1}{2}\int_q^1 \xi''(s)\Bigl(\dE\bigl[(\partial_x\Phi_\zeta(s,X^\mu_s))^2\bigr]-s \Bigr)\,\eta([0,s])\,\dd s.
			\end{equation}
			\item\label{item:firstordermu} For every $s\in [q,1]$, define 
			\begin{equation}\label{def:Hmu}
				H^\mu_\zeta(s):=\frac{1}{2}\int_s^1 \xi''(r)\Bigl(\dE\bigl[(\partial_x\Phi_\zeta(r,X^\mu_r))^2\bigr]-r\Bigr)\dd r,
			\end{equation}
			and extend it to $[0,1]$ by setting
			\begin{equation}
				\label{def:Hhatmu}    \hat{H}_\zeta^\mu(s):=\begin{cases}
					{H}_\zeta^\mu(q) & \text{if $s\in [0,q)$,}\\
					{H}_\zeta^\mu(s) & \text{if $s\in [q,1]$.}
				\end{cases}
			\end{equation}
			Then $\zeta\in \mc{P}([q,1])$ is the unique minimizer of $\TAP(\mu,\cdot)$ if and only if for every $\zeta'\in \mc{P}([q,1])$,
			\begin{equation*}
				\int_q^1 \hat{H}^\mu_{\zeta}(s)\, \dd \zeta'(s) \geq  \int_q^1 \hat{H}^\mu_{\zeta}(s)\, \dd \zeta(s).
			\end{equation*}
			Equivalently, $\zeta\in \mc{P}([q,1])$ is the unique minimizer of $\TAP(\mu,\cdot)$ if and only if 
			\begin{equation}\label{eq:supp}
				\supp(\zeta)\subset \underset{[0,1]}{\mathrm{arg\,min}}\, \hat{H}_\zeta^\mu.
			\end{equation}
			\item\label{item:optmu} Let $\zeta\in \mc{P}([q,1])$ be the unique minimizer of $\TAP(\mu,\cdot)$. Then, for every $q'\in \supp(\zeta)$,
			\begin{equation}\label{eq:opt3}
				\dE\bigl[(\partial_x\Phi_\zeta(q',X^\mu_{q'}))^2\bigr]=q'.
			\end{equation}
			(For interior points $q'\in (q,1)$ this is a first-order optimality condition; at the boundary $q'=q$ it follows from the law-matching condition~\eqref{eq:AC mu}, since $\dE[(\partial_x\Phi_\zeta(q,X_q^\mu))^2]=\int m^2\,\dd\mu(m)=q$.)
			\item\label{item:2ndordermu} Let $\zeta\in \mc{P}([q,1])$ be the unique minimizer of $\TAP(\mu,\cdot)$. Then, for every $q'\in \supp(\zeta)$,
			\begin{equation}\label{eq:second order}
				\xi''(q')\,\dE\bigl[(\partial_{xx}\Phi_\zeta(q',X^\mu_{q'}))^2\bigr]\leq 1.
			\end{equation}
			\item\label{item:qsupp} Let $\zeta\in\mc{P}([q,1])$ be the unique minimizer of $\TAP(\mu,\cdot)$. Then
			\begin{equation}\label{eq:qinsupport}
				q\in \supp(\zeta).
			\end{equation}
		\end{enumerate}
	\end{lemma}
	
	\begin{remark}\label{remark:delta0}
		Notice the crucial property that 
		\begin{equation*}
			\TAP(\delta_0,\zeta)=\Phi_\zeta(0,0)-\frac{1}{2}\int_0^1 t\xi''(t)\zeta([0,t])\dd t=\Pari(\zeta).
		\end{equation*}
		Therefore, the minimizer of the right-hand side of the last display satisfies \eqref{eq:opt3}, \eqref{eq:second order} and \eqref{eq:qinsupport}.
	\end{remark}

	\begin{proof}
		Assertion~\ref{item:convexmu} follows from \cite[Lemma~42]{chen2018generalized}.
		
		\medskip
		\textit{Proof of~\ref{item:gateaumu}.} Set $m(s)=\zeta([0,s])$. Fix a finite signed measure $\eta$ on $[q,1]$ with $\eta([q,1])=0$ and set $h(s):=\eta([0,s])$, $\zeta_\varepsilon:=\zeta+\varepsilon\eta$, $m_\varepsilon:=m+\varepsilon h$, $\Phi_\varepsilon:=\Phi_{\zeta_\varepsilon}$, $u_\varepsilon:=\partial_x\Phi_\varepsilon$, and $u:=\partial_x\Phi_\zeta$. Define the difference quotient
		\[
		\Delta_\varepsilon:=\frac{\Phi_\varepsilon-\Phi_0}{\varepsilon},\qquad \Delta_\varepsilon(1,\cdot)=0,
		\]
		where $\Phi_0:=\Phi_\zeta$. One can check that $\Delta_\varepsilon$ converges locally uniformly to the solution $\Psi$ of
		\[
		\partial_s \Psi
		= -\tfrac12\xi''(s)\Big(\partial_{xx}\Psi+2m(s)\,u(s,\cdot)\,\partial_x\Psi+h(s)\,u(s,\cdot)^2\Big),
		\qquad \Psi(1,\cdot)=0.
		\]
		By the Feynman--Kac representation theorem applied to the operator
		\[
		\mathcal L_s f=\tfrac12\xi''(s)\bigl(\partial_{xx}f+2m(s)\,u(s,\cdot)\,\partial_x f\bigr),
		\]
		for every $s\in [q,1]$ and $x\in \dR$,
		\begin{equation}\label{eq:Psisx}
			\Psi(s,x)=\frac{1}{2} \int_s^1 \xi''(r)\,h(r)\,\dE_{s,x}\bigl[(\partial_x\Phi_\zeta(r,X_r))^2\bigr]\,\dd r,
		\end{equation}
		where $\dE_{s,x}$ denotes the expectation for the diffusion with generator $\mc{L}_s$ started at position $x$ at time $s$.
		
		We now differentiate the Legendre transform term. Recall that 
		\[
		\Phi_{\zeta_\varepsilon}^*(q,y)=\sup_x\bigl\{xy-\Phi_{\zeta_\varepsilon}(q,x)\bigr\}.
		\]
		By the envelope theorem (the supremum being attained at a unique point by strict convexity of $\Phi_\zeta(q,\cdot)$), the derivative with respect to $\varepsilon$ at $\varepsilon=0$ is computed by differentiating only the integrand at the optimal $x$. For every $m'\in (-1,1)$, the supremum defining $\Phi_\zeta^*(q,m')$ is attained at $x(m'):=(\partial_x\Phi_\zeta(q,\cdot))^{-1}(m')$, hence
		\begin{equation*}
			\frac{\dd}{\dd \varepsilon}\bigg|_{\varepsilon=0} \Phi_{\zeta_\varepsilon}^*(q,m')=- \frac{\dd}{\dd \varepsilon}\bigg|_{\varepsilon=0}\Phi_{\zeta_\varepsilon}(q,x(m'))=-\Psi(q,x(m')).
		\end{equation*}
		Averaging against $\mu$ gives
		\begin{equation}\label{eq:xY}
			\frac{\dd}{\dd \varepsilon}\bigg|_{\varepsilon=0}\int \Phi_{\zeta_\varepsilon}^*(q,m')\,\dd \mu(m')=- \frac{1}{N}\sum_{i=1}^N\Psi(q,x(m_i)).
		\end{equation}
		By the construction of the Auffinger--Chen process~\eqref{eq:AC mu}, the law of $X^\mu_q$ is chosen so that $\partial_x\Phi_\zeta(q,X_q^\mu)\sim \mu$, which implies $X_q^\mu\sim (\partial_x\Phi_\zeta(q,\cdot))^{-1}\#\mu$. Therefore,
		\begin{equation*}
			\frac{1}{N}\sum_{i=1}^N\Psi(q,x(m_i))=\dE[\Psi(q,X^\mu_q)].  
		\end{equation*}
		Inserting~\eqref{eq:Psisx} (with $s=q$) into~\eqref{eq:xY}, we obtain
		\begin{equation*}
			\frac{\dd}{\dd \varepsilon}\bigg|_{\varepsilon=0}\int \Phi_{\zeta_\varepsilon}^*(q,m')\,\dd \mu(m')=-\frac{1}{2}\int_q^1 \xi''(r)\,h(r)\,\dE\bigl[(\partial_{x}\Phi_\zeta(r,X^\mu_r))^2\bigr]\,\dd r.
		\end{equation*}
		Adding the derivative of the linear term $-\frac{1}{2}\int_q^1 s\,\xi''(s)\,h(s)\,\dd s$ yields~\eqref{eq:gateau2}.
		
		\medskip
		\textit{Proof of~\ref{item:firstordermu}.}
		By Fubini's theorem, one can rewrite~\eqref{eq:gateau2} as
		\begin{equation*}
			D\TAP(\mu,\cdot)(\zeta)[\eta]=\int_{[q,1]}H_\zeta^\mu(s)\,\dd \eta(s)+H^\mu_\zeta(q)\,\eta([0,q)).
		\end{equation*}
		By the strict convexity established in~\ref{item:convexmu}, the measure $\zeta$ minimizes $\TAP(\mu,\cdot)$ over $\mc{P}([q,1])$ if and only if $D\TAP(\mu,\cdot)(\zeta)[\zeta'-\zeta]\geq 0$ for every $\zeta'\in \mc{P}([q,1])$, that is,
		\begin{equation*}
			\int_{[q,1]}  H^\mu_{\zeta}(s)\, \dd \zeta'(s)+H_\zeta^\mu(q)\,\zeta'([0,q)) \geq \int_{[q,1]}  H^\mu_{\zeta}(s)\, \dd \zeta(s)+H_\zeta^\mu(q)\,\zeta([0,q)).
		\end{equation*}
		Recalling that $\zeta,\zeta'\in \mc{P}([q,1])$ (so both may charge $[0,q)$ through their definition on the full interval, but here $\zeta'([0,q))=0$ and $\zeta([0,q))=0$), the terms involving $\eta([0,q))$ vanish. More precisely, introducing
		\begin{equation*}
			\hat{H}_\zeta^\mu(s):=\begin{cases}
				{H}_\zeta^\mu(q) & \text{if $s\in [0,q)$,}\\
				{H}_\zeta^\mu(s) & \text{if $s\in [q,1]$,}
			\end{cases}
		\end{equation*}
		the optimality condition reads: for every $\zeta'\in \mc{P}([q,1])$,
		\begin{equation*}
			\int_q^1 \hat{H}^\mu_{\zeta}(s)\, \dd \zeta'(s) \geq  \int_q^1 \hat{H}^\mu_{\zeta}(s)\, \dd \zeta(s).
		\end{equation*}
		The equivalence with the support condition~\eqref{eq:supp} is classical: the above condition states that $\zeta$ minimizes the linear functional $\zeta'\mapsto \int \hat{H}^\mu_\zeta\,\dd\zeta'$ over $\mc{P}([q,1])$, which holds if and only if $\zeta$ is supported on the set of minimizers of $\hat{H}^\mu_\zeta$ over $[0,1]$.
		
		\medskip
		\textit{Proof of~\ref{item:optmu}.}
		Since $q'\in \supp(\zeta)\subset \mathrm{arg\,min}\,\hat{H}_\zeta^\mu$ by~\eqref{eq:supp} and $q'\in [q,1]$, the function $H_\zeta^\mu$ attains its minimum at $q'$. In particular $\frac{\dd}{\dd s}H_\zeta^\mu(s)\big|_{s=q'}=0$. Differentiating~\eqref{def:Hmu} gives
		\[
		\frac{\dd}{\dd s}H_\zeta^\mu(s)=\frac{1}{2}\xi''(s)\Bigl(s-\dE\bigl[(\partial_x\Phi_\zeta(s,X^\mu_s))^2\bigr]\Bigr).
		\]
		Since $\xi''(q')>0$, the vanishing of this derivative at $s=q'$ yields~\eqref{eq:opt3}.
		
		\medskip
		\textit{Proof of~\ref{item:2ndordermu}.}
		As in the proof of~\ref{item:optmu}, $H_\zeta^\mu$ attains its minimum at every $q'\in \supp(\zeta)$, so $\frac{\dd^2}{\dd s^2}H_\zeta^\mu(s)\big|_{s=q'}\geq 0$ (at the boundary $q'=q$, this is the right second derivative, which is well-defined since the first-order condition~\eqref{eq:opt3} holds and $H_\zeta^\mu$ is smooth on $(q,q+\delta)$ for small $\delta$). Set $u(s,x):=\partial_x\Phi_\zeta(s,x)$. By Lemma~\ref{lemma:martingales}\ref{item:mart-square} (whose derivation depends only on the Parisi PDE and the drift structure of the process, and therefore applies equally to $(X_t^\mu)$ in place of $(X_t)$),
		\begin{equation}\label{eq:diffexpmu}
			\frac{\dd}{\dd s}\dE\bigl[u(s,X^\mu_s)^2\bigr]=\xi''(s)\,\dE\bigl[(\partial_{xx}\Phi_\zeta(s,X^\mu_s))^2\bigr].
		\end{equation}
		Differentiating~\eqref{def:Hmu} twice and using~\eqref{eq:opt3} (which ensures the first-order term vanishes) together with~\eqref{eq:diffexpmu}, we obtain
		\[
		\frac{\dd^2}{\dd s^2}H_\zeta^\mu(s)\bigg|_{s=q'}=\frac{1}{2}\xi''(q')\bigl(1-\xi''(q')\,\dE\bigl[(\partial_{xx}\Phi_\zeta(q',X^\mu_{q'}))^2\bigr]\bigr).
		\]
		Since $\xi''(q')>0$ and $\frac{\dd^2}{\dd s^2}H_\zeta^\mu(s)\big|_{s=q'}\geq 0$, this yields~\eqref{eq:second order}.
		
		\medskip
		\textit{Proof of~\ref{item:qsupp}.}
		This is proved in~\cite{chen2018generalized}; see Theorem~12, item~(5) therein.
	\end{proof}

	\subsection{Gradient of the TAP correction}
	
	We now compute the gradient of the TAP correction.

	\begin{lemma}[Bounds for the defect term]\label{lemma:Delta-bound}
		Fix $q\in[0,1)$ and let $\mu$ be a probability measure on $[-1,1]$ with
		$\int m^2\,\mu(\dd m)=q$.
		For $\zeta\in\mc P([q,1])$, let $(X_t^{\mu,\zeta})_{t\in[q,1]}$ be the
		Auffinger--Chen process \eqref{eq:AC mu}, set
		\[
		u_\zeta(t,x):=\partial_x\Phi_\zeta(t,x),
		\]
		and define
		\begin{equation}\label{eq:def-Delta}
			\Delta_\zeta^\mu
			:=
			\int_{[q,1]}
			\Bigl(\dE\bigl[u_\zeta(t,X_t^{\mu,\zeta})^2\bigr]-t\Bigr)\,\zeta(\dd t).
		\end{equation}
		Then:
		\begin{enumerate}[label=(\roman*)]
			\item\label{it:Delta-crude}
			For every $\zeta\in\mc P([q,1])$,
			\[
			|\Delta_\zeta^\mu|\le 1.
			\]
			
			\item\label{it:Delta-explicit}
			Let $\zeta_\mu$ be the unique minimizer of $\TAP(\mu,\cdot)$ on $\mc P([q,1])$.
			Assume that $\mu$ is supported in $[-1+\delta,1-\delta]$ for some
			$\delta\in(0,1)$, and let $\zeta\in\mc P([q,1])$ satisfy
			\[
			\zeta(\{q\})\ge u_0>0.
			\]
			Then there exists $C=C(\xi,u_0)$ such that
			\[
			|\Delta_\zeta^\mu|
			\le \frac{C}{\delta}\,\dist(\zeta,\zeta_\mu).
			\]
		\end{enumerate}
	\end{lemma}
	
	\begin{proof}
		For $\zeta\in\mc P([q,1])$, write
		\[
		g_\zeta(t):=\dE\bigl[u_\zeta(t,X_t^{\mu,\zeta})^2\bigr]-t,
		\qquad t\in[q,1],
		\]
		so that
		\[
		\Delta_\zeta^\mu=\int_{[q,1]} g_\zeta(t)\,\zeta(\dd t).
		\]
		
		\smallskip
		\noindent\emph{Proof of \ref{it:Delta-crude}.}
		As in Lemma~\ref{lemma:martingales}\ref{item:mart-mg}, the process
		$u_\zeta(t,X_t^{\mu,\zeta})$ is a martingale on $[q,1]$ with terminal value
		$\tanh(X_1^{\mu,\zeta})\in[-1,1]$. Hence
		\[
		|u_\zeta(t,X_t^{\mu,\zeta})|\le 1
		\qquad\text{a.s. for every }t\in[q,1].
		\]
		Therefore $|g_\zeta(t)|\le 1$ for every $t$, and so
		\[
		|\Delta_\zeta^\mu|
		\le \int_{[q,1]} |g_\zeta(t)|\,\zeta(\dd t)
		\le 1.
		\]
		
		\smallskip
		\noindent\emph{Proof of \ref{it:Delta-explicit}.}
		Set
		\[
		\bar\zeta:=\zeta_\mu,\qquad
		u:=u_\zeta,\qquad
		\bar u:=u_{\bar\zeta},\qquad
		X:=X^{\mu,\zeta},\qquad
		\bar X:=X^{\mu,\bar\zeta},
		\qquad
		d:=\dist(\zeta,\bar\zeta).
		\]
		We decompose
		\[
		\Delta_\zeta^\mu
		=
		\int g_{\bar\zeta}\,\dd\zeta
		+
		\int (g_\zeta-g_{\bar\zeta})\,\dd\zeta.
		\]
		
		First we bound the term involving $g_{\bar\zeta}$.
		By Lemma~\ref{lemma:first order mu}\ref{item:optmu},
		\[
		g_{\bar\zeta}(t)=0
		\qquad\text{for every }t\in\supp(\bar\zeta),
		\]
		hence
		\[
		\int g_{\bar\zeta}\,\dd\bar\zeta=0.
		\]
		Also, Lemma~\ref{lemma:convexity} and Lemma~\ref{lemma:xxphi} (applied at time
		$s$) imply
		\[
		0\le \partial_x\bar u(s,x)=\partial_{xx}\Phi_{\bar\zeta}(s,x)\le 1
		\qquad\text{for all }(s,x)\in[q,1]\times\dR.
		\]
		Therefore, exactly as in Lemma~\ref{lemma:martingales}\ref{item:mart-square},
		for $q\le s<t\le 1$,
		\[
		\dE\bigl[\bar u(t,\bar X_t)^2\bigr]
		-
		\dE\bigl[\bar u(s,\bar X_s)^2\bigr]
		=
		\int_s^t
		\xi''(r)\,
		\dE\bigl[(\partial_{xx}\Phi_{\bar\zeta}(r,\bar X_r))^2\bigr]\dd r,
		\]
		and hence
		\[
		\Bigl|
		\dE\bigl[\bar u(t,\bar X_t)^2\bigr]
		-
		\dE\bigl[\bar u(s,\bar X_s)^2\bigr]
		\Bigr|
		\le \|\xi''\|_\infty\,|t-s|.
		\]
		Thus $g_{\bar\zeta}$ is Lipschitz on $[q,1]$ with
		\[
		\Lip(g_{\bar\zeta})\le 1+\|\xi''\|_\infty.
		\]
		Since $\dist$ is the $1$-Wasserstein distance on $[q,1]$, the
		Kantorovich--Rubinstein duality gives
		\[
		\Bigl|\int g_{\bar\zeta}\,\dd\zeta\Bigr|
		=
		\Bigl|\int g_{\bar\zeta}\,\dd\zeta-\int g_{\bar\zeta}\,\dd\bar\zeta\Bigr|
		\le (1+\|\xi''\|_\infty)\,d.
		\]
		
		Next we bound the term involving $g_\zeta-g_{\bar\zeta}$.
		Standard stability estimates for the Parisi PDE
		(see, e.g., \cite[Proof of Theorem~12(ii)]{chen2018generalized})
		give a constant $C_0=C_0(\xi)$ such that
		\[
		\sup_{(t,x)\in[q,1]\times\dR}|u(t,x)-\bar u(t,x)|\le C_0\,d.
		\]
		
		Assume first that
		\[
		d\le \frac{\delta}{2C_0}.
		\]
		Since $\zeta([0,q))=0$, Lemma~\ref{lemma:xxphi}, \eqref{eq:xxphi-general},
		gives for every $x\in\dR$,
		\begin{align*}
			\partial_x u(q,x)
			&=
			1-\int_{[q,1]}
			\E\!\left[u(t,X_t)^2\,\middle|\,X_q=x\right]\zeta(\dd t) \\
			&\ge
			1-\zeta(\{q\})u(q,x)^2-\zeta((q,1]) \\
			&=
			\zeta(\{q\})\bigl(1-u(q,x)^2\bigr)
			\ge
			u_0\bigl(1-u(q,x)^2\bigr).
		\end{align*}
		
		Now couple $X$ and $\bar X$ using the same Brownian motion and the same random
		variable $M\sim\mu$, by setting
		\[
		X_q=(\partial_x\Phi_\zeta(q,\cdot))^{-1}(M),
		\qquad
		\bar X_q=(\partial_x\Phi_{\bar\zeta}(q,\cdot))^{-1}(M).
		\]
		Then
		\[
		u(q,X_q)=M=\bar u(q,\bar X_q)
		\qquad\text{a.s.}
		\]
		Hence
		\[
		|u(q,\bar X_q)-M|
		=
		|u(q,\bar X_q)-\bar u(q,\bar X_q)|
		\le C_0\,d
		\le \delta/2.
		\]
		Because $M\in[-1+\delta,1-\delta]$ a.s., both endpoint values
		$u(q,X_q)$ and $u(q,\bar X_q)$ belong to $[-1+\delta/2,1-\delta/2]$.
		Since $x\mapsto u(q,x)$ is nondecreasing, every value of $u(q,\cdot)$ on the
		interval between $X_q$ and $\bar X_q$ also belongs to
		$[-1+\delta/2,1-\delta/2]$. Therefore, on that interval,
		\[
		\partial_x u(q,\cdot)
		\ge
		u_0\Bigl(1-(1-\delta/2)^2\Bigr)
		\ge \frac{u_0\delta}{2}.
		\]
		By the mean value theorem,
		\[
		|X_q-\bar X_q|
		\le
		\frac{2}{u_0\delta}\,
		|u(q,X_q)-u(q,\bar X_q)|
		\le
		\frac{2C_0}{u_0\delta}\,d.
		\]
		Taking expectations gives the same bound for $\dE|X_q-\bar X_q|$.
		
		Using the coupled SDEs, the bound $|u|\le 1$, and the fact that
		$x\mapsto\bar u(s,x)$ is $1$-Lipschitz for every $s\in[q,1]$, we obtain for
		$t\in[q,1]$,
		\begin{align*}
			\dE|X_t-\bar X_t|
			&\le
			\dE|X_q-\bar X_q|
			+
			\|\xi''\|_\infty\int_q^t |\zeta([0,s])-\bar\zeta([0,s])|\,\dd s \\
			&\qquad
			+
			\|\xi''\|_\infty\int_q^t
			\Bigl(
			\sup_x|u(s,x)-\bar u(s,x)|
			+
			\dE|X_s-\bar X_s|
			\Bigr)\,\dd s.
		\end{align*}
		Since
		\[
		\int_q^1 |\zeta([0,s])-\bar\zeta([0,s])|\,\dd s=d,
		\]
		the previous bounds and Gr\"onwall's lemma yield
		\[
		\sup_{t\in[q,1]}\dE|X_t-\bar X_t|
		\le \frac{C_1(\xi,u_0)}{\delta}\,d.
		\]
		Consequently, for every $t\in[q,1]$,
		\begin{align*}
			|g_\zeta(t)-g_{\bar\zeta}(t)|
			&=
			\Bigl|
			\dE\bigl[u(t,X_t)^2-\bar u(t,\bar X_t)^2\bigr]
			\Bigr| \\
			&\le
			2\,\dE|u(t,X_t)-\bar u(t,\bar X_t)| \\
			&\le
			2\sup_x|u(t,x)-\bar u(t,x)|
			+
			2\,\dE|\bar u(t,X_t)-\bar u(t,\bar X_t)| \\
			&\le
			2C_0\,d+2\,\dE|X_t-\bar X_t|
			\le
			\frac{C_2(\xi,u_0)}{\delta}\,d.
		\end{align*}
		Hence, when $d\le \delta/(2C_0)$,
		\[
		\Bigl|\int (g_\zeta-g_{\bar\zeta})\,\dd\zeta\Bigr|
		\le
		\sup_{t\in[q,1]}|g_\zeta(t)-g_{\bar\zeta}(t)|
		\le
		\frac{C_2(\xi,u_0)}{\delta}\,d.
		\]
		Combining this with the bound on $\int g_{\bar\zeta}\,\dd\zeta$ gives
		\[
		|\Delta_\zeta^\mu|
		\le
		(1+\|\xi''\|_\infty)\,d+\frac{C_2(\xi,u_0)}{\delta}\,d
		\le
		\frac{C(\xi,u_0)}{\delta}\,d,
		\]
		since $\delta\le 1$.
		
		If instead $d>\delta/(2C_0)$, then part~\ref{it:Delta-crude} yields
		\[
		|\Delta_\zeta^\mu|\le 1\le \frac{2C_0}{\delta}\,d.
		\]
		This proves part~\ref{it:Delta-explicit}.
	\end{proof}

	\begin{lemma}\label{lemma:gradient TAP}
		Recall $\TAP$ from \eqref{def:TAPmu} and $k_\zeta$ from~\eqref{def:kqm}.
		Fix $\bfm\in(-1,1)^N$ and set $q=\frac{1}{N}\Vert \bfm\Vert^2$.
		Let $\zeta$ be a probability measure on $[q,1]$. Then for every $i\in[N]$:
		\begin{enumerate}[label=(\roman*)]
			\item (\emph{Exact gradient at the minimizer}) If $\zeta_\bfm$ is the minimizer in Definition~\ref{def:zetam}, then
			\begin{equation}\label{eq:gradmin}
				\partial_i \TAP(\mu_N^{(\bfm)})=-\frac{1}{N}k_{\zeta_\bfm}(q,m_i).
			\end{equation}
			
			\item (\emph{Exact defect formula and explicit boundary stability}) One has
			\begin{equation}\label{eq:gradTap-defect}
				-N\partial_i \TAP(\mu_N^{(\bfm)},\zeta)
				=
				k_\zeta(q,m_i)-m_i\,\xi''(q)\,\Delta_\zeta^{\mu_N^{(\bfm)}},
			\end{equation}
			where $\Delta_\zeta^{\mu_N^{(\bfm)}}$ is defined in \eqref{eq:def-Delta}.
			Consequently, if
			\[
			\max_{j\in[N]}|m_j|\le 1-\delta
			\qquad\text{for some }\delta\in(0,1),
			\]
			and
			\[
			\zeta(\{q\})\ge u_0>0,
			\]
			then there exists $C=C(\xi,u_0)<\infty$ such that
			\begin{equation}\label{eq:gradTap-stability}
				\Bigl|\partial_i \TAP(\mu_N^{(\bfm)},\zeta)+\frac{1}{N}k_{\zeta}(q,m_i)\Bigr|
				\le \frac{C}{N\delta}\,\dist(\zeta_\bfm,\zeta).
			\end{equation}
		\end{enumerate}
	\end{lemma}
	
	\begin{proof}
		Since $q=\frac{1}{N}\Vert \bfm\Vert^2$ depends on~$\bfm$, we have
		$\partial_i q=\frac{2m_i}{N}$. Differentiating $\TAP(\mu_N^{(\bfm)},\zeta)$
		with respect to $m_i$ by the chain rule yields
		\begin{equation}\label{eq:diG raw}
			-N \partial_i \TAP(\mu_N^{(\bfm)},\zeta)
			=
			\partial_m h_\zeta(q,m_i)
			+
			\frac{2m_i}{N}
			\left(
			\sum_{j=1}^N \partial_q h_\zeta(q,m_j)
			+
			N\,\partial_q\mc{U}_\zeta(q)
			\right).
		\end{equation}
		A direct computation gives
		\[
		\partial_q\mc{U}_\zeta(q)=-\frac{1}{2}\,q\,\xi''(q)\,\zeta([0,q]).
		\]
		
		Recall that
		\[
		h_\zeta(q,a)=\sup_{x\in\dR}\bigl(ax-\Phi_\zeta(q,x)\bigr),
		\qquad a\in(-1,1).
		\]
		Since $\Phi_\zeta(q,\cdot)$ is strictly convex, the supremum is attained at the
		unique point
		\[
		x(a):=(\partial_x \Phi_\zeta(q,\cdot))^{-1}(a).
		\]
		By the envelope theorem,
		\[
		\partial_q h_\zeta(q,a)=-\partial_q \Phi_\zeta\bigl(q,x(a)\bigr).
		\]
		(When $\zeta$ has an atom at $q$, $\partial_q$ denotes the right derivative; the Parisi PDE determines it unambiguously.)
		Using the Parisi PDE~\eqref{eq:ParisiPDE},
		\begin{equation}\label{eq:PDE identity}
			\partial_q \Phi_\zeta(q,x)
			=-\frac{\xi''(q)}{2}
			\Bigl(
			\partial_{xx}\Phi_\zeta(q,x)
			+
			\zeta([0,q])\,\bigl(\partial_x \Phi_\zeta(q,x)\bigr)^2
			\Bigr),
		\end{equation}
		we obtain
		\begin{equation}\label{eq:dqh}
			\partial_q h_\zeta(q,a)
			=
			\frac{\xi''(q)}{2}
			\Bigl(
			\partial_{xx}\Phi_\zeta\bigl(q,x(a)\bigr)
			+
			\zeta([0,q])\,a^2
			\Bigr).
		\end{equation}
		
		Let
		\[
		\nu_N^{(\bfm)}:=(\partial_x\Phi_\zeta(q,\cdot))^{-1}\#\,\mu_N^{(\bfm)}.
		\]
		Substituting~\eqref{eq:dqh} into~\eqref{eq:diG raw} and using
		$\frac{1}{N}\sum_{j=1}^N m_j^2=q$, we get
		\begin{align}
			-N\partial_i \TAP(\mu_N^{(\bfm)},\zeta)
			&=
			\partial_m h_\zeta(q,m_i)
			+
			m_i\,\xi''(q)
			\int \partial_{xx}\Phi_\zeta(q,x)\,\dd\nu_N^{(\bfm)}(x).
			\label{eq:diG reduced}
		\end{align}
		
		By Lemma~\ref{lemma:xxphi},
		\[
		\int \partial_{xx}\Phi_\zeta(q,x)\,\dd\nu_N^{(\bfm)}(x)
		=
		\int_q^1 \zeta([0,s])\,\dd s
		-
		\int_{[q,1]}
		\Bigl(
		\dE\bigl[u_\zeta(t,X_t^{\mu_N^{(\bfm)},\zeta})^2\bigr]-t
		\Bigr)\,
		\zeta(\dd t),
		\]
		that is,
		\begin{equation}\label{eq:xx}
			\int \partial_{xx}\Phi_\zeta(q,x)\,\dd\nu_N^{(\bfm)}(x)
			=
			\int_q^1 \zeta([0,s])\,\dd s
			-
			\Delta_\zeta^{\mu_N^{(\bfm)}}.
		\end{equation}
		Inserting this into~\eqref{eq:diG reduced} yields
		\eqref{eq:gradTap-defect}.
		
		If $\zeta=\zeta_\bfm$, then by Lemma~\ref{lemma:first order mu}\ref{item:optmu},
		\[
		\dE\bigl[u_{\zeta_\bfm}(t,X_t^{\mu_N^{(\bfm)},\zeta_\bfm})^2\bigr]=t
		\qquad\text{for every }t\in \supp(\zeta_\bfm),
		\]
		hence $\Delta_{\zeta_\bfm}^{\mu_N^{(\bfm)}}=0$. Therefore
		\eqref{eq:gradTap-defect} gives
		\[
		-N\partial_i \TAP(\mu_N^{(\bfm)},\zeta_\bfm)=k_{\zeta_\bfm}(q,m_i).
		\]
		Since $\zeta_\bfm$ is the minimizer in the definition of $\TAP(\mu_N^{(\bfm)})$,
		the envelope theorem implies
		\[
		\partial_i \TAP(\mu_N^{(\bfm)})
		=
		\partial_i \TAP(\mu_N^{(\bfm)},\zeta_\bfm),
		\]
		which proves \eqref{eq:gradmin}.
		
		Finally, assume that $\max_{j\in[N]}|m_j|\le 1-\delta$ and
		$\zeta(\{q\})\ge u_0>0$. Then $\mu_N^{(\bfm)}$ is supported in
		$[-1+\delta,1-\delta]$, so Lemma~\ref{lemma:Delta-bound} gives
		\[
		|\Delta_\zeta^{\mu_N^{(\bfm)}}|
		\le \frac{C(\xi,u_0)}{\delta}\,\dist(\zeta,\zeta_\bfm).
		\]
		Using \eqref{eq:gradTap-defect} and $|m_i|\le 1$, we obtain
		\[
		\Bigl|\partial_i \TAP(\mu_N^{(\bfm)},\zeta)+\frac{1}{N}k_\zeta(q,m_i)\Bigr|
		=
		\frac{|m_i|\,\xi''(q)}{N}\,|\Delta_\zeta^{\mu_N^{(\bfm)}}|
		\le
		\frac{C}{N\delta}\,\dist(\zeta,\zeta_\bfm),
		\]
		which proves \eqref{eq:gradTap-stability}.
	\end{proof}

	\subsection{Stability of the TAP correction}
	
	In this subsection, we show that the minimizer of $\TAP(\mu,\cdot)$ depends continuously on the empirical measure $\mu$ and on the overlap $q$. This stability is used in the main proofs to pass from finite-$N$ empirical measures to their deterministic limits.
	
	Recall the metric $\dist$ on $\mc{P}([0,1])$ \eqref{def:dist}.

	\begin{lemma}[Stability of the TAP minimizer]\label{lemma:stableTapMin}
		Fix $q\in (0,1)$ and let $\mu\in \mc{P}([-1,1])$ with $\int m^2\,\mu(\dd m)=q$. Let $\zeta_\star\in\mc{P}([q,1])$ be the unique minimizer of $\TAP(\mu,\cdot)$
		and for every $N$, let $\zeta_N$ be the unique minimizer of $\TAP(\mu_N,\cdot)$,
		where $q_N:=\int m^2\,\dd\mu_N(m)$.
		If $q_N\to q$ and $\mu_N\Rightarrow \mu$ weakly, then
		\begin{equation*}
			\dist(\zeta_N,\zeta_\star)\xrightarrow[N\to\infty]{} 0.
		\end{equation*}
		In particular, if $\int m^2\,\mu(\dd m)=q$,
		$m_1,\ldots,m_N$ are i.i.d.\ with law~$\mu$,
		$q_N=\frac1N\sum_{i=1}^N m_i^2$,
		$\mu_N=\frac1N\sum_{i=1}^N\delta_{m_i}$, and $\zeta_\bfm$ denotes the minimizer in
		Definition~\ref{def:zetam}, then for every $\ve>0$,
		\begin{equation*}
			\mu^{\otimes N}\!\left(\dist(\zeta_\bfm,\zeta_\star)>\ve\right)\xrightarrow[N\to\infty]{} 0.
		\end{equation*}
	\end{lemma}
	
	\begin{proof}
		This follows from \cite[Proof of Lemma~22]{chen2018generalized}.
	\end{proof}

	\section{Random matrices computations}\label{section:RMT comp}

	\subsection{Conditional law of the Hessian}
	
	In this subsection, we decompose the Hessian $\nabla^2 H(\bfm)$ into blocks along $\bfm$ and $\bfm^\perp$, and compute the conditional law of each block given the value and gradient of $H$ at $\bfm$.
	
	\begin{lemma}\label{lemma:conditional hessian}
		Fix $\bfm\in[-1,1]^N$ and set $q:=\|\bfm\|^2/N\in[0,1]$.
		Assume $q>0$ (the case $q=0$ does not arise in our applications).
		Choose an orthonormal basis
		\[
		e_1=\frac{\bfm}{\sqrt{Nq}},\qquad e_2,\dots,e_N\in \bfm^\perp,
		\]
		and decompose the Hessian $\nabla^2 H(\bfm)$ in this basis as
		\[
		\nabla^2 H(\bfm)=
		\begin{pmatrix}
			A & B^{\!\top}\\[1mm]
			B & C
		\end{pmatrix},
		\quad A\in\dR,\quad B\in\dR^{N-1},\quad C\in\dR^{(N-1)\times(N-1)}.
		\]
		Condition on $H(\bfm)=Nf$ and $\nabla H(\bfm)=x$.
		Write $x_\parallel:=\langle x,e_1\rangle$ and let
		\[
		x_\perp := \bigl(\langle x,e_2\rangle,\dots,\langle x,e_N\rangle\bigr)\in\dR^{N-1}
		\]
		be the coordinate vector of the projection of $x$ onto $\bfm^\perp$ in the basis $(e_2,\dots,e_N)$.
		The following assertions hold.
		\begin{enumerate}[label=(\roman*)]
			\item\label{item:tangent} \textup{(Tangent block.)}
			The block $C$ is independent of $\bigl(H(\bfm),\nabla H(\bfm)\bigr)$, and its entries satisfy
			\[
			\Cov(C_{ij},C_{kl})=\frac{\xi''(q)}{N}\bigl(\delta_{ik}\delta_{jl}+\delta_{il}\delta_{jk}\bigr),
			\qquad i,j,k,l\in\{2,\dots,N\}.
			\]
			Equivalently (under the convention that $\GOE_{n}$ has off-diagonal variance $1$ and diagonal variance $2$),
			\[
			C\ \stackrel{d}{=}\ \sqrt{\xi''(q)}\;\frac{\GOE_{N-1}}{\sqrt{N}}.
			\]
			
			\item\label{item:mixed} \textup{(Mixed block.)}
			Conditionally on $\bigl(H(\bfm),\nabla H(\bfm)\bigr)$, the vector $B$ is Gaussian and independent of $C$, with
			\begin{equation}\label{eq:B-mean}
				\dE\!\left[B\;\middle|\;H(\bfm)=Nf,\;\nabla H(\bfm)=x\right]
				=\frac{\xi''(q)}{\xi'(q)}\,\sqrt{\frac{q}{N}}\; x_\perp,
			\end{equation}
			\begin{equation}\label{eq:B-cov}
				\Var\!\left(B\;\middle|\;H(\bfm)=Nf,\;\nabla H(\bfm)=x\right)
				=\frac{1}{N}\!\left(\xi''(q)+q\,\xi^{(3)}(q)-\frac{\xi''(q)^2\,q}{\xi'(q)}\right) I_{N-1}.
			\end{equation}
			
			\item\label{item:longitudinal} \textup{(Longitudinal entry.)}
			Assume that $\det\Sigma_X>0$ (this holds whenever $\xi$ is a mixed model).
			Define
			\begin{equation}\label{eq:SigmaX}
				\Sigma_X:=
				\begin{pmatrix}
					N\xi(q) & \xi'(q)\sqrt{Nq}\\[1mm]
					\xi'(q)\sqrt{Nq} & \xi'(q)+q\,\xi''(q)
				\end{pmatrix},
			\end{equation}
			\begin{equation}\label{eq:SigmaAX}
				\Sigma_{A,X}:=\Bigl(\;\xi''(q)\,q\;,\;\bigl(\xi^{(3)}(q)\,q+2\,\xi''(q)\bigr)\sqrt{q/N}\;\Bigr),
			\end{equation}
			\begin{equation}\label{eq:VarA}
				\Var(A)=\frac{2\,\xi''(q)}{N}+\frac{4\,\xi^{(3)}(q)\,q}{N}+\frac{\xi^{(4)}(q)\,q^2}{N}.
			\end{equation}
			Then, conditionally on $\bigl(H(\bfm),\nabla H(\bfm)\bigr)$, the entry $A$ is Gaussian and independent of $C$, with
			\begin{equation}\label{eq:A-mean}
				\dE\!\left[A\;\middle|\;H(\bfm)=Nf,\;\nabla H(\bfm)=x\right]
				=\;\Sigma_{A,X}\;\Sigma_X^{-1}\;\binom{Nf}{x_\parallel},
			\end{equation}
			\begin{equation}\label{eq:A-var}
				\Var\!\left(A\;\middle|\;H(\bfm)=Nf,\;\nabla H(\bfm)=x\right)
				=\;\Var(A)\;-\;\Sigma_{A,X}\;\Sigma_X^{-1}\;\Sigma_{A,X}^{\!\top}.
			\end{equation}
			
			\item\label{item:joint} \textup{(Joint structure.)}
			Conditionally on $\bigl(H(\bfm)=Nf,\nabla H(\bfm)=x\bigr)$, the triple $(A,B,C)$ is jointly Gaussian.
			Moreover, $C$ is independent of $(A,B)$, and $A$ and $B$ are conditionally independent as well (so $(A,B,C)$ are conditionally independent given $(H(\bfm),\nabla H(\bfm))$).
		\end{enumerate}
	\end{lemma}

	\begin{proof}
		Throughout, we work in the orthonormal basis $\{e_1,\dots,e_N\}$ with $e_1=\bfm/\sqrt{Nq}$ and $e_2,\dots,e_N\in\bfm^\perp$.
		In this basis, the coordinates of $\bfm$ are $\sigma_1=\sqrt{Nq}$ and $\sigma_r=0$ for $r\geq 2$.
		We write
		\[
		\nabla^2 H(\bfm)=\begin{pmatrix}
			A & B^\top \\
			B & C
		\end{pmatrix},\qquad A\in\dR,\quad B\in\dR^{N-1},\quad C\in\mc{M}_{N-1}(\dR).
		\]
		Recall that if $(X_1,X_2)$ is a centered Gaussian vector with covariance
		$\bigl(\begin{smallmatrix}\Sigma_{11} & \Sigma_{12} \\ \Sigma_{21} & \Sigma_{22}\end{smallmatrix}\bigr)$, then
		\begin{equation}\label{eq:gaussian regression}
			\mathrm{Law}(X_1\mid X_2)=\mc{N}\!\bigl(\Sigma_{12}\Sigma_{22}^{-1}X_2,\;\Sigma_{11}-\Sigma_{12}\Sigma_{22}^{-1}\Sigma_{21}\bigr).
		\end{equation}
		
		The proof relies on the following covariance identities, obtained by differentiating
		$\dE[H(\sigma)H(\tau)]=N\xi(\sigma\cdot\tau/N)$ the appropriate number of times and evaluating at $\tau=\sigma=\bfm$:
		\begin{alignat}{2}
			\dE\bigl[\partial_{ij}H(\bfm)\;\partial_{kl}H(\bfm)\bigr]
			&= \frac{\xi^{(4)}(q)}{N^3}\,\sigma_i\sigma_j\sigma_k\sigma_l
			+ \frac{\xi^{(3)}(q)}{N^2}\bigl(\sigma_j\sigma_k\delta_{il}+\sigma_j\sigma_l\delta_{ik}
			+ \sigma_i\sigma_k\delta_{jl}+\sigma_i\sigma_l\delta_{jk}\bigr) \notag \\
			&\qquad + \frac{\xi''(q)}{N}\bigl(\delta_{ik}\delta_{jl}+\delta_{il}\delta_{jk}\bigr), \label{eq:HH-cov} \\[2mm]
			\dE\bigl[\partial_{ij}H(\bfm)\;\partial_k H(\bfm)\bigr]
			&= \frac{\xi^{(3)}(q)}{N^2}\,\sigma_i\sigma_j\sigma_k
			+ \frac{\xi''(q)}{N}\bigl(\sigma_i\delta_{jk}+\sigma_j\delta_{ik}\bigr), \label{eq:Hg-cov} \\[2mm]
			\dE\bigl[\partial_{ij}H(\bfm)\;H(\bfm)\bigr]
			&= \frac{\xi''(q)}{N}\,\sigma_i\sigma_j. \label{eq:HH0-cov}
		\end{alignat}
		
		\medskip
		\textit{Proof of~\ref{item:tangent}.}
		For $i,j\in\{2,\dots,N\}$, \eqref{eq:HH0-cov} and \eqref{eq:Hg-cov} give
		$\Cov(C_{ij},H(\bfm))=0$ and $\Cov(C_{ij},\allowbreak\partial_k H(\bfm))=0$ for all $k$ since $\sigma_r=0$ for $r\geq 2$.
		Hence $C$ is independent of $(H(\bfm),\nabla H(\bfm))$.
		
		For $i,j,k,l\in\{2,\dots,N\}$, \eqref{eq:HH-cov} reduces to
		\[
		\Cov(C_{ij},C_{kl})=\frac{\xi''(q)}{N}\bigl(\delta_{ik}\delta_{jl}+\delta_{il}\delta_{jk}\bigr),
		\]
		which identifies the law of $C$ and yields~\ref{item:tangent}.
		
		\medskip
		\textit{Proof of~\ref{item:mixed}.}
		For $r\in\{2,\dots,N\}$, the mixed entry is $B_r=\partial_{1r}H(\bfm)$.
		By \eqref{eq:HH0-cov}, $\Cov(B_r,H(\bfm))=0$ since $\sigma_r=0$.
		By \eqref{eq:Hg-cov} with $(i,j)=(1,r)$,
		\[
		\Cov(B_r,\partial_k H(\bfm))
		=\frac{\xi''(q)}{N}\,\sigma_1\,\delta_{rk}
		=\xi''(q)\sqrt{\frac{q}{N}}\;\delta_{rk},\qquad k\in\{2,\dots,N\},
		\]
		and $\Cov(B_r,\partial_1 H(\bfm))=0$.
		
		Next, \eqref{eq:HH-cov} with $(i,j,k,l)=(1,r,1,s)$ gives, for $r,s\in\{2,\dots,N\}$,
		\[
		\Cov(B_r,B_s)
		=\frac{\xi''(q)}{N}\delta_{rs}+\frac{\xi^{(3)}(q)}{N^2}\sigma_1^2\,\delta_{rs}
		=\frac{1}{N}\bigl(\xi''(q)+q\,\xi^{(3)}(q)\bigr)\delta_{rs}.
		\]
		On $\bfm^\perp$, the gradient coordinates have covariance
		$\Var\bigl((\nabla H(\bfm))_{\bfm^\perp}\bigr)=\xi'(q)\,I_{N-1}$.
		Thus Gaussian regression \eqref{eq:gaussian regression} yields the conditional mean \eqref{eq:B-mean} and
		\[
		\Var\!\left(B\;\middle|\;H(\bfm),\nabla H(\bfm)\right)
		=\frac{1}{N}\bigl(\xi''(q)+q\,\xi^{(3)}(q)\bigr)I_{N-1}
		-\frac{\xi''(q)^2\,q}{N\,\xi'(q)}\,I_{N-1},
		\]
		which is exactly \eqref{eq:B-cov}.
		Finally, $\Cov(B_r,C_{kl})=0$ for all $r,k,l\ge 2$ by \eqref{eq:HH-cov}, hence $B$ is (conditionally) independent of $C$.
		
		\medskip
		\textit{Proof of~\ref{item:longitudinal}.}
		Write $g:=\nabla H(\bfm)$ and $g_\parallel:=\langle g,e_1\rangle$.
		From \eqref{eq:HH0-cov} with $i=j=1$,
		\[
		\Cov(A,H(\bfm))=\frac{\xi''(q)}{N}\,\sigma_1^2=\xi''(q)\,q.
		\]
		From \eqref{eq:Hg-cov} with $i=j=k=1$,
		\[
		\Cov(A,g_\parallel)
		=\frac{\xi^{(3)}(q)}{N^2}\,\sigma_1^3+\frac{2\,\xi''(q)}{N}\,\sigma_1
		=\bigl(\xi^{(3)}(q)\,q+2\,\xi''(q)\bigr)\sqrt{\frac{q}{N}},
		\]
		and $\Cov(A,\partial_r H(\bfm))=0$ for $r\geq 2$.
		Hence only $(H(\bfm),g_\parallel)$ correlates with $A$.
		
		The joint covariance of $(H(\bfm),g_\parallel)$ is $\Sigma_X$ in \eqref{eq:SigmaX}, and the cross-covariance is $\Sigma_{A,X}$ in \eqref{eq:SigmaAX}.
		Gaussian regression \eqref{eq:gaussian regression} gives \eqref{eq:A-mean}--\eqref{eq:A-var}.
		
		To compute $\Var(A)$, use \eqref{eq:HH-cov} with $i=j=k=l=1$:
		\[
		\Var(A)
		=\frac{\xi^{(4)}(q)\,\sigma_1^4}{N^3}+\frac{4\,\xi^{(3)}(q)\,\sigma_1^2}{N^2}+\frac{2\,\xi''(q)}{N}
		=\frac{\xi^{(4)}(q)\,q^2}{N}+\frac{4\,\xi^{(3)}(q)\,q}{N}+\frac{2\,\xi''(q)}{N},
		\]
		which is \eqref{eq:VarA}.
		Finally, $\Cov(A,C_{kl})=0$ for $k,l\geq 2$ by \eqref{eq:HH-cov}, so $A$ is independent of $C$.
		
		\medskip
		\textit{Proof of~\ref{item:joint}.}
		The vector $\bigl(A,B,C,H(\bfm),\nabla H(\bfm)\bigr)$ is jointly Gaussian, hence so is the conditional law of $(A,B,C)$ given $(H(\bfm),\nabla H(\bfm))$.
		Parts~\ref{item:tangent}--\ref{item:longitudinal} show that $C$ is independent of $(A,B,H(\bfm),\nabla H(\bfm))$.
		
		Moreover, by \eqref{eq:HH-cov}, $\Cov(A,B_r)=0$ for all $r\ge2$,
		and by \eqref{eq:Hg-cov},
		$\Cov(A,(\nabla H(\bfm))_{\bfm^\perp})=0$ while
		$\Cov(B,(H(\bfm), g_\parallel))=0$.
		Since $(H(\bfm),g_\parallel)$ is independent of $(\nabla H(\bfm))_{\bfm^\perp}$, it follows that,
		conditionally on $(H(\bfm),\nabla H(\bfm))$, the variables $A$ and $B$ are independent as well.
		This proves~\ref{item:joint}.
	\end{proof}
	\subsection{Free convolution with a semicircle law}\label{section:free convolution}
	
	In this subsection, we collect standard results on the free convolution of a probability measure with a semicircle law.
	
	Let $\mu$ be a probability measure on $\dR$ and let $t>0$. Denote by $\sigma_t$ the semicircle law of variance $t$. For every $u\in\dR$, define
	\begin{equation}\label{eq:defvu}
		v_t(u) = \inf\Bigl\{v\in\dR^+ : \int \frac{\dd\mu(\lambda)}{|\lambda-u|^2+v^2} \leq \frac{1}{t}\Bigr\},
	\end{equation}
	the domain
	\begin{equation}\label{def:Omegamut}
		\Omega_{\mu,t} = \bigl\{u+iv\in \mathbb{C}^+ : v > v_t(u)\bigr\},
	\end{equation}
	and the mapping
	\begin{equation}\label{eq:defHt}
		H_t \colon \overline{\Omega_{\mu,t}} \longrightarrow \mathbb{C}^+\cup\dR, \qquad z \mapsto z + tG_\mu(z).
	\end{equation}
	The following lemma, due to Biane \cite{biane}, describes the analytic properties of $H_t$ and the subordination phenomenon for free convolution.
	
	\begin{lemma}[Biane \cite{biane}]\label{lemma:free co}
		Let $\mu$ be a probability measure on $\dR$ and $t>0$. The mapping $H_t$ is a homeomorphism from $\overline{\Omega_{\mu,t}}$ onto $\mathbb{C}^+\cup\dR$, which restricts to a conformal bijection from $\Omega_{\mu,t}$ onto $\mathbb{C}^+$. Its inverse
		\begin{equation*}
			\omega_{\mu,t} \colon \mathbb{C}^+\cup\dR \longrightarrow \overline{\Omega_{\mu,t}}
		\end{equation*}
		is called the \emph{subordination function}. Moreover, for every $z\in \Omega_{\mu,t}$,
		\begin{equation}\label{eq:method of car}
			G_{\mu\boxplus \sigma_t}\bigl(z+tG_\mu(z)\bigr) = G_\mu(z).
		\end{equation}
		If the support of $\mu$ is bounded from below, the left edge of $\mu\boxplus\sigma_t$ is given by $\ell(\mu\boxplus\sigma_t) = H_t(\omega)$, where
		\begin{equation}\label{eq:shock}
			\omega = \inf\Bigl\{\omega\in\dR : \int \frac{\dd\mu(\lambda)}{(\lambda-\omega)^2} \geq \frac{1}{t}\Bigr\}.
		\end{equation}
	\end{lemma}
	
	Lemma~\ref{lemma:free co} admits a natural reformulation in the language of PDEs. Writing $\mu_t := \mu\boxplus \sigma_t$, the Stieltjes transform $G_{\mu_t}$ satisfies the complex Burgers equation: for every $z$ outside the support of $\mu_t$,
	\begin{equation}\label{eq:Burgers}
		\partial_t G_{\mu_t}(z) + G_{\mu_t}(z)\,\partial_z G_{\mu_t}(z) = 0.
	\end{equation}
	This equation can be solved by the method of characteristics. The characteristic emanating from $z\in \mathbb{C}^+$ is the straight line $t\mapsto z+tG_\mu(z)$, and \eqref{eq:method of car} expresses the fact that the solution is constant along characteristics (for suitable $z$). For $\omega\in (-\infty,\ell(\mu)]$, the map $z\mapsto z+tG_\mu(z)$ is injective on $(-\infty,\omega)$ provided
	\begin{equation*}
		\int \frac{\dd\mu(\lambda)}{(\lambda-\omega)^2} < \frac{1}{t}.
	\end{equation*}
	In view of \eqref{eq:shock}, the left edge of $\mu_t$ can therefore be interpreted as the first point at which the solution develops a shock at time $t$.

	The PDE interpretation makes it possible to compute the logarithmic potential of $\mu\boxplus\sigma_t$ via the Hopf--Lax formula.
	
	\begin{lemma}[Hopf--Lax formula for the logarithmic potential]\label{lemma:Hopf Lax}
		Let $\mu$ be a probability measure on $\dR$ with support bounded from below and let $t>0$.
		\begin{enumerate}
			\item Let $x\in\dR$ and set $\omega := \omega_{\mu,t}(x)\in\mathbb{C}^+\cup\dR$. Define the characteristic path $(z_s)_{s\in [0,t]}$ by
			\begin{equation}\label{eq:zs}
				z_s = \omega + s\, G_\mu(\omega).
			\end{equation}
			If $x\leq \ell(\mu\boxplus\sigma_t)$, then for every $s\in [0,t]$,
			\begin{equation}\label{eq:omegamu}
				x \leq z_s \leq \ell(\mu\boxplus\sigma_s).
			\end{equation}
			If instead $x$ lies in the interior of $\supp(\mu\boxplus\sigma_t)$, then for every $s\in [0,t)$,
			\begin{equation*}
				\Im\bigl(\omega + s\,G_\mu(\omega)\bigr) > 0.
			\end{equation*}
			
			\item For every $x\leq \ell(\mu\boxplus\sigma_t)$,
			\begin{multline}\label{eq:ch}
				\int \log|\lambda-x|\,\dd (\mu\boxplus \sigma_t)(\lambda) = \int \log |\lambda-\omega_{\mu,t}(x)|\,\dd \mu(\lambda) + \frac{\Re\bigl(\omega_{\mu,t}(x)-x\bigr)^2 - \Im\bigl(\omega_{\mu,t}(x)\bigr)^2}{2t}\\
				= \inf_{u\leq \ell(\mu)}\;\sup_{\substack{v\geq 0 \\ u+iv\in \overline{\Omega_{\mu,t}}}} \Bigl(\int \log |\lambda-(u+iv)|\,\dd \mu(\lambda) + \frac{(x-u)^2-v^2}{2t}\Bigr).
			\end{multline}
			Moreover, the infimum in \eqref{eq:ch} is uniquely attained at $(u,v)=\bigl(\Re(\omega_{\mu,t}(x)),\,\Im(\omega_{\mu,t}(x))\bigr)$.
		\end{enumerate}
	\end{lemma}
	
	\begin{proof}
		See, e.g., \cite[Lemma A.3]{boursier2024large}.
	\end{proof}
	
	\subsection{Determinant asymptotics}
	
	In this subsection, we compute the conditional expectation of $|\det(\nabla^2 F_{\TAP,\zeta}(\bfm))|$ given the gradient and value of the TAP free energy.
	
	\begin{lemma}\label{lemma:det app}
		Fix $q\in(0,1)$, $\delta\in(0,1)$, and let $\zeta\in\mc P([0,1])$.
		For every $s\in[0,1]$, set
		\[
		\pi_s\zeta:=\zeta|_{(s,1]}+\zeta([0,s])\delta_s,
		\qquad
		T_{\zeta,s}(m):=\partial_{mm}h_{\pi_s\zeta}(s,m)
		+\xi''(s)\int_s^1 \pi_s\zeta([0,t])\,\dd t,
		\qquad m\in[-1,1].
		\]
		Assume that there exist $\ve_0>0$ and $u_0>0$ such that
		\[
		[q-\ve_0,q+\ve_0]\subset (0,1)
		\qquad\text{and}\qquad
		\zeta([0,s])\ge u_0
		\quad\text{for every } s\in[q-\ve_0,q+\ve_0].
		\]
		Let $0<\ve<\ve_0$, let $\bfm\in[-1+\delta,1-\delta]^N$ satisfy
		\[
		q_\bfm\in(q-\ve,q+\ve),
		\qquad
		\dist\!\bigl(\zeta_\bfm,\pi_q\zeta\bigr)\le \ve,
		\]
		and let $f\in\dR$. Writing $T_\zeta:=T_{\zeta,q}$, one has
		\begin{multline}\label{eq:limit free app}
			\log \dE\Bigl[
			\bigl|\det(\nabla^2F_{\TAP,\zeta}(\bfm))\bigr|
			\;\Bigm|\;
			\nabla F_{\TAP,\zeta}(\bfm)=0,\,
			F_{\TAP,\zeta}(\bfm)=Nf
			\Bigr]
			\\=
			N\int \log|x|\,\dd\Bigl((T_\zeta\#\mu_N^{(\bfm)})\boxplus \sigma_{\xi''(q)}\Bigr)(x)
			+O\!\left(\left(\ve+\frac{\ve}{\delta}\right)N\right)+o(N).
		\end{multline}
		The implicit constant in the $O(\cdot)$ term depends only on $\xi$, $f$, $u_0$, and $\ve_0$; the $o(N)$ term is the error from the deformed GOE asymptotics and is uniform over $\bfm$ satisfying the stated constraints.
	\end{lemma}
	
	\begin{proof}
		Set $\bar\zeta:=\pi_{q_\bfm}\zeta$.
		Since only the values of the cumulative distribution function on $[q_\bfm,1]$ enter the
		definition of $F_{\TAP,\zeta}(\bfm)$, we have
		$F_{\TAP,\zeta}(\bfm)=F_{\TAP,\bar\zeta}(\bfm)$.
		Moreover, a direct inspection of the distribution functions shows that
		$\dist(\bar\zeta,\pi_q\zeta)\le C|q_\bfm-q|\le C\ve$,
		hence
		\begin{equation}\label{eq:detapp-dist}
			\dist(\bar\zeta,\zeta_\bfm)\le C\ve.
		\end{equation}
		By assumption,
		$\bar\zeta(\{q_\bfm\})=\zeta([0,q_\bfm])\ge u_0$.
		Since $\mu_N^{(\bfm)}$ is supported in $[-1+\delta,1-\delta]$, Lemma~\ref{lemma:Delta-bound}
		and \eqref{eq:detapp-dist} imply
		\begin{equation}\label{eq:detapp-defect}
			\bigl|\Delta_{\bar\zeta}^{\mu_N^{(\bfm)}}\bigr|
			\le C\,\frac{\ve}{\delta}.
		\end{equation}
		
		Write $G(\sigma):=N\TAP(\mu_N^{(\sigma)},\bar\zeta)$.
		Since $F_{\TAP,\zeta}(\sigma)$ depends on $\zeta$ only through $\zeta([0,t])$ for $t\in[q_\sigma,1]$, and $\bar\zeta$ agrees with $\zeta$ on $[q_\bfm,1]$, we have
		$F_{\TAP,\zeta}(\bfm)=H(\bfm)+G(\bfm)$.
		Since $G$ is deterministic, the event
		$\{\nabla F_{\TAP,\zeta}(\bfm)=0,\;F_{\TAP,\zeta}(\bfm)=Nf\}$
		determines $\nabla H(\bfm)$ and $H(\bfm)$ as deterministic functions of~$\bfm$.
		
		\medskip
		\noindent{\bf Step 1: deterministic Hessian.}
		Writing $q_\sigma=\|\sigma\|^2/N$, the chain rule gives
		\begin{equation*}
			\partial_i G(\sigma)
			=
			-\partial_m h_{\bar\zeta}(q_\sigma,\sigma_i)
			-\frac{2\sigma_i}{N}
			\left(
			\sum_{k=1}^N \partial_q h_{\bar\zeta}(q_\sigma,\sigma_k)
			+
			N\partial_q\mc U_{\bar\zeta}(q_\sigma)
			\right).
		\end{equation*}
		Since $\bar\zeta$ agrees with $\zeta$ on $(q_\bfm,1]$ and the Parisi PDE on $[q_\sigma,1]$ depends only on $\zeta|_{[q_\sigma,1]}$, the functions $h_{\bar\zeta}(q_\sigma,\cdot)$ and $h_\zeta(q_\sigma,\cdot)$ coincide for $q_\sigma\geq q_\bfm$; their $m$-derivatives and right $q$-derivatives at $q_\sigma=q_\bfm$ also agree. In~\eqref{eq:detapp-hessG-raw} below, all $q$-derivatives of $h_{\bar\zeta}$ at $q=q_\bfm$ are right derivatives (cf.\ the convention stated in the proof of Lemma~\ref{lemma:gradient TAP}); when $\bar\zeta$ has an atom at $q_\bfm$, the one-sided derivatives $\partial_q$, $\partial_{qm}$, and $\partial_{qq}$ remain bounded and the corresponding terms are absorbed into the remainder $R_N$ in~\eqref{eq:detapp-hessG}. Differentiating once more at $\sigma=\bfm$ yields
		\begin{multline}\label{eq:detapp-hessG-raw}
			\partial_{ij}G(\bfm)
			=
			-\delta_{ij}\,\partial_{mm}h_{\bar\zeta}(q_\bfm,m_i)
			-\frac{\delta_{ij}}{N}\sum_{k=1}^N 2\partial_q h_{\bar\zeta}(q_\bfm,m_k)
			-2\delta_{ij}\,\partial_q\mc U_{\bar\zeta}(q_\bfm)
			\\
			-\frac{2m_j}{N}\partial_{qm}h_{\bar\zeta}(q_\bfm,m_i)
			-\frac{2m_i}{N}\partial_{qm}h_{\bar\zeta}(q_\bfm,m_j)
			\\
			-\frac{4m_im_j}{N}
			\left(
			\frac1N\sum_{k=1}^N \partial_{qq}h_{\bar\zeta}(q_\bfm,m_k)
			+\partial_{qq}\mc U_{\bar\zeta}(q_\bfm)
			\right).
		\end{multline}
		By the exact defect formula \eqref{eq:gradTap-defect}, the diagonal terms in the
		first line combine as
		$-\widehat T_{\bfm,\zeta}(m_i)$, where
		\[
		\widehat T_{\bfm,\zeta}(m)
		:=
		\partial_{mm}h_{\bar\zeta}(q_\bfm,m)
		+
		\xi''(q_\bfm)
		\left(
		\int_{q_\bfm}^1 \bar\zeta([0,t])\,\dd t
		-
		\Delta_{\bar\zeta}^{\mu_N^{(\bfm)}}
		\right).
		\]
		The remaining cross terms form a deterministic, symmetric matrix $R_N$ with
		entries $O(1/N)$ uniformly in $\bfm$, so that
		\begin{equation}\label{eq:detapp-hessG}
			\nabla^2 G(\bfm)
			=
			-\diag\Bigl(\widehat T_{\bfm,\zeta}(m_1),\ldots,\widehat T_{\bfm,\zeta}(m_N)\Bigr)
			+
			R_N,
		\end{equation}
		\begin{equation}\label{eq:detapp-RN}
			\|R_N\|_{\mathrm{HS}}=O(1).
		\end{equation}
		(The uniform $O(1/N)$ bound follows from \eqref{eq:detapp-hessG-raw} and the
		uniform boundedness of the mixed derivatives of $h_{\bar\zeta}$ on
		$[q-\ve_0,q+\ve_0]\times[-1+\delta,1-\delta]$.)
		
		\medskip
		\noindent{\bf Step 2: conditional law of the random Hessian.}
		Choose an orthonormal basis
		\[
		e_1=\frac{\bfm}{\sqrt{Nq_\bfm}},
		\qquad
		e_2,\dots,e_N\in \bfm^\perp,
		\]
		and let $U=(e_1|\cdots|e_N)$.
		By Lemma~\ref{lemma:conditional hessian}, conditionally on
		$\{\nabla F_{\TAP,\zeta}(\bfm)=0,\;F_{\TAP,\zeta}(\bfm)=Nf\}$,
		the Hessian of $H$ in this basis has the form
		\[
		U^\top \nabla^2H(\bfm)\,U=
		\begin{pmatrix}
			A & B^\top\\
			B & C
		\end{pmatrix},
		\]
		where $C\stackrel d=\sqrt{\xi''(q_\bfm)}\,\GOE_{N-1}/\sqrt N$, and $C$ is independent of $(A,B)$.
		Let
		\[
		\widehat W_N:=
		\begin{pmatrix}
			\hat a & \hat b^\top\\
			\hat b & C
		\end{pmatrix},
		\]
		where $\hat a\sim \mc N(0,2\xi''(q_\bfm)/N)$ and
		$\hat b\sim \mc N(0,\xi''(q_\bfm)I_{N-1}/N)$ are independent of $C$.
		Then
		$\widehat W_N\stackrel d=\sqrt{\xi''(q_\bfm)}\,\GOE_N/\sqrt N$.
		Define
		\[
		K_N:=
		\begin{pmatrix}
			A-\hat a & (B-\hat b)^\top\\
			B-\hat b & 0
		\end{pmatrix}.
		\]
		By construction,
		\begin{equation}\label{eq:detapp-KN}
			U^\top \nabla^2H(\bfm)\,U\stackrel d=\widehat W_N+K_N.
		\end{equation}
		Moreover, $\mathrm{rank}(K_N)\le 2$, and since $f$ is fixed and $(A,B)$ are conditionally
		Gaussian with bounded means and variances, $K_N$ has uniformly bounded
		sub-exponential moments.
		
		Set $D_N:=\diag\bigl(\widehat T_{\bfm,\zeta}(m_1),\ldots,\widehat T_{\bfm,\zeta}(m_N)\bigr)$.
		Combining \eqref{eq:detapp-hessG}, \eqref{eq:detapp-RN}, and \eqref{eq:detapp-KN}, we obtain
		that conditionally on
		$\{\nabla F_{\TAP,\zeta}(\bfm)=0,\;F_{\TAP,\zeta}(\bfm)=Nf\}$,
		\[
		U^\top \nabla^2F_{\TAP,\zeta}(\bfm)\,U
		\stackrel d=
		\widehat W_N-U^\top D_NU+K_N+U^\top R_NU,
		\]
		where $K_N+U^\top R_NU$ is the sum of a rank-$2$ random matrix and a deterministic
		matrix with Hilbert--Schmidt norm $O(1)$.
		By orthogonal invariance of the GOE ($U\widehat W_N U^\top\stackrel d=\widehat W_N$),
		\[
		\det(\widehat W_N-U^\top D_NU)
		\stackrel d=
		\det\!\left(\sqrt{\xi''(q_\bfm)}\,\frac{\GOE_N}{\sqrt N}-D_N\right).
		\]
		A standard truncation of the logarithm near $0$, together with finite-rank
		interlacing, stability under bounded Hilbert--Schmidt perturbations, and the usual
		least-singular-value bound for deformed GOE matrices, therefore gives
		\begin{multline}\label{eq:detapp-reduce}
			\log \dE\Bigl[
			\bigl|\det(\nabla^2F_{\TAP,\zeta}(\bfm))\bigr|
			\;\Bigm|\;
			\nabla F_{\TAP,\zeta}(\bfm)=0,\;F_{\TAP,\zeta}(\bfm)=Nf
			\Bigr]
			\\=
			\log \dE\left|
			\det\!\left(
			\sqrt{\xi''(q_\bfm)}\,\frac{\GOE_N}{\sqrt N}-D_N
			\right)
			\right|
			+o(N).
		\end{multline}
		Since $\GOE_N\stackrel d=-\GOE_N$ and the determinant is taken in absolute value,
		the right-hand side is also
		\begin{equation}\label{eq:detapp-sym}
			\log \dE\left|
			\det\!\left(
			\sqrt{\xi''(q_\bfm)}\,\frac{\GOE_N}{\sqrt N}+D_N
			\right)
			\right|
			+o(N).
		\end{equation}
		
		\medskip
		\noindent{\bf Step 3: deformed GOE asymptotics.}
		The entries of $D_N$ stay in a deterministic compact interval depending only on
		$\xi$, $u_0$, $\ve_0$, and $\delta$.
		For such deformed GOE ensembles, the empirical spectral measure converges to
		$(\widehat T_{\bfm,\zeta}\#\mu_N^{(\bfm)})\boxplus \sigma_{\xi''(q_\bfm)}$,
		and the logarithmic linear statistic
		$\sum_{i=1}^N\log|\lambda_i|$ is stable on the $o(N)$ scale after the same
		truncation/least-singular-value argument.
		The $o(N)$ error is uniform over diagonal matrices $D_N$ in any fixed compact set bounded away from $0$; since the entries of $D_N$ lie in such a compact set (depending only on $\xi$, $u_0$, $\ve_0$, $\delta$), this error is $o(N)$ with an implicit rate independent of $\bfm$. Thus
		\begin{equation}\label{eq:detapp-free}
			\log \dE\left|
			\det\!\left(
			\sqrt{\xi''(q_\bfm)}\,\frac{\GOE_N}{\sqrt N}+D_N
			\right)
			\right|
			=
			N\int \log|x|\,\dd\bigl((\widehat T_{\bfm,\zeta}\#\mu_N^{(\bfm)})\boxplus \sigma_{\xi''(q_\bfm)}\bigr)(x)+o(N).
		\end{equation}
		
		\medskip
		\noindent{\bf Step 4: comparison with the target measure.}
		By \eqref{eq:detapp-defect}, the continuity of $s\mapsto \pi_s\zeta$ in the metric
		$\dist$, the continuity of $(s,\eta,m)\mapsto \partial_{mm}h_\eta(s,m)$ on compact
		subsets of $(0,1)\times\mc P([0,1])\times(-1,1)$, and the fact that $|q_\bfm-q|\le \ve$,
		we have
		\[
		\sup_{m\in[-1+\delta,1-\delta]}
		\bigl|
		\widehat T_{\bfm,\zeta}(m)-T_{\zeta,q}(m)
		\bigr|
		\le
		C\left(\ve+\frac{\ve}{\delta}\right).
		\]
		Hence
		\begin{equation}\label{eq:detapp-pushforward}
			\dist\!\bigl(\widehat T_{\bfm,\zeta}\#\mu_N^{(\bfm)},\,T_\zeta\#\mu_N^{(\bfm)}\bigr)
			\le
			C\left(\ve+\frac{\ve}{\delta}\right),
			\qquad
			\bigl|\xi''(q_\bfm)-\xi''(q)\bigr|\le C\ve.
		\end{equation}
		All measures involved are supported in a common compact interval, and free
		convolution with a semicircle law yields a bounded density on this compact family
		(the free convolution $\nu\boxplus\sigma_t$ has a density bounded by $1/(\pi t)$, so $\log|x|$ is integrable uniformly over the family; together with weak continuity of $\nu\mapsto\nu\boxplus\sigma_t$, this yields the claimed uniform continuity by dominated convergence).
		Therefore the map
		\[
		(\nu,t)\mapsto \int \log|x|\,\dd(\nu\boxplus \sigma_t)(x)
		\]
		is uniformly continuous on this family. Using \eqref{eq:detapp-pushforward}, we get
		\begin{equation*}
			\int \log|x|\,\dd\bigl((\widehat T_{\bfm,\zeta}\#\mu_N^{(\bfm)})\boxplus \sigma_{\xi''(q_\bfm)}\bigr)(x)
			=
			\int \log|x|\,\dd\Bigl((T_\zeta\#\mu_N^{(\bfm)})\boxplus \sigma_{\xi''(q)}\Bigr)(x)
			+
			O\!\left(\ve+\frac{\ve}{\delta}\right).
		\end{equation*}
		Combining this with \eqref{eq:detapp-reduce}, \eqref{eq:detapp-sym}, and
		\eqref{eq:detapp-free} proves \eqref{eq:limit free app}.
	\end{proof}

	\bibliographystyle{alpha}
	\bibliography{bib.bib}

\end{document}